\documentclass[preprint]{elsarticle}

\usepackage{lineno,hyperref}
\modulolinenumbers[5]

\journal{Linear Algebra and its Applications}

\usepackage{amsmath,amsfonts,amssymb,amsthm}
\usepackage{mathtools}
\usepackage{mathrsfs}
\usepackage{enumerate}
\usepackage[inline]{enumitem}
\usepackage{multirow}
\makeatletter
\newcommand{\inlineitem}[1][]{%
	\ifnum\enit@type=\tw@
	{\descriptionlabel{#1}}
	\hspace{\labelsep}%
	\else
	\ifnum\enit@type=\z@
	\refstepcounter{\@listctr}\fi
	\quad\@itemlabel\hspace{\labelsep}%
	\fi}
\makeatother

\theoremstyle{plain}
\newtheorem{theorem}{Theorem}[section]
\newtheorem{corollary}[theorem]{Corollary}
\newtheorem{proposition}[theorem]{Proposition}
\newtheorem{lemma}[theorem]{Lemma}
\theoremstyle{definition}
\newtheorem{definition}[theorem]{Definition}
\newtheorem{remark}[theorem]{Remark}
\newtheorem{example}[theorem]{Example}

\newcommand{\cU}{\mathcal U}

\newcommand{\cL}{\mathcal L}
\newcommand{\cR}{\mathcal R}

\newcommand{\bx}{\mathbf x}
\newcommand{\by}{\mathbf y}

\newcommand{\bu}{\mathbf u}

\newcommand{\ba}{\mathbf a}
\newcommand{\bb}{\mathbf b}

\newcommand{\bv}{\mathbf v}
\newcommand{\wtA}{\widetilde A}
\newcommand{\wtE}{\widetilde E}

\begin{document}

\begin{frontmatter}
%\title{Elsevier \LaTeX\ template\tnoteref{mytitlenote}}

\title{Gram mates, sign changes in singular values, and isomorphism\tnoteref{mytitlenote}}
\tnotetext[mytitlenote]{Funding: This work is supported by a University of Manitoba Graduate Fellowship and by a Discovery Grant from the Natural Sciences and Engineering Research Council of Canada under grant number RGPIN--2019--05408.}


\author{Sooyeong Kim\corref{mycorrespondingauthor}}
\cortext[mycorrespondingauthor]{Corresponding author}
\ead{kims3428@myumanitoba.ca}

\author{Steve Kirkland}
\ead{Stephen.Kirkland@umanitoba.ca}

\address{Department of Mathematics, University of Manitoba, Winnipeg MB, R3T 2N2, Canada}

%In order to understand the retention of the information in a two-mode network for the conversion approach in social networks,

\begin{abstract}
We study distinct $(0,1)$ matrices $A$ and $B$, called \textit{Gram mates}, such that $AA^T=BB^T$ and $A^TA=B^TB$. We characterize Gram mates where one can be obtained from the other by changing signs of some positive singular values. We classify Gram mates such that the rank of their difference is at most $2$. Among such Gram mates, we further produce equivalent conditions in order that one is obtained from the other by changing signs of at most $2$ positive singular values. Moreover, we provide some tools for constructing Gram mates where the rank of their difference is more than $2$. Finally, we characterize non-isomorphic Gram mates whose difference is of rank $1$ with some extra conditions. 
\end{abstract}

\begin{keyword}
Zero-one matrices\sep Gram mates\sep singular values\sep singular vectors\sep isomorphic
\MSC[2010] 91D30\sep 15A18
\end{keyword}

\end{frontmatter}

%\linenumbers

%%%%%%%%%%%%%%%%%%%%%%%%%%%%%%%%%%%%%%%%%%%%%%%%%%%%%%%%%%%%%%%%%%%%%%%%%%%%%%%%%%%%%%%%%%%%%%%%
\section{Introduction}
%%%%%%%%%%%%%%%%%%%%%%%%%%%%%%%%%%%%%%%%%%%%%%%%%%%%%%%%%%%%%%%%%%%%%%%%%%%%%%%%%%%%%%%%%%%%%%%%

The quantitative analysis of social networks is a thriving discipline. A two-mode network in social networks describes the relationships between actors and events. A $(0,1)$ matrix $A$ can represent a two-mode network in a way that rows and columns of $A$ represent actors and events, respectively, with a $1$ in the corresponding position of the matrix if an actor joins an event, and a $0$ otherwise (see \cite{Hanneman:twomode} for introductory material on two-mode networks). One of the basic approaches to the study of two-mode networks is the so-called conversion approach \cite{Borgatti:Conversion}, which analyses the features of A by considering the related (single-mode) matrices $AA^T$ and $A^TA$. However, how can we be assured that the conversion approach retains the information of $A$? In other words, can there be a different $(0,1)$ matrix $B$ such that $AA^T=BB^T$ and $A^TA=B^TB$; further, is it possible that such a $B$ is not obtained from $A$ by permuting rows and columns?

This paper is a study of pairs of $(0,1)$ matrices $A$ and $B$, called \textit{Gram mates}, such that $AA^T=BB^T$, $A^TA=B^TB$, and $A\neq B$. That study arises from a question in \cite{everett:dual} as to whether the conversion approach loses structural features of a two-mode network, and is established mathematically in \cite{Steve:two-mode}. In \cite{everett:dual}, it is shown how to recover a $(0,1)$ matrix $A$ from $AA^T$ and $A^TA$ under certain circumstances, and how non-isomorphic Gram mates $A$ and $B$ can cause `data loss of the information for $A$'. To be clear about our work in this paper, we briefly introduce the way of recovering $A$ from $AA^T$ and $A^TA$, and the speculation in \cite{everett:dual}. For the recovery of $A$ from $AA^T$ and $A^TA$, the singular value decomposition is used as follows: under the assumption that $A$ has distinct positive singular values, for fixed singular vectors of $A$ (which are uniquely determined up to sign), we change some signs of singular values until a $(0,1)$ matrix is obtained from a singular value decomposition, or until no $(0,1)$ matrix is produced. It is speculated in \cite{everett:dual} that there is a very high probability that matrices from this reconstruction are isomorphic.

Kirkland \cite{Steve:two-mode} presents techniques for a systematic study of a pair of Gram mates. One of those techniques is to consider a $(0,1,-1)$ matrix $E$ and to investigate Gram mates $A$ and $A+E$. We use that technique in order to understand the relation between $(0,1)$ matrices $A$ and $B$, where $B$ is obtained from $A$ by changing signs of some positive singular values. Furthermore, regarding the speculation, we provide an infinite family of pairs of non-isomorphic Gram mates with that property.  

In addition to the works motivated by \cite{everett:dual}, we study families of pairs of Gram mates. One of the results of Kirkland \cite{Steve:two-mode} is that given two $(0,1)$ matrices at random of the same large size, the probability that they are Gram mates is `very' small. For that reason, we furnish infinite families of pairs of Gram mates according to the rank of their difference. Moreover, from those families, we give tools to construct other families.

In the present paper, we discuss the following in each section. We characterize matrices from the reconstruction regardless of whether they have all distinct positive singular values in Section \ref{section:Svd} (Theorem \ref{Thm:equivalent conditions} and Corollary \ref{Cor:equivalent conditions for convertible}). Section \ref{section:rank1 and 2} establishes all pairs of Gram mates $A$ and $B$ where the rank of $A-B$ is at most $2$ (Theorems \ref{Thm:Gram mates of rank1}, \ref{Thm:Gram mates rank 21} and \ref{Thm:Grammates Rank22}). Moreover, we provide equivalent conditions for $A$ being obtained from $B$ by changing signs of at most two positive singular values of $A$ (Theorems \ref{Thm:Gram mates of rank1}, \ref{Thm:Gramsingular rank21} and \ref{Thm:Gramsingular rank 22}). In Section \ref{section:construction}, we provide several tools for attaining pairs of Gram mates $A$ and $B$ where the rank of $A-B$ is more than $2$. Section \ref{section:noniso} exhibits families of pairs of non-isomorphic Gram mates $A$ and $B$, with some extra conditions, where the rank of $A-B$ is $1$ (Proposition \ref{Prop:fixable iff row and column sum vectors} and Theorem \ref{Thm:non-isomorphic distinct sv}).

%%%%%%%%%%%%%%%%%%%%%%%%%%%%%%%%%%%%%%%%%%%%%%%%%%%%%%%%
\section{Preliminaries}\label{Subsec:preli Gram}
%%%%%%%%%%%%%%%%%%%%%%%%%%%%%%%%%%%%%%%%%%%%%%%%%%%%%%%%

\begin{definition}\label{Def:GramMates}
	Let $A$ and $B$ be $(0,1)$ matrices. The matrices $A$ and $B$ are \textit{Gram mates} and $A$ is called a \textit{Gram mate to} $B$ if $AA^T=BB^T$, $A^TA=B^TB$ and $A\neq B$.
\end{definition}

\begin{definition}
	Let $A$ and $B$ be $(0,1)$ matrices. The matrices $A$ and $B$ are \textit{isomorphic} if there exist permutation matrices $P$ and $Q$ such that $B=PAQ$.
\end{definition}

It is readily verified that $A$ and $B$ are Gram mates if and only if for any permutation matrices $P$ and $Q$ of the appropriate sizes, $PAQ$ and $PBQ$ are Gram mates; $A$ and $B$ are Gram mates if and only if $A^T$ and $B^T$ are Gram mates. So, those statements are used when we need to simplify some hypotheses of a claim in terms of Gram mates to focus on particular cases without the loss of generality of the claim.

We have a basic observation \cite{Steve:two-mode} that Gram mates $A$ and $B$ must have the same row sum vectors and the same column sum vectors. Thus, both $A-B$ and $B-A$ are $(0,1,-1)$ matrices such that their row sum and column sum vectors are $0$.

\begin{definition}
	Let $E$ be a $(0,1,-1)$ matrix such that $E\mathbf{1}=0$ and $\mathbf{1}^TE=0^T$. The matrix $E$ is said to be \textit{realizable (with respect to Gram mates)} if there is a pair of Gram mates $A$ and $A+E$. We say that $A$ and $B$ are \textit{Gram mates via $E$} if $A$ and $B$ are Gram mates, and either $A-B=E$ or $B-A=E$.
\end{definition}

Evidently, any zero matrix is not realizable. It is easily seen that $E$ is realizable if and only if for any permutation matrices $P$ and $Q$ of the appropriate sizes, $PEQ$ is realizable; $E$ is realizable if and only if $E^T$ is realizable.

We shall introduce some notation in matrix theory. Let $A$ be an $m\times n$ matrix. Let $\alpha\subset\{1,\dots,m\}$ and $\beta\subset\{1,\dots,n\}$. We denote by $A[\alpha,\beta]$ the submatrix of $A$ whose rows and columns are indexed by $\alpha$ and $\beta$, respectively. Let $\alpha^c$ denote the complement of $\alpha$. We denote by $\mathbf{1}_n$ the all ones column vector of size $n$, by $I_n$ the identity matrix of size $n\times n$, and by $J_{n,m}$ the all ones matrix of size $n\times m$. If $k=n=m$, then we denote $J_{n,m}$ by $J_k$. The subscripts of $\mathbf{1}_n$, $I_n$, $J_{n,m}$ and $J_k$ are omitted if their sizes are clear from the context. We also use $\mathbf{0}_n$ to denote the all zeros column vector of size $n$. We write $\mathbf{0}_n$ as $0$ when the meaning is clear from the context. A block diagonal matrix the main diagonal blocks of which are square matrices $A_1,\dots,A_k$ is indicated as $\mathrm{diag}(A_1,\dots,A_k)$. In particular, if all the main diagonal blocks are scalars, it is a diagonal matrix. For a subset $X$ in $\mathbb{R}^n$, we denote the subspace spanned by $X$ as $\mathrm{span}(X)$. We use $\mathrm{Row}(A)$ and $\mathrm{Col}(A)$ to denote the row space of $A$ and the column space of $A$, respectively. The rank of $A$ is denoted as $\mathrm{rank}(A)$. 

%A \textit{permutation matrix} is a $(0,1)$ matrix in which $1$ appears exactly once in each row and column. A complex matrix $U$ is said to be \textit{unitary} if $UU^T=U^TU=I$. If $U$ is real and unitary, then $U$ is said to be \textit{orthogonal}.

The following are used in Section \ref{section:Svd}.

\begin{proposition}\cite{Horn:MatrixAnalysis}\label{Prop:commute with diagonal matrix}
	Let $\ell\geq 1$, and let $D=\mathrm{diag}(d_1I_{k_1},\dots, d_\ell I_{k_\ell})$ where $d_1,\dots,d_\ell$ are distinct and $k_i\geq 1$ for $i=1,\dots,\ell$. Suppose that $A$ commutes with $D$, and $A$ is orthogonal. Then, $A$ is a block diagonal matrix compatible with the partition of $D$ such that each of the main diagonal blocks of $A$ is orthogonal.
\end{proposition}

\begin{proposition}\cite{Strang:LinearAlgebra}\label{rowcolumnspace}
	Let $A$ and $B$ be matrices of compatible sizes for $AB$ to be defined. Then, 
	\begin{align*}
	&\mathrm{Col}(AA^T)=\mathrm{Row}(AA^T)=\mathrm{Col}(A),\; \mathrm{Col}(A^TA)=\mathrm{Row}(A^TA)=\mathrm{Row}(A),\\&\mathrm{Col}(AB)\subseteq\mathrm{Col}(A).
	\end{align*}
	Furthermore, if $AB$ is of rank $k$, then there are $k$ columns $\ba_1,\dots,\ba_k$ of $A$ that comprise a basis of $\mathrm{Col}(AB)$.
\end{proposition}

We briefly introduce the singular value decomposition (the SVD) \cite{Horn:MatrixAnalysis} which is a factorization of a real (or complex) matrix. We also state a basic observation regarding change of signs of some positive singular values. Our interest lies in zero-one matrices, so we assume our matrices to be real. Given an $m\times n$ matrix $A$, there exist $m\times m$ and $n\times n$ orthogonal matrices $U$ and $V$, respectively, such that $A=U\Sigma V^T$ for some $m\times n$ diagonal matrix $\Sigma$ whose entries on the main diagonal are non-negative in non-increasing order. The diagonal entries of $\Sigma$ are called singular values of $A$. The columns of $U$ and the columns of $V$ are called the left and right singular vectors of $A$, respectively. Note that the number of positive singular values of $A$ equals $\mathrm{rank}(A)$.

Suppose that $B$ is obtained from $A$ by changing signs of some positive singular values. Then,
$B$ can be written as $B=US\Sigma V^T$, where $S$ is a diagonal matrix the main diagonal entries of which consist of $r$ $-1$'s and $m-r$ ones, for some number $r$. Let $\widetilde{U}=US$. Clearly, $\widetilde{U}$ is an orthogonal matrix, so $\widetilde{U}\Sigma V^T$ is a singular value decomposition of $B$. Since $S\Sigma \Sigma^T S^T=\Sigma \Sigma^T$ and $\Sigma^T S^TS \Sigma=\Sigma^T \Sigma$, we have $AA^T=BB^T$ and $A^TA=B^TB$. Furthermore, since $A-B=U(\Sigma-S\Sigma)V^T$ and $\Sigma-S\Sigma\neq 0$, we have $A-B\neq 0$.

\begin{remark}\label{Remark:interp of change of signs}
	Continuing with $A$ and $B=US\Sigma V^T$ above, $S$ gives different interpretations for the relationship between $A$ and $B$. For example, if $S$ is given by $S=\mathrm{diag}(-1,-1,1,\dots,1)$, then we can say that $A$ is obtained from $B$ by changing the signs of the first two singular values; or of the first two either left or right singular vectors; or of the first (resp. second) left and the second (resp. first) right singular vectors. In order to avoid confusion from the choices of singular vectors for sign change, we adopt the interpretation `changing signs of singular values' despite the fact that $A$ and $B$ have the same singular values.
\end{remark}

\begin{lemma}\label{Lemma2:change sign implies AAT=BBT}
	Let $A$ be an $m\times n$ $(0,1)$ matrix. If $B$ is obtained from $A$ by changing signs of some positive singular values of $A$, then $AA^T=BB^T$, $A^TA=B^TB$ and $A\neq B$.
\end{lemma}

\begin{remark}
	For the result of Lemma \ref{Lemma2:change sign implies AAT=BBT}, if $B$ is not a zero-one matrix, then $A$ and $B$ are not Gram mates.
\end{remark}

%%%%%%%%%%%%%%%%%%%%%%%%%%%%%%%%%%%%%%%%%%%%%%%%%%%%%%%%%%%%%%%%%%%%%%%%%%%%%%%%%%%%%%%%%%%%%%%%
\section{Gram mates and the SVD}\label{section:Svd}
%%%%%%%%%%%%%%%%%%%%%%%%%%%%%%%%%%%%%%%%%%%%%%%%%%%%%%%%%%%%%%%%%%%%%%%%%%%%%%%%%%%%%%%%%%%%%%%%

\begin{proposition}\label{Prop:Grammates:SVD decomp}
	Let $A$ and $B$ be $m\times n$ real matrices, and let $\sigma_1,\dots,\sigma_\ell$ be distinct singular values (not necessarily in non-increasing order) of $A$ where $\ell\geq 1$. Then, $AA^T=BB^T$ and $A^TA=B^TB$ if and only if there exist orthogonal matrices $U_1$, $U_2$ and some orthogonal $V$ such that $A=U_1\Sigma V^T$, $B=U_2\Sigma V^T$, $U_1=U_2\mathrm{diag}(W_1,\dots,W_\ell)$, where for $i=1,\dots,\ell$, $W_i$ is an $m_i\times m_i$ orthogonal matrix, and $m_i$ is the multiplicity of $\sigma_i$ as a singular value. (Here the multiplicity of $0$ as a singular value coincides with that of $0$ as an eigenvalue of $AA^T$.)
\end{proposition}
\begin{proof}
	By the singular value decomposition, we find from $A^TA=B^TB$ that there exists an $n\times n$ orthogonal matrix $V$ such that $A=U_1\Sigma V^T$ and $B=U_2\Sigma V^T$ for some $m\times n$ rectangular diagonal matrix $\Sigma$, and $m\times m$ orthogonal matrices $U_1$ and $U_2$. Since $AA^T=BB^T$, we have $U_2^TU_1\Sigma\Sigma^T=\Sigma\Sigma^TU_2^TU_1$. By Proposition \ref{Prop:commute with diagonal matrix}, our desired conclusion is obtained.
	
	It is straightforward to prove the converse.
\end{proof}

\begin{remark}\label{Remark:fixing U}
	In the proof of Proposition \ref{Prop:Grammates:SVD decomp}, considering $AA^T=BB^T$ first instead of $A^TA=B^TB$, we can fix the same left singular vectors for $A$ and $B$. So, it can be deduced that there exist orthogonal matrices $V_1$, $V_2$ and some orthogonal $U$ such that $A=U\Sigma V_1^T$, $B=U\Sigma V_2^T$, $V_1=V_2\mathrm{diag}(W_1,\dots,W_\ell)$, where $\ell$ is the number of distinct singular values, $W_i$ is an $m_i\times m_i$ orthogonal matrix, and $m_i$ is the multiplicity of $\sigma_i$ as a singular value for $i=1,\dots,\ell$. (Here the multiplicity of $0$ as a singular value coincides with that of $0$ as an eigenvalue of $A^TA$.)
\end{remark}

\begin{proposition}\label{Prop:rowcolumnspaceofE}
	Let $E$ be a realizable matrix, and $(A,A+E)$ be a pair of Gram mates. For $\bx\in\mathrm{Row}(E)$, we have $A\bx\in \mathrm{Col}(E)$, and for $\by\in\mathrm{Col}(E)$, $A^T\by\in \mathrm{Row}(E)$.
\end{proposition}
\begin{proof}
	Since $AA^T=(A+E)(A+E)^T$, we have $AE^T=-E(A^T+E^T)$. It follows that for any $i^\text{th}$ row vector $\bx_i^T$ of $E$, $A\bx_i\in\mathrm{Col}(E)$. Similarly, $A^TE=-E^T(A+E)$ implies that for any $i^\text{th}$ column vector $\by_i$ of $E$, $A^T\by_i\in\mathrm{Row}(E)$. Therefore, we obtain our desired results.
\end{proof}

\begin{lemma}\label{Lemma:rowspace(E) and span of right singular vectors}
	Let $E$ be a realizable matrix of rank $k$, and $(A,A+E)$ be a pair of Gram mates. Then, there exist $k$ positive singular values of $A$ such that the set of their corresponding right (resp. left) singular vectors is a basis of $\mathrm{Row}(E)$ (resp. $\mathrm{Col}(E)$).
\end{lemma}
\begin{proof}
	By Proposition \ref{Prop:Grammates:SVD decomp}, there exist orthogonal matrices $U_1$, $U_2$ and $V$ such that $A=U_1\Sigma V^T$ and $A+E=U_2\Sigma V^T$ for some rectangular diagonal matrix $\Sigma$. Then, $EV=(U_2-U_1)\Sigma$. Since $A^TA=(A+E)^T(A+E)$, we have $E^TE=-A^TE-E^TA$. Substituting $V\Sigma^T U_1^T$ and $V\Sigma^T (U_2-U_1)^T$ for $A^T$ and $E^T$, respectively, in $-A^TE-E^TA$, we have 
	$$E^TE=V(-\Sigma^TU_1^TE-\Sigma^T(U_2-U_1)^TA).$$
	By Lemma \ref{rowcolumnspace}, $\mathrm{Col}(E^TE)=\mathrm{Row}(E)$. Since $\mathrm{rank}(E)=k$, $\mathrm{rank}(E^TE)=k$. Again by Lemma \ref{rowcolumnspace}, the column space of $V(-\Sigma^TU_1^TE-\Sigma^T(U_2-U_1)^TA)$ is spanned by $k$ columns $\bv_1,\dots, \bv_k$ of $V$; thus, the vectors $\bv_1,\dots, \bv_k$ comprise a basis of $\mathrm{Row}(E)$. Moreover, for any right singular vector $\bv\notin\mathrm{Row}(E)$, $\bv$ is orthogonal to $\mathrm{Row}(E)$. Thus, $E\bv=0$. The rank of $EV$ is $k$, so $E\bv_i\neq 0$ for $i=1,\dots,k$. Considering $EV=(U_2-U_1)\Sigma$, for $i=1,\dots, k$ the singular value corresponding to $\bv_i$ must be positive.
	
	By Proposition \ref{Prop:rowcolumnspaceofE}, $A\bv_i\in\mathrm{Col}(E)$ for $i=1,\dots,k$. Note that $AV=U_1\Sigma$. Since $A\bv_1,\dots, A\bv_k$ are linearly independent and $\mathrm{rank}(E)=k$, the set of $A\bv_1,\dots,A\bv_k$ is a basis of $\mathrm{Col}(E)$.  Our desired conclusion follows.
\end{proof}

Given a $(0,1)$ matrix $A$, let $B$ be a $(0,1)$ matrix obtained from $A$ by changing signs of some positive singular values. Since $A$ and $B$ are $(0,1)$ matrices, by Lemma \ref{Lemma2:change sign implies AAT=BBT} $A$ and $B$ are Gram mates. Suppose that $\bv$ is a right singular vector corresponding to one of those singular values. Then, $A\bv=-B\bv$ and so, $(A+B)\bv=0$. Considering Lemma \ref{Lemma:rowspace(E) and span of right singular vectors}, either $\bv\in\mathrm{Row}(A-B)$ or $\bv\notin\mathrm{Row}(A-B)$. We shall investigate the relation between $\bv$ and $\mathrm{Row}(A-B)$.

\begin{remark}\label{Remark:meaning of sign changes}
	Suppose that $B$ is obtained from $A$ by changing the signs of positive singular values $\sigma_1,\dots,\sigma_k$ of $A$ for some $k\geq 1$. We can find from Remark \ref{Remark:interp of change of signs} that in the context of obtaining $B$ from $A$, converting the signs of $\sigma_1,\dots,\sigma_k$ is equivalent to changing the sign of one of the left and right singular vectors of $A$ corresponding to $\sigma_i$ for $i=1,\dots,k$.
\end{remark}

\begin{proposition}\label{Prop:change of sign}
	Let $E$ be an $m\times n$ realizable matrix of rank $k$, and let $A$ be a $(0,1)$ matrix such that $A+E$ is a $(0,1)$ matrix. Then, the following are equivalent:
	\begin{enumerate}[label=(\alph*)]
		\item\label{enum:a:change of sign} $(A,A+E)$ is a pair of Gram mates and $(2A+E)E^T=0$,
		\item\label{enum:b:change of sign} $A+E$ is obtained from $A$ by changing the signs of some positive singular values.
	\end{enumerate}
	Furthermore, if one of \ref{enum:a:change of sign} and \ref{enum:b:change of sign} holds, then the following are satisfied: 
	\begin{enumerate}[label=(\roman*)]
		\item the number of positive singular values whose signs are changed is $\mathrm{rank}(E)$, 
		\item\label{cond1:change of sign} the positive singular values of $A$ whose signs are changed are the same as the $k$ positive singular values of $-\frac{1}{2}E$, and
		\item\label{cond2:change of sign} the corresponding left (resp. right) singular vectors of $A$ can be obtained from the corresponding left (resp. right) singular vectors of $-\frac{1}{2}E$. This implies that the corresponding left (resp. right) singular vectors of $A$ comprise a basis of $\mathrm{Col}(E)$ (resp. $\mathrm{Row}(E)$).
	\end{enumerate}
\end{proposition}
\begin{proof}
	Suppose that $(A,A+E)$ is a pair of Gram mates and $(2A+E)E^T=0$. By Proposition \ref{Prop:Grammates:SVD decomp}, there exist orthogonal matrices $U_1$, $U_2$ and $V$ such that $A=U_1\Sigma V^T$ and $A+E=U_2\Sigma V^T$ for some rectangular diagonal matrix $\Sigma$. Let $k=\mathrm{rank}(E)$. By Lemma \ref{Lemma:rowspace(E) and span of right singular vectors}, there exist right singular vectors $\bv_1,\dots,\bv_k$ of $A$ corresponding to positive singular values $\sigma_1,\dots,\sigma_k$ that form a basis of $\mathrm{Row}(E)$. Since $(2A+E)E^T=0$, we have $(2A+E)\bv_i=0$ for $i=1,\dots,k$. Then, we have $E\bv_i=-2A\bv_i=-2\sigma_i\bu_i$ where $\bu_i$ is a left singular vector of $A$ corresponding to $\sigma_i$. Furthermore, for any right singular vector $\bv\notin\mathrm{Row}(E)$, $E\bv=0$. Then, without loss of generality, we have
	$$EV=-2\begin{bmatrix}
	\bu_1 & \cdots & \bu_k & 0
	\end{bmatrix}\Sigma.$$
	Since $E=(U_2-U_1)\Sigma V^T$, we have $(U_2-U_1)\Sigma=-2\begin{bmatrix}
	\bu_1 & \cdots & \bu_k & 0
	\end{bmatrix}\Sigma$. If $AA^T$ is singular, then we may choose the same left singular vectors corresponding to the singular value $0$ for $A$ and $A+E$. Hence, $U_2-U_1=-2\begin{bmatrix}
	\bu_1 & \cdots & \bu_k & 0
	\end{bmatrix}$. It follows from Remark \ref{Remark:meaning of sign changes} that $A+E$ is obtained from $A$ by changing the signs of $\sigma_1,\dots,\sigma_k$. Furthermore, applying the Gram–Schmidt process to a basis of the orthogonal complement of $\mathrm{Row}(E)$, we obtain an orthonormal basis, say $\{\tilde{\bu}_{k+1},\dots,\tilde{\bu}_m\}$. Then,
	$$-E=\begin{bmatrix}
	\bu_1 & \cdots & \bu_k & \tilde{\bu}_{k+1} & \cdots & \tilde{\bu}_m
	\end{bmatrix}(2\widetilde{\Sigma})V^T$$
	where $\widetilde{\Sigma}=\mathrm{diag}(\sigma_1,\dots,\sigma_k,0,\dots,0)$. Rearranging the diagonal entries of $\widetilde{\Sigma}$ in non-decreasing order, one can obtain a singular value decomposition of $-E$. Therefore, $2\sigma_1,\dots,2\sigma_k$ are the positive singular values of $-E$.
	
	Suppose that $A+E$ is obtained from $A$ by changing the signs of $\ell$ positive singular values for some $\ell>0$. By Lemma \ref{Lemma2:change sign implies AAT=BBT}, $A$ and $A+E$ are Gram mates. By Remark \ref{Remark:meaning of sign changes}, there exist orthogonal matrices $U$, $\widetilde{U}$ and $V$ such that $A=U\Sigma V^T$ and $A+E=\widetilde{U}\Sigma V^T$ for some rectangular diagonal matrix $\Sigma$, where $\widetilde{U}$ is obtained from $U$ by changing the signs of, without loss of generality, the first $\ell$ columns $\bu_1,\dots,\bu_\ell$ of $U$. Let $\sigma_1,\dots,\sigma_\ell$ be the corresponding positive singular values. Then, $E=(\widetilde{U}-U)\Sigma V^T=-2\begin{bmatrix}
	\bu_1 & \cdots & \bu_\ell & 0
	\end{bmatrix}\Sigma V^T$. So, each row of $E$ is a linear combination of the first $\ell$ rows $\bv_1^T,\dots,\bv_\ell^T$ of $V^T$. Hence, $\mathrm{Row}(E)=\mathrm{span}\{\bv_1\,\dots,\bv_\ell\}$, and this implies $\ell=k=\mathrm{rank}(E)$. Moreover, since $EV=(\widetilde{U}-U)\Sigma$, we have $E\bv_i=-2\sigma_i\bu_i=-2A\bv_i$ for $i=1,\dots, k$. So, $(2A+E)\bv_i=0$ for $i=1,\dots,k$. Therefore, $(2A+E)E^T=0$. Furthermore, applying the same argument above for finding the singular value decomposition of $-E$, we can find that \ref{cond1:change of sign} and \ref{cond2:change of sign} hold.
\end{proof}

\begin{remark}\label{Remark:E^T(2A+E)=0}
	By a similar argument as in the proof of Proposition \ref{Prop:change of sign}, one can establish that $(A,A+E)$ is a pair of Gram mates, $E^T(2A+E)=0$, and $k=\mathrm{rank}(E)$ if and only if for a $(0,1)$ matrix $A$, a $(0,1)$ matrix $A+E$ is obtained from $A$ by changing the signs of $k$ positive singular values. For the proof of the converse, one can begin with $A=U\Sigma V^T$ and $A+E=U\Sigma \widetilde{V}^T$ for some orthogonal matrices $U$, $V$, and $\widetilde{V}$, where $\widetilde{V}$ is obtained from $V$ by changing the signs of $k$ columns of $V$ corresponding to the $k$ positive singular values. 
\end{remark}

\begin{remark}
	Let $E$ be an $m\times n$ realizable matrix of rank $k$, and $(A,A+E)$ be a pair of Gram mates. Suppose that $(2A+E)E^T=0$. Then, $\mathrm{rank}(2A+E)\leq n-k$. There are $k$ right singular vectors of $A$ that correspond to positive singular values and comprise a basis of $\mathrm{Row}(E)$. So, we have $E\bv=0$ for any right singular vector $\bv\notin\mathrm{Row}(E)$. It follows that $E\bx=0$ for $\bx\notin\mathrm{Row}(E)$. Therefore, $A^TA$ is singular if and only if  $\mathrm{rank}(2A+E)<n-k$. Furthermore, $\mathrm{rank}(2A+E)=n-k-l$ where $l$ is the nullity of $A^TA$.
\end{remark}

\begin{theorem}\label{Thm:equivalent conditions}
	Let $E$ be a realizable matrix of rank $k$, and let $A$ be a $(0,1)$ matrix such that $A+E$ is a $(0,1)$ matrix. Then, the following are equivalent:
	\begin{enumerate}[label=(\roman*)]
		\item\label{c0} $(A,A+E)$ is a pair of Gram mates and $(2A+E)E^T=0$.
		\item \label{c1} $(A,A+E)$ is a pair of Gram mates and $E^T(2A+E)=0$.
		\item\label{c2} $A+E$ is obtained from $A$ by changing the signs of $k$ positive singular values of $A$. (Here the $k$ positive singular values are the same as those of $-\frac{1}{2}E$.)
		\item\label{c3} There exist $k$ right singular vectors $\bv_1,\dots,\bv_k$ of $A$ corresponding to positive singular values such that the vectors $\bv_1,\dots,\bv_k$ form a basis of $\mathrm{Row}(E)$ and $\bv_i$ is a null vector of $2A+E$ for $i=1,\dots,k$. (Here $\bv_1,\dots,\bv_k$ can be obtained from right singular vectors corresponding to the positive singular values of $-\frac{1}{2}E$.)
		\item\label{c4} There exist $k$ left singular vectors $\bu_1,\dots,\bu_k$ of $A$ corresponding to positive singular values such that the vectors $\bu_1,\dots,\bu_k$ form a basis of $\mathrm{Col}(E)$ and $\bu_i$ is a null vector of $(2A+E)^T$ for $i=1,\dots,k$. (Here $\bu_1,\dots,\bu_k$ can be obtained from left singular vectors corresponding to the positive singular values of $-\frac{1}{2}E$.)
		\item\label{c5} $(A,A+E)$ is a pair of Gram mates and $AE^T$ is symmetric.
		\item\label{c6} $(A,A+E)$ is a pair of Gram mates and $A^TE$ is symmetric.
	\end{enumerate}
\end{theorem}
\begin{proof}
	\ref{c0} $\Leftrightarrow$ \ref{c1} $\Leftrightarrow$ \ref{c2}: It is clear from Proposition \ref{Prop:change of sign} and Remark \ref{Remark:E^T(2A+E)=0}.
	
	\ref{c0} $\Leftrightarrow$ \ref{c3} and \ref{c1} $\Leftrightarrow$ \ref{c4}: Using Lemma \ref{Lemma:rowspace(E) and span of right singular vectors} and Proposition \ref{Prop:change of sign}, the proof is straightforward.
	
	\ref{c0} $\Leftrightarrow$ \ref{c5} and \ref{c1} $\Leftrightarrow$ \ref{c6}: Since $AA^T=(A+E)(A+E)^T$, we have $AE^T+EA^T+EE^T=0$. Hence, $(2A+E)E^T=0$ implies $EA^T=AE^T$, and vice versa. Similarly, from $A^TA=(A+E)^T(A+E)$, we find \ref{c1} $\Leftrightarrow$ \ref{c6}.
\end{proof}

Theorem \ref{Thm:equivalent conditions} can be recast with respect to Gram mates $A$ and $B$.

\begin{corollary}\label{Cor:equivalent conditions for convertible}
	Let $A$ and $B$ be $(0,1)$ matrices with $A\neq B$, and let $k=\mathrm{rank}(A-B)$. Then, the following are equivalent:
	\begin{enumerate}[label=(\roman*)]
		\item\label{cond:convertible 1} $(A,B)$ is a pair of Gram mates and $(A+B)(A-B)^T=0$.
		\item $(A,B)$ is a pair of Gram mates and $(A-B)^T(A+B)=0$.
		\item\label{d2} $B$ is obtained from $A$ by changing the signs of $k$ positive singular values. (Here the $k$ positive singular values are the same as those of $\frac{1}{2}(A-B)$.)
		\item\label{d3} There exist $k$ right singular vectors $\bv_1,\dots,\bv_k$ of $A$ corresponding to positive singular values such that the vectors $\bv_1,\dots,\bv_k$ form a basis of $\mathrm{Row}(A-B)$ and $\bv_i$ is a null vector of $A+B$ for $i=1,\dots,k$. (Here $\bv_1,\dots,\bv_k$ are obtained from right singular vectors corresponding to the positive singular values of $\frac{1}{2}(A-B)$.)
		\item\label{d4} There exist $k$ left singular vectors $\bu_1,\dots,\bu_k$ of $A$ corresponding to positive singular values such that the vectors $\bu_1,\dots,\bu_k$ form a basis of $\mathrm{Col}(A-B)$ and $\bu_i$ is a null vector of $(A+B)^T$ for $i=1,\dots,k$. (Here $\bu_1,\dots,\bu_k$ are obtained from left singular vectors corresponding to the positive singular values of $\frac{1}{2}(A-B)$.)
		\item\label{cond:convertible 6} $(A,B)$ is a pair of Gram mates and $A(A-B)^T$ is symmetric.
		\item $(A,B)$ is a pair of Gram mates and $A^T(A-B)$ is symmetric.
	\end{enumerate}
\end{corollary}

\begin{definition}\label{Definition:convertible}
	Let $A$ be a rectangular $(0,1)$ matrix. The matrix $A$ is said to be \textit{convertible} (to $B$) if there exists a $(0,1)$ matrix $B$ obtained from $A$ by changing the signs of $k$ positive singular values (possibly with repetition) for some $k\geq 1$. Such $A$ and $B$ are called \textit{convertible Gram mates}. Let $\bv_1,\dots,\bv_k$ be right singular vectors of $A$ corresponding to the $k$ singular values whose signs are changed. We say that the $k$ positive singular values are the \textit{Gram singular values} of Gram mates $A$ and $B$ (corresponding to $\bv_1,\dots,\bv_k$). The matrix $A$ is said to \textit{have Gram singular values} if $A$ is convertible. 
\end{definition}

\begin{remark}
	In order to clarify Definition \ref{Definition:convertible}, consider a $(0,1)$ matrix $A$ convertible to $B$. Even though the signs of $k$ positive singular values of $A$ for some $k\geq 1$ are changed for obtaining $B$, by Remark \ref{Remark:meaning of sign changes} $A$ and $B$ have the same singular values. By converting the signs of the $k$ singular values of $B$, we can obtain $A$ from $B$. So, $B$ is convertible to $A$. Hence, we may use the term `convertible Gram mates'.
	
	Let $A=U\Sigma V^T$ where $U$ and $V$ are orthogonal matrices and $\Sigma$ is a rectangular diagonal matrix. Suppose that there exist two ways of obtaining $B$ from $A$ by changing the signs of $k$ positive diagonal entries of $\Sigma$, say $B=U\Sigma_1 V^T$ and $B=U\Sigma_2 V^T$. Clearly, $\Sigma_1=\Sigma_2$. Thus, if $A$ is convertible to $B$, then the $k$ positive singular values of $A$ whose signs are changed and their corresponding singular vectors are uniquely determined. Thus, we may use the term `the Gram singular values of $A$ and $B$'. Furthermore, if there are repeated values among the Gram singular values, then we need to indicate which positions on the main diagonal of $\Sigma$ corresponding to the repeated values are chosen for the sign changes. Therefore, if the Gram singular values are not distinct, then we need to specify corresponding right singular vectors (or left singular vectors).
\end{remark}

\begin{remark}\label{Remark3:convertible}	
	Let $A$ and $B$ be $(0,1)$ matrices with $A\neq B$. If one of the conditions in Corollary \ref{Cor:equivalent conditions for convertible} holds, then $A$ and $B$ are convertible Gram mates; furthermore, all singular vectors of $A$ corresponding to the Gram singular values of $A$ and $B$ can be obtained from those corresponding to all positive singular values of $\frac{1}{2}(A-B)$. One can establish analogous results with respect to Gram mates $A$ and $A+E$ via a realizable matrix $E$ by using Theorem \ref{Thm:equivalent conditions}.
\end{remark}

\begin{example}
	Suppose that a $(0,1)$ matrix $Q$ is a Gram mate to the identity matrix $I$. Clearly, $Q\neq I$ and $Q$ is a permutation matrix. By \ref{cond:convertible 6} of Corollary \ref{Cor:equivalent conditions for convertible}, $Q$ is convertible to $I$ if and only if $I(I-Q)^T$ is symmetric, \textit{i.e.,} $Q$ is symmetric. Therefore, any non-convertible Gram mate to $I$ is a non-symmetric permutation matrix.
\end{example}

\iffalse 
\begin{example}
	In this example, we find all $(0,1)$ matrices convertible to the identity matrix $I_n$. Suppose that $Q$ is convertible to $I$. By \ref{cond:convertible 6} of Corollary \ref{Cor:equivalent conditions for convertible}, $Q$ is a Gram mate to $I$ and $I(I-Q)^T$ is symmetric. Since $QQ^T=Q^TQ=I$, $Q$ is a permutation matrix. Moreover, $Q$ is symmetric. Therefore, any $(0,1)$ matrix obtained from $I$ by changing signs of some singular values is a symmetric permutation matrix which is not the identity matrix, and vice versa.
\end{example}
\fi

\begin{example}\label{Example:Gram mates A1 A2; A2 A1}
	Let $A=\begin{bmatrix}
	A_1 & A_2\\
	A_2 & A_1
	\end{bmatrix}$ be a $(0,1)$ matrix where $A_1$ and $A_2$ have the same size and $A_1\neq A_2$. It can be checked that $B=\begin{bmatrix}
	A_2 & A_1\\
	A_1 & A_2
	\end{bmatrix}$ is a Gram mate to $A$. Furthermore, one can verify that $(A+B)(A-B)^T=0$. Since the condition \ref{cond:convertible 1} of Corollary \ref{Cor:equivalent conditions for convertible} holds, $A$ and $B$ are convertible Gram mates. Furthermore, it follows from the structure of $A-B$ that the Gram singular values of $A$ and $B$ can be obtained from the $k$ positive singular values of $A_1-A_2$ where $k=\mathrm{rank}(A_1-A_2)$.
\end{example}

%%%%%%%%%%%%%%%%%%%%%%%%%%%%%%%%%%%%%%%%%%%%%%%%%%%%%%%%%%%%%%%%%%%%%%%%%%%%%%%%%%%%%%%%%%%%%%%%
\section{Gram mates via realizable matrices of rank 1 and 2}\label{section:rank1 and 2}
%%%%%%%%%%%%%%%%%%%%%%%%%%%%%%%%%%%%%%%%%%%%%%%%%%%%%%%%%%%%%%%%%%%%%%%%%%%%%%%%%%%%%%%%%%%%%%%%

In this section, we shall completely characterize Gram mates $A$ and $B$ where the rank of $A-B$ is $1$ or $2$. We also investigate convertible Gram mates $A$ and $B$, their Gram singular values, and corresponding singular vectors.

Recall that given a realizable matrix $E$ and a pair of Gram mates $(A,A+E)$, $PAQ$ and $P(A+E)Q$ are Gram mates for any appropriately sized permutation matrices $P$ and $Q$. Hence, we may consider a $(0,1,-1)$ matrix $E=\begin{bmatrix}
\widetilde{E} & 0\\
0 & 0
\end{bmatrix}$ such that $E\mathbf{1}=0$ and $\mathbf{1}^TE=0^T$.

\begin{proposition}[{\cite[Lemma 2.1]{Steve:two-mode}}]\label{Prop:Grammates for E from tilde of E}
	Let $E=\begin{bmatrix}
	\widetilde{E} & 0 \\
	0 & 0
	\end{bmatrix}$ be realizable, and let $A=\begin{bmatrix}
	\widetilde{A} & X_1 \\
	X_2 & X_3
	\end{bmatrix}$ be compatible with the partition of $E$. Then, $A$ and $A+E$ are Gram mates if and only if $(\widetilde{A},\widetilde{A}+\widetilde{E})$ is a pair of Gram mates, $\widetilde{E}X_2^T=0$ and $\widetilde{E}^TX_1=0$.
\end{proposition}

\begin{remark}\label{Remark:left right null vectors for E}
	Note that $X_1^T\widetilde{E}=0$ and $\widetilde{E}X_2^T=0$ if and only if columns of $X_1$ are $(0,1)$ left null vectors of $\wtE$, and rows of $X_2$ are $(0,1)$ right null vectors of $\wtE$.
\end{remark}

\begin{proposition}\label{Prop4: E realizable iff wtE is realizable}
	Let $E=\begin{bmatrix}
	\widetilde{E} & 0\\
	0 & 0
	\end{bmatrix}$ be a $(0,1,-1)$ matrix such that $E\mathbf{1}=0$ and $\mathbf{1}^TE=0^T$. Then, $E$ is realizable if and only if $\widetilde{E}$ is realizable.
\end{proposition}
\begin{proof}
	It follows from Proposition \ref{Prop:Grammates for E from tilde of E} that if $E$ is realizable, so is $\widetilde{E}$. Conversely, suppose that $\widetilde{E}$ is realizable. Let $(\widetilde{A},\widetilde{A}+\widetilde{E})$ be a pair of Gram mates. Then, $\begin{bmatrix}
	\widetilde{A} & 0 \\
	0 & 0
	\end{bmatrix}$ and $\begin{bmatrix}
	\widetilde{A}+\widetilde{E} & 0 \\
	0 & 0
	\end{bmatrix}$ that are compatible with the partition of $E$ are Gram mates. Hence, $E$ is realizable.
\end{proof}

\begin{proposition}\label{Prop: Gramsingular A and tilde A}
	Let $E=\begin{bmatrix}
	\widetilde{E} & 0 \\
	0 & 0
	\end{bmatrix}$ be realizable, and let $A=\begin{bmatrix}
	\widetilde{A} & X_1 \\
	X_2 & X_3
	\end{bmatrix}$ be compatible with the partition of $E$. Suppose that $A$ and $A+E$ are Gram mates. Then, $AE^T=EA^T$ if and only if $\widetilde{A}\widetilde{E}^T=\widetilde{E}\widetilde{A}^T$. This implies that if $A$ is convertible to $A+E$, so is $\widetilde{A}$ to $\widetilde{A}+\widetilde{E}$, and vice versa.
\end{proposition}
\begin{proof}
	From Proposition \ref{Prop:Grammates for E from tilde of E}, we have $\widetilde{E}X_2^T=0$. It can be readily checked from computation that $AE^T=EA^T$ if and only if $\widetilde{A}\widetilde{E}^T=\widetilde{E}\widetilde{A}^T$. By Remark \ref{Remark3:convertible}, the desired conclusion follows.
\end{proof}

\begin{proposition}\label{Prop:singular vecotrs of E and tilde of E}
	Let $E=\begin{bmatrix}
	\widetilde{E} & 0 \\
	0 & 0
	\end{bmatrix}$ be realizable, and let $A=\begin{bmatrix}
	\widetilde{A} & X_1 \\
	X_2 & X_3
	\end{bmatrix}$ be compatible with the partition of $E$. Suppose that $A$ and $A+E$ are Gram mates. We may assume (by Lemma \ref{Lemma:rowspace(E) and span of right singular vectors}) that vectors $\tilde{\bv}_1,\dots,\tilde{\bv}_k$ form a basis of $\mathrm{Row}(\widetilde{E})$, where $\tilde{\bv}_i$ is a right singular vector associated to a positive singular value $\sigma_i$ of $\widetilde{A}$ for $i=1,\dots,k$. Furthermore, suppose that for $i=1,\dots,k$, $\tilde{\bu}_i$ is the corresponding left singular vector. Then, $\begin{bmatrix} \tilde{\bv}_i \\ 0\end{bmatrix}$ $\left(\text{resp.} \begin{bmatrix} \tilde{\bu}_i \\ 0\end{bmatrix}\right)$ is a right (resp. left) singular vector corresponding to $\sigma_i$ of $A$ for $i=1,\dots,k$.
\end{proposition}
\begin{proof}
	Using Proposition \ref{Prop:Grammates for E from tilde of E}, we have $\widetilde{E}X_2^T=0$ and $\widetilde{E}^TX_1=0$, \textit{i.e.}, $\mathrm{Row}(\widetilde{E})$ and $\mathrm{Col}(\widetilde{E})$ are orthogonal to $\mathrm{Row}(X_2)$ and $\mathrm{Col}(X_1)$, respectively. Let $\tilde{\bv}$ be a right singular vector  corresponding to a positive singular value $\sigma$ of $\widetilde{A}$ such that $\tilde{\bv}\in\mathrm{Row}(\widetilde{E})$. Since $\widetilde{E}X_2^T=0$, we have $X_2\tilde{\bv}=0$. For $\sigma\tilde{\bu}=\widetilde{A}\tilde{\bv}$ where $\tilde{\bu}$ is the corresponding left singular vector of $\widetilde{A}$, we have $A\begin{bmatrix} \tilde{\bv} \\ 0\end{bmatrix}=\begin{bmatrix} \widetilde{A}\tilde{\bv} \\ X_2\tilde{\bv}\end{bmatrix}=\sigma\begin{bmatrix} \tilde{\bu} \\ 0\end{bmatrix}$. Since $\tilde{\bv}\in\mathrm{Row}(\widetilde{E})$, by Proposition \ref{Prop:rowcolumnspaceofE} we obtain $\sigma\tilde{\bu}=\widetilde{A}\tilde{\bv}\in\mathrm{Col}(\widetilde{E})$. We find from $\widetilde{E}^TX_1=0$ that $X_1^T(\sigma\tilde{\bu})=X_1^T\widetilde{A}\tilde{\bv}=0$. Note that $\widetilde{A}^T\tilde{\bu}=\sigma\tilde{\bv}$. Hence, $A^T\begin{bmatrix}\tilde{\bu} \\ 0\end{bmatrix}=\begin{bmatrix} \widetilde{A}^T\tilde{\bu} \\ X_1^T\tilde{\bu}\end{bmatrix}=\sigma\begin{bmatrix} \tilde{\bv} \\ 0\end{bmatrix}$. Therefore, our desired result is obtained.
\end{proof}

\begin{corollary}\label{Cor:A and tilde of A has Gram singular values}
	Let $E=\begin{bmatrix}
	\widetilde{E} & 0 \\
	0 & 0
	\end{bmatrix}$ be realizable, and let $A=\begin{bmatrix}
	\widetilde{A} & X_1 \\
	X_2 & X_3
	\end{bmatrix}$ be compatible with the partition of $E$. Suppose that $A$ and $A+E$ are Gram mates. If $\widetilde{A}$ is convertible to $\widetilde{A}+\widetilde{E}$, then $A$ is convertible to $A+E$; the Gram singular values of $A$ and $A+E$ are the same as those of $\widetilde{A}$ and $\widetilde{A}+\widetilde{E}$; and the corresponding singular vectors of $A$ are obtained from those of $\widetilde{A}$ by adjoining a column of zeros.
\end{corollary}
\begin{proof}
	Combining Propositions \ref{Prop: Gramsingular A and tilde A} and \ref{Prop:singular vecotrs of E and tilde of E}, the conclusion is straightforward.
\end{proof}

Summarizing Propositions \ref{Prop:Grammates for E from tilde of E}--\ref{Prop:singular vecotrs of E and tilde of E} and Corollary \ref{Cor:A and tilde of A has Gram singular values}, given a realizable matrix $E=\begin{bmatrix}
\widetilde{E} & 0 \\
0 & 0
\end{bmatrix}$, characterizing realizability of $\widetilde{E}$, Gram mates via $\widetilde{E}$, convertible Gram mates via $\widetilde{E}$, and their Gram singular values and corresponding singular vectors endows $E$ with the same properties except that we need the extra conditions in Proposition \ref{Prop:Grammates for E from tilde of E} to find Gram mates via $E$.

%%%%%%%%%%%%%%%%%%%%%%%%%%%%%%%%%%%%%%%%%
\subsection{Gram mates via matrices of rank $1$}\label{Subsec:realizable rank 1}
%%%%%%%%%%%%%%%%%%%%%%%%%%%%%%%%%%%%%%%%%
Suppose that a realizable matrix $E$ is of rank $1$. Without loss of generality,
$$
E=\begin{bmatrix}
J_{k_1,k_2} & -J_{k_1,k_2} & 0\\
-J_{k_1,k_2} & J_{k_1,k_2} & 0\\
0 & 0 & 0
\end{bmatrix}
$$
for some $k_1, k_2>0$. Let $\wtE=\begin{bmatrix}
J_{k_1,k_2} & -J_{k_1,k_2}\\
-J_{k_1,k_2} & J_{k_1,k_2}
\end{bmatrix}$. It is straightforward that $\wtA=\begin{bmatrix}
0 & J_{k_1,k_2}\\
J_{k_1,k_2} & 0
\end{bmatrix}$ and $\wtA+\wtE$ are the only pair of Gram mates via $E$. Thus, $\wtE$ is realizable. Furthermore, $\wtA\wtE^T=\wtE\wtA^T$. This implies that $\wtA$ is convertible to $\wtA+\wtE$. Since $\mathrm{rank}(\wtE)=1$, there is only one positive singular value of $-\frac{1}{2}\wtE$, which is the Gram singular value of $\wtA$ and $\wtA+\wtE$. One can find that the positive singular value of $-\frac{1}{2}\wtE$ is $\sqrt{k_1k_2}$ and the corresponding left and the corresponding right singular vector are $\frac{1}{\sqrt{2k_1}}\begin{bmatrix} -\mathbf{1}_{k_1}\\\mathbf{1}_{k_1}\end{bmatrix}$ and $\frac{1}{\sqrt{2k_2}}\begin{bmatrix} \mathbf{1}_{k_2}\\-\mathbf{1}_{k_2}\end{bmatrix}$ up to sign, respectively.

\begin{theorem}\label{Thm:Gram mates of rank1}
	Suppose that $E$ is a realizable matrix of rank $1$: 
	$$
	E=\begin{bmatrix}
	J_{k_1,k_2} & -J_{k_1,k_2} & 0\\
	-J_{k_1,k_2} & J_{k_1,k_2} & 0\\
	0 & 0 & 0
	\end{bmatrix}
	$$
	for some $k_1, k_2>0$. Let $A=\begin{bmatrix}
	0 & J_{k_1,k_2} & X_1\\
	J_{k_1,k_2} & 0 & X_2\\
	X_3 & X_4 & Y
	\end{bmatrix}$ be a $(0,1)$ matrix compatible with the partition of $E$. Then, we have the following:
	\begin{enumerate}[label=(\roman*)]
		\item $A$ and $A+E$ are Gram mates if and only if $\mathbf{1}^TX_1=\mathbf{1}^TX_2$ and $X_3\mathbf{1}=X_4\mathbf{1}$.
		\item $E$ is realizable.
		\item For any Gram mates via $E$, they are convertible each other.
		\item For any Gram mates via $E$, their Gram singular value is $\sqrt{k_1k_2}$, and the corresponding left and right singular vectors are (up to sign) $\frac{1}{\sqrt{2k_1}}\begin{bmatrix} -\mathbf{1}_{k_1}\\\mathbf{1}_{k_1}\\0\end{bmatrix}$ and $\frac{1}{\sqrt{2k_2}}\begin{bmatrix} \mathbf{1}_{k_2}\\-\mathbf{1}_{k_2}\\0\end{bmatrix}$, respectively.
	\end{enumerate}
\end{theorem}
\begin{proof}
	Applying Proposition \ref{Prop:Grammates for E from tilde of E}, we have 
	\begin{align*}
	\begin{bmatrix}
	J_{k_2,k_1} & -J_{k_2,k_1}\\
	-J_{k_2,k_1} & J_{k_2,k_1}
	\end{bmatrix}\begin{bmatrix} X_1 \\ X_2\end{bmatrix}=0,\;\begin{bmatrix}
	X_3 & X_4
	\end{bmatrix}\begin{bmatrix}
	J_{k_2,k_1} & -J_{k_2,k_1}\\
	-J_{k_2,k_1} & J_{k_2,k_1}
	\end{bmatrix}=0.
	\end{align*}
	Since $J_{k_1,k_2}$ is an all ones matrix, we have $\mathbf{1}^TX_1=\mathbf{1}^TX_2$ and $X_3\mathbf{1}=X_4\mathbf{1}$. Applying Propositions \ref{Prop4: E realizable iff wtE is realizable}--\ref{Prop:singular vecotrs of E and tilde of E} and Corollary \ref{Cor:A and tilde of A has Gram singular values} with the argument immediately preceding this theorem, the desired conclusions follow.
\end{proof}

\begin{remark}
	Consider $A$ and $A+E$ in Theorem \ref{Thm:Gram mates of rank1}. Let $B=A+E$ and $k=k_1=k_2$. One can check that $AB=BA$ if and only if $A$ and $B$ are Gram mates. Suppose that $A$ and $B$ are diagonalizable. It is known (see \cite{Horn:MatrixAnalysis}) that $A$ and $B$ are simultaneously diagonalizable. Since $\mathrm{rank}(A-B)=1$, we have $|\sigma(A)-\sigma(B)|=1$ where $\sigma(A)$ and $\sigma(B)$ are multisets of eigenvalues of $A$ and $B$, respectively. It follows that $\sigma(A)-\sigma(B)=\{-k\}$ and $\sigma(B)-\sigma(A)=\{k\}$.
\end{remark}

\begin{example}\label{Example:rank1 Gram mates}
	Let $$A=\left[\begin{array}{cc|cc|ccc}
	0 & 0 & 1 & 1 & 1 & 1 & 0\\
	0 & 0 & 1 & 1 & 0 & 0 & 1\\\hline
	1 & 1 & 0 & 0 & 1 & 1 & 1\\
	1 & 1 & 0 & 0 & 0 & 0 & 0\\\hline
	1 & 0 & 1 & 0 & 1 & 1 & 1\\
	1 & 0 & 1 & 0 & 1 & 1 & 1\\
	0 & 1 & 1 & 0 & 1 & 1 & 1
	\end{array}\right],\;E=\left[\begin{array}{cc|cc|ccc}
	1 & 1 & -1 & -1 & 0 & 0 & 0\\
	1 & 1 & -1 & -1 & 0 & 0 & 0\\\hline
	-1 & -1 & 1 & 1 & 0 & 0 & 0\\
	-1 & -1 & 1 & 1 & 0 & 0 & 0\\\hline
	0 & 0 & 0 & 0 & 0 & 0 & 0\\
	0 & 0 & 0 & 0 & 0 & 0 & 0\\
	0 & 0 & 0 & 0 & 0 & 0 & 0
	\end{array}\right].$$ 
	Clearly, $E$ has rank $1$. By Theorem \ref{Thm:Gram mates of rank1}, $A$ and $A+E$ are Gram mates, and the Gram singular value is $2$. Moreover, $\frac{1}{2}\begin{bmatrix} \mathbf{1}_{2}\\-\mathbf{1}_{2}\\ \mathbf{0}_3\end{bmatrix}$ is a corresponding right singular vector.
\end{example}

%%%%%%%%%%%%%%%%%%%%%%%%%%%%%%%%%%%%%%%%%
\subsection{Gram mates via matrix of rank $2$}
%%%%%%%%%%%%%%%%%%%%%%%%%%%%%%%%%%%%%%%%%
We first show that given a $(0,1,-1)$ matrix $E=\begin{bmatrix}
\widetilde{E} & 0\\
0 & 0
\end{bmatrix}$ of rank $2$ with $E\mathbf{1}=0$ and $\mathbf{1}^TE=0^T$, $\widetilde{E}$ is in one of forms \ref{index m1}--\ref{index m5}. As done in Subsection \ref{Subsec:realizable rank 1}, we examine properties related to Gram mates via $\widetilde{E}$ for types \ref{index m4} and \ref{index m5}. We also show that the types \ref{index m1}--\ref{index m3} inherit the same properties from the type \ref{index m4}. For ease of exposition, we only present an interpretation from the viewpoint of the null space in order to find Gram mates via $E$ by using Gram mates via $\wtE$ with Proposition \ref{Prop:Grammates for E from tilde of E}.

Unless stated otherwise, we assume that a $(0,1,-1)$ matrix $E$ has neither a row of zeros nor a column of zeros. Suppose that a $(0,1,-1)$ matrix $E$ with $E\mathbf{1}=0$ and $\mathbf{1}^TE=0^T$ is of rank $2$. Then, there are two $(0,1,-1)$ rows $\bx_1^T$ and $\bx_2^T$ of $E$ that form a basis of $\mathrm{Row}(E)$. Without loss of generality,
\begin{align*}
\bx_1^T&=\left[\begin{array}{cccc|cccc}
\mathbf{1}_{\alpha_1}^T & \mathbf{1}_{\alpha_2}^T & -\mathbf{1}_{\alpha_3}^T & -\mathbf{1}_{\alpha_4}^T & \mathbf{1}_{\beta_1}^T & -\mathbf{1}_{\beta_2}^T & \mathbf{0}_{\beta_3}^T & \mathbf{0}_{\beta_4}^T \end{array}\right],\\
\bx_2^T&=\left[\begin{array}{cccc|cccc}
\mathbf{1}_{\alpha_1}^T & -\mathbf{1}_{\alpha_2}^T & \mathbf{1}_{\alpha_3}^T & -\mathbf{1}_{\alpha_4}^T & \mathbf{0}_{\beta_1}^T & \mathbf{0}_{\beta_2}^T & \mathbf{1}_{\beta_3}^T & -\mathbf{1}_{\beta_4}^T
\end{array}\right]
\end{align*}
where $\alpha_1+\alpha_2+\beta_1=\alpha_3+\alpha_4+\beta_2>0$, $\alpha_1+\alpha_3+\beta_3=\alpha_2+\alpha_4+\beta_4>0$, $\alpha_i,\beta_i\geq 0$ for $i=1,2,3, 4$.

Consider further conditions for the indices and pairs $(\alpha,\beta)$ such that $\alpha\bx_1+\beta\bx_2$ is a $(0,1,-1)$ vector for the following three cases: \begin{enumerate*}[label=(\roman*)]
	\item $\beta_1+\beta_2>0$ and $\beta_3+\beta_4>0$, \item either $\beta_1+\beta_2=0$ or $\beta_3+\beta_4=0$, \item $\beta_i=0$ for $i=1,\dots,4$.
\end{enumerate*}

\smallskip

\begin{itemize}[wide=0pt]
	\item Suppose that $\beta_1+\beta_2>0$ and $\beta_3+\beta_4>0$. Evidently, $\bx_1$ and $\bx_2$ are linearly independent. Since $\alpha\bx_1+\beta\bx_2$ is a $(0,1,-1)$ vector, we have $\alpha,\beta\in\{0,1,-1\}$. Considering $\bx_1^T\mathbf{1}=\bx_2^T\mathbf{1}=0$ and possible $(0,1,-1)$ linear combinations of $\bx_1$ and $\bx_2$, we have three subcases:
	\begin{enumerate}[label=(C\arabic*)]
		\item Suppose that $\alpha_1+\alpha_4>0$ and $\alpha_2+\alpha_3>0$. If $\alpha,\beta\in\{1,-1\}$, then $\alpha\bx_1+\beta\bx_2$ is not a $(0,1,-1)$ vector. So, $(\alpha,\beta)\in\{(\pm 1,0),(0,\pm 1)\}$. Moreover, we have $\alpha_1+\alpha_2+\beta_1=\alpha_3+\alpha_4+\beta_2$, $\alpha_1+\alpha_3+\beta_3=\alpha_2+\alpha_4+\beta_4$.
		\item If without loss of generality $\alpha_1+\alpha_4>0$ and $\alpha_2+\alpha_3=0$, then $\alpha_2=\alpha_3=0$, $\alpha_1+\beta_1=\alpha_4+\beta_2$, $\alpha_1+\beta_3=\alpha_4+\beta_4$ and $(\alpha,\beta)\in\{(\pm 1,0),(0,\pm 1),(\pm 1,\mp 1)\}$.
		\item If $\alpha_1+\alpha_4=0$ and $\alpha_2+\alpha_3=0$, then $\alpha_i=0$ for $i=1,\dots,4$, $\beta_1=\beta_2$, $\beta_3=\beta_4$ and $(\alpha,\beta)\in\{(\pm 1,0),(0,\pm 1),(\pm 1,\pm 1),(\pm 1,\mp 1)\}$.
	\end{enumerate}
	
	\smallskip
	
	\item Suppose that either $\beta_1+\beta_2=0$ or $\beta_3+\beta_4=0$. Without loss of generality, $\beta_1+\beta_2>0$ and $\beta_3+\beta_4=0$. For $\bx_1$ and $\bx_2$ to be linearly independent, $\alpha_1+\alpha_2+\alpha_3+\alpha_4>0$. Considering $\bx_1^T\mathbf{1}=\bx_2^T\mathbf{1}=0$ and possible $(0,1,-1)$ linear combinations of $\bx_1$ and $\bx_2$, we have two subcases:
	\begin{enumerate}[label=(C\arabic*)]\addtocounter{enumi}{+3}
		\item Let $\alpha_1+\alpha_4>0$ and $\alpha_2+\alpha_3>0$. By an analogous argument as in \ref{index c1}, $(\alpha,\beta)\in\{(\pm 1,0),(0,\pm 1)\}$. Furthermore, $\alpha_1+\alpha_2+\beta_1=\alpha_3+\alpha_4+\beta_2$, $\alpha_1+\alpha_3=\alpha_2+\alpha_4$.
		\item If without loss of generality, $\alpha_1+\alpha_4>0$ and $\alpha_2+\alpha_3=0$, then $\alpha_2=\alpha_3=0$, $\alpha_1=\alpha_4$, $\beta_1=\beta_2$ and $(\alpha,\beta)\in\{(\pm 1,0),(0,\pm 1),(\pm 1,\mp 1)\}$.
	\end{enumerate}
	\item Assume that $\beta_i=0$ for $i=1,\dots,4$. Then, we have a single subcase:
	\begin{enumerate}[label=(C\arabic*)]\addtocounter{enumi}{+5}
		\item  Obviously, $\alpha_1=\alpha_4$ and $\alpha_2=\alpha_3$. Since $\bx_1$ and $\bx_2$ are linearly independent, $\alpha_i>0$ for $i=1,\dots,4$. Moreover, $$(\alpha,\beta)\in\left\{(\pm 1, 0),(0,\pm 1),\left(\pm \frac{1}{2},\pm \frac{1}{2}\right),\left(\pm \frac{1}{2},\mp \frac{1}{2}\right)\right\}.$$
	\end{enumerate}
\end{itemize}

Summarizing the conditions for the indices and pairs $(\alpha,\beta)$ in each of the six subcases, they can be recast as:
\begin{enumerate}[label=(C\arabic*)]
	\item\label{index c1} $\alpha_1+\alpha_4>0$, $\alpha_2+\alpha_3>0$, $\beta_1+\beta_2>0$, $\beta_3+\beta_4>0$, $\alpha_1+\alpha_2+\beta_1=\alpha_3+\alpha_4+\beta_2$, $\alpha_1+\alpha_3+\beta_3=\alpha_2+\alpha_4+\beta_4$ and $(\alpha,\beta)\in\{(\pm 1,0),(0,\pm 1)\}$;
	\item\label{index c2} $\alpha_1+\alpha_4>0$, $\alpha_2=\alpha_3=0$, $\beta_1+\beta_2>0$, $\beta_3+\beta_4>0$, $\alpha_1+\beta_1=\alpha_4+\beta_2$, $\alpha_1+\beta_3=\alpha_4+\beta_4$, and $(\alpha,\beta)\in\{(\pm 1,0),(0,\pm 1),(\pm 1,\mp 1)\}$;
	\item\label{index c3} $\alpha_i=0$ for $i=1,\dots,4$, $\beta_1=\beta_2>0$, $\beta_3=\beta_4>0$, and $(\alpha,\beta)\in\{(\pm 1,0),\linebreak(0,\pm 1),(\pm 1,\pm 1),(\pm 1,\mp 1)\}$;
	\item\label{index c4} $\alpha_1+\alpha_4>0$, $\alpha_2+\alpha_3>0$,  $\beta_1+\beta_2>0$, $\beta_3=\beta_4=0$,  $\alpha_1+\alpha_2+\beta_1=\alpha_3+\alpha_4+\beta_2$, $\alpha_1+\alpha_3=\alpha_2+\alpha_4$, and $(\alpha,\beta)\in\{(\pm 1,0),(0,\pm 1)\}$;	
	\item\label{index c5} $\alpha_1=\alpha_4>0$, $\alpha_2=\alpha_3=0$, $\beta_1=\beta_2>0$, $\beta_3=\beta_4=0$, and $(\alpha,\beta)\in\{(\pm 1,0),(0,\pm 1),(\pm 1,\mp 1)\}$;
	\item\label{index c6} $\alpha_1=\alpha_4>0$, $\alpha_2=\alpha_3>0$, $\beta_i=0$ for $i=1,\dots,4$, and $(\alpha,\beta)\in\{(\pm 1, 0),\linebreak(0,\pm 1),(\pm \frac{1}{2},\pm \frac{1}{2}),(\pm \frac{1}{2},\mp \frac{1}{2})\}$.	
\end{enumerate}

Now, we shall see that any $(0,1,-1)$ matrix $E$ with each of the conditions \ref{index c1}--\ref{index c6} corresponds to one of the following types \ref{index m1}--\ref{index m5} (up to transposition and permutation of rows and columns). Unless stated otherwise, we assume that all the indices of each block in matrices of types \ref{index m1}--\ref{index m5} are nonnegative and each of their row and column sum vectors is zero.
\begin{enumerate}[label=(M\arabic*)]
	\item\label{index m1}$\begin{bmatrix}
	J_{k,a} & J_{k,b} & -J_{k,b} & -J_{k,a} \\
	-J_{k,a} & -J_{k,b} & J_{k,b} & J_{k,a} \\
	J_{l,a} & -J_{l,b} & J_{l,b} & -J_{l,a} \\
	-J_{l,a} & J_{l,b} & -J_{l,b} & J_{l,a}
	\end{bmatrix}$ where all indices of each block are positive.\vspace{0.3cm}
	\item\label{index m2}$\begin{bmatrix}
	J_{k,e} & -J_{k,f} & 0 & 0 \\
	-J_{k,e} & J_{k,f} & 0 & 0 \\
	0 & 0 & J_{l,g} & -J_{l,h} \\
	0 & 0 & -J_{l,g} & J_{l,h} 
	\end{bmatrix}$ where all indices of each block are positive.\vspace{0.3cm} 
	\item\label{index m3}$\begin{bmatrix}
	J_{k,a} & J_{k,b} & -J_{k,c} & -J_{k,d} & J_{k,e} & -J_{k,f} \\
	-J_{k,a} & -J_{k,b} & J_{k,c} & J_{k,d} & -J_{k,e} & J_{k,f}\\
	J_{l,a} & -J_{l,b} & J_{l,c} & -J_{l,d} & 0 & 0 \\
	-J_{l,a} & J_{l,b} & -J_{l,c} & J_{l,d} & 0 & 0 \\
	\end{bmatrix}$ where $k,l>0$, $e+f>0$ and $a+b+c+d>0$.\vspace{0.3cm}
	\item\label{index m4}$\begin{bmatrix}
	J_{k,a} & J_{k,b} & -J_{k,c} & -J_{k,d} & J_{k,e} & -J_{k,f} & 0 & 0 \\
	-J_{k,a} & -J_{k,b} & J_{k,c} & J_{k,d} & -J_{k,e} & J_{k,f} & 0 & 0 \\
	J_{l,a} & -J_{l,b} & J_{l,c} & -J_{l,d} & 0 & 0 & J_{l,g} & -J_{l,h} \\
	-J_{l,a} & J_{l,b} & -J_{l,c} & J_{l,d} & 0 & 0 & -J_{l,g} & J_{l,h} \\
	\end{bmatrix}$ where $k,l>0$, $a+b+c+d>0$, $e+f>0$, and $g+h>0$.\vspace{0.3cm}
	\item\label{index m5}$\begin{bmatrix}
	J_{k,a} & -J_{k,b}  & J_{k,c} & -J_{k,d} & 0 & 0 \\
	-J_{l,a} & J_{l,b} & -J_{l,c} & J_{l,d} & 0 & 0 \\
	J_{p,a} & -J_{p,b}  & 0 & 0 & J_{p,e} & -J_{p,f} \\
	-J_{q,a} & J_{q,b} & 0 & 0 & -J_{q,e} & J_{q,f} \\
	0 & 0 & J_{r,c} & -J_{r,d} & -J_{r,e} & J_{r,f} \\
	0 & 0 & -J_{s,c} & J_{s,d} & J_{s,e} & -J_{s,f} \\
	\end{bmatrix}$ where $a+b$, $c+d$, $e+f$, $k+l$, $p+q$ and $r+s$ are positive.\vspace{0.3cm}
\end{enumerate}

Let us consider $\bx_1^T$ and $\bx_2^T$ where the condition \ref{index c6} holds. Suppose that $\alpha\bx^T_1+\beta\bx^T_2$ is a $(0,1,-1)$ vector for some $\alpha$ and $\beta$. Since $$(\alpha,\beta)\in\left\{(\pm 1, 0),(0,\pm 1),\left(\pm \frac{1}{2},\pm \frac{1}{2}\right),\left(\pm \frac{1}{2},\mp \frac{1}{2}\right)\right\},$$ we have four distinct $(0,1,-1)$ rows up to sign as follows:
\begin{align*}
\bx_1^T=\left[\begin{array}{cccc}
\mathbf{1}_{\alpha_1}^T  & \mathbf{1}_{\alpha_2}^T & -\mathbf{1}_{\alpha_2}^T & -\mathbf{1}_{\alpha_1}^T \end{array}\right]&,\;\frac{1}{2}(\bx_1+\bx_2)^T=\left[\begin{array}{cccc}
\mathbf{1}_{\alpha_1}^T & \mathbf{0}_{\alpha_2}^T & \mathbf{0}_{\alpha_2}^T & -\mathbf{1}_{\alpha_1}^T \end{array}\right],\\
\bx_2^T=\left[\begin{array}{cccc}
\mathbf{1}_{\alpha_1}^T & -\mathbf{1}_{\alpha_2}^T & \mathbf{1}_{\alpha_2}^T & -\mathbf{1}_{\alpha_1}^T \end{array}\right]&,\;\frac{1}{2}(\bx_1-\bx_2)^T=\left[\begin{array}{cccc}
\mathbf{0}_{\alpha_1}^T & \mathbf{1}_{\alpha_2}^T & -\mathbf{1}_{\alpha_2}^T & \mathbf{0}_{\alpha_1}^T \end{array}\right].
\end{align*}
Consider all possible combinations of the four row vectors that span the row space of a $(0,1,-1)$ matrix $E$ such that $\mathrm{rank}(E)=2$, $E\mathbf{1}=0$ and $\mathbf{1}^TE=0^T$. Suppose that $\bx_1^T$ and $\frac{1}{2}(\bx_1+\bx_2)^T$ are the only distinct rows in $E$ up to sign. Then, we generate $E'$ from $E$ by permuting rows as follows:
$$E'=\begin{bmatrix}
J_{\gamma_1,\alpha_1} & J_{\gamma_1,\alpha_2} & -J_{\gamma_1,\alpha_2} & -J_{\gamma_1,\alpha_1} \\
-J_{\gamma_2,\alpha_1} & -J_{\gamma_2,\alpha_2} & J_{\gamma_2,\alpha_2} & J_{\gamma_2,\alpha_1} \\
J_{\gamma_3,\alpha_1} & 0 & 0 & -J_{\gamma_3,\alpha_1} \\
-J_{\gamma_4,\alpha_1} & 0 & 0 & J_{\gamma_4,\alpha_1}
\end{bmatrix}$$
for some $\gamma_i\geq 0$ for $i=1,\dots,4$. Since $\mathbf{1}^TE'=0^T$, we have $\gamma_1=\gamma_2$ and $\gamma_3=\gamma_4$. By $\mathrm{rank}(E')=2$, $\gamma_i>0$ for $i=1,\dots,4$. Taking the transpose of $E'$ and permuting rows of $(E')^T$, we find that the resulting matrix is of type \ref{index m3}. Similarly, one can check that for the other choices among the four rows, $E$ must be of one of types \ref{index m1}--\ref{index m4} (up to transposition and permutation of rows and columns). 

Given $\bx_1^T$ and $\bx_2^T$ with \ref{index c2}, we have $(\alpha,\beta)\in\{(\pm 1,0),(0,\pm 1),(\pm 1,\mp 1)\}$ so that there are three distinct rows up to sign:
\begin{align*}
\bx_1^T&=\left[\begin{array}{cccccc}
\mathbf{1}_{\alpha_1}^T  & -\mathbf{1}_{\alpha_4}^T & \mathbf{1}_{\beta_1}^T & -\mathbf{1}_{\beta_2}^T & \mathbf{0}_{\beta_3}^T & \mathbf{0}_{\beta_4}^T \end{array}\right],\\
\bx_2^T&=\left[\begin{array}{cccccc}
\mathbf{1}_{\alpha_1}^T  & -\mathbf{1}_{\alpha_4}^T & \mathbf{0}_{\beta_1}^T & \mathbf{0}_{\beta_2}^T & \mathbf{1}_{\beta_3}^T & -\mathbf{1}_{\beta_4}^T \end{array}\right],\\
(\bx_1-\bx_2)^T&=\left[\begin{array}{cccccc}
\mathbf{0}_{\alpha_1}^T  & \mathbf{0}_{\alpha_4}^T & \mathbf{1}_{\beta_1}^T & -\mathbf{1}_{\beta_2}^T & -\mathbf{1}_{\beta_3}^T & \mathbf{1}_{\beta_4}^T \end{array}\right].
\end{align*}
Then, if the rows of $E$ consists of rows $\pm\bx_1^T$ and $\pm\bx_2^T$, then $E$ is of form \ref{index m4}; if $E$ consists of rows $\pm\bx_1^T$, $\pm\bx_2^T$ and $\pm(\bx_1-\bx_2)^T$, then $E$ is of form \ref{index m5}.

In this manner, one can verify that any $(0,1,-1)$ matrix $E$ whose either row space or column space spanned by $\bx_1$ and $\bx_2$ with \ref{index c1} is in type \ref{index m4}; $E$ with \ref{index c3} is in one of types\ref{index m1}--\ref{index m4}; $E$ with \ref{index c4} is in type \ref{index m3}; and $E$ with \ref{index c5} is in one of types \ref{index m2}--\ref{index m4}.

Here is a useful lemma for characterizing Gram mates via $E$ of rank $2$.

\begin{lemma}\label{Lemma:JY+XJ}
	Let $X_1$, $X_2$, $Y_1$ and $Y_2$ be matrices of sizes $k\times a$, $k\times b$, $l\times c$ and $l\times d$, respectively. Then, $J_{k,c}Y_1^T+J_{k,d}Y_2^T+X_1J_{a,l}+X_2J_{b,l}=\alpha J_{k,l}$ if and only if $X_1\mathbf{1}_a+X_2\mathbf{1}_b=x\mathbf{1}_k$ and $Y_1\mathbf{1}_c+Y_2\mathbf{1}_d=y\mathbf{1}_l$, where $x+y=\alpha$.
\end{lemma}
\begin{proof}
	Consider the sufficiency of the statements. Suppose that $\bx_1$, $\bx_2$, $\by_1$ and $\by_2$ are row sum vectors of $X_1$, $X_2$, $Y_1$ and $Y_2$, respectively. Then, we have $\begin{bmatrix}
	\by_1^T+\by_2^T\\
	\vdots\\
	\by_1^T+\by_2^T
	\end{bmatrix}+\begin{bmatrix}
	\bx_1+\bx_2 & \cdots & \bx_1+\bx_2
	\end{bmatrix}= \alpha J$. Let $(\bx)_i$ be the $i^{\text{th}}$ component of $\bx$. Considering the $j^\text{th}$ columns of the both sides, $\bx_1+\bx_2+(\by_1+\by_2)_j\mathbf{1}=\alpha\mathbf{1}$. Then, $\bx_1+\bx_2=x\mathbf{1}$ for some $x$. Similarly, from the rows of both sides, it can be deduced that $\by_1+\by_2=y\mathbf{1}$ for some $y$. Hence, the equation $J_{k,c}Y_1^T+J_{k,d}Y_2^T+X_1J_{a,l}+X_2J_{b,l}=\alpha J_{k,l}$ can be recast as $yJ+xJ=\alpha J$, and so $x+y=\alpha$. 
	
	The converse is straightforward.
\end{proof}

\subsubsection{Realizable matrices in the form \ref{index m1}--\ref{index m4}}

In this subsection, we first focus on Gram mates via a matrix of the form \ref{index m4}. It is shown that the cases \ref{index m1}--\ref{index m3} are special cases of \ref{index m4}.

\begin{lemma}\label{Lemma:Grammates A+B=J}
	Let $A$ and $B$ be $m\times n$ $(0,1)$ matrices such that $A+B=J$. Then, $AA^T=BB^T$ if and only if $n$ is even, and $A\mathbf{1}=B\mathbf{1}=\frac{n}{2}\mathbf{1}$. Similarly, $A^TA=B^TB$ if and only if $m$ is even, and $\mathbf{1}^TA=\mathbf{1}^TB=\frac{m}{2}\mathbf{1}m^T$.
\end{lemma}
\begin{proof}
	Suppose that $AA^T=BB^T$ and $A+B=J$. Since $(0,1)$ matrices $A$ and $B$ have the same row sum vector, $n$ is even and $A\mathbf{1}=B\mathbf{1}=\frac{n}{2}\mathbf{1}$. Conversely, assume that $n$ is even, and $A\mathbf{1}=B\mathbf{1}=\frac{n}{2}$. Then,
	\begin{align*}
	AA^T&=(J-B)(J-B)^T\\
	&=JJ^T-BJ^T-JB^T+BB^T=nJ-\frac{n}{2}J-\frac{n}{2}J+BB^T=BB^T.
	\end{align*}
	Similarly, one can deduce the remaining conclusions.
\end{proof}

\begin{remark}\label{Remark:Grammates M1}
	Suppose that $E$ is of form \ref{index m1}. By Lemma \ref{Lemma:Grammates A+B=J}, $(A,B)=(\frac{1}{2}(J+E),\frac{1}{2}(J-E))$ is a pair of Gram mates. Moreover, it is the only pair of Gram mates via $E$.
\end{remark}

\begin{lemma}\label{Lemma:supporting theorem of rank 21}
	Let $A=\begin{bmatrix}
	A_1 & A_2 & X\\
	A_3 & Y & A_4
	\end{bmatrix}$ and $B=\begin{bmatrix}
	B_1 & B_2 & X\\
	B_3 & Y & B_4
	\end{bmatrix}$ be $(0,1)$ matrices. Suppose that $(A-B)\mathbf{1}=0$, and $\mathbf{1}^TA_i=\mathbf{1}^TB_i$, $A_i+B_i=J$ for $i=1,\dots 4$. Then, $A$ and $B$ are Gram mates if and only if 
	\begin{align}\label{rank2:conditions1}
	&(B_2-A_2)Y^T+X(B_4^T-A_4^T)=A_1J^T+JA_3^T-JJ^T\\\label{rank2:conditions2}
	&(B_1^T-A_1^T)X=0,\\\label{rank2:conditions3}
	&(B_3^T-A_3^T)Y=0,\\\label{rank2:conditions4}
	&(B_2^T-A_2^T)X+Y^T(B_4-A_4)=0.
	\end{align}
\end{lemma}
\begin{proof}
	Since $(A-B)\mathbf{1}=0$, we have $$\begin{bmatrix}
	A_1 & A_2
	\end{bmatrix}\mathbf{1}=\begin{bmatrix}
	B_1 & B_2
	\end{bmatrix}\mathbf{1}\;\text{and}\;\begin{bmatrix}
	A_3 & A_4
	\end{bmatrix}\mathbf{1}=\begin{bmatrix}
	B_3 & B_4
	\end{bmatrix}\mathbf{1}.$$ 
	Applying Lemma \ref{Lemma:Grammates A+B=J} to pairs $(\begin{bmatrix}
	A_1 & A_2
	\end{bmatrix},\begin{bmatrix}
	B_1 & B_2
	\end{bmatrix})$ and $(\begin{bmatrix}
	A_3 & A_4
	\end{bmatrix},\begin{bmatrix}
	B_3 & B_4
	\end{bmatrix})$, we obtain $A_1A_1^T+A_2A_2^T=B_1B_1^T+B_2B_2^T$ and $A_3A_3^T+A_4A_4^T=B_3B_3^T+B_4B_4^T$. Let the numbers of rows of $A_1$ and $A_3$ be $k$ and $l$, respectively. Since $\mathbf{1}^TA_i=\mathbf{1}^TB_i$ and $A_i+B_i=J$ for $i=1,\dots 4$, the numbers $k$ and $l$ are even, $\mathbf{1}^TA_1=\mathbf{1}^TA_2=\frac{k}{2}\mathbf{1}^T$ and $\mathbf{1}^TA_3=\mathbf{1}^TA_4=\frac{l}{2}\mathbf{1}^T$. From $B_i=J-A_i$ for $i=1,\dots 4$, we have
	\begin{align}\label{temp:lemma4:1}
	B_1B_3^T-A_1A_3^T&=(J-A_1)(J-A_3)^T-A_1A_3^T=JJ^T-A_1J^T-JA_3^T,\\\label{temp:lemma4:2}
	B_1^TB_2-A_1^TA_2&=(J-A_1)^T(J-A_2)-A_1^TA_2=J^TJ-A_1^TJ-J^TA_2=0,\\\label{temp:lemma4:3}
	B_3^TB_4-A_3^TA_4&=(J-A_3)^T(J-A_4)-A_3^TA_4=J^TJ-A_3^TJ-J^TA_4=0.
	\end{align}
	Furthermore, using Lemma \ref{Lemma:Grammates A+B=J} for each pair $(A_i,B_i)$ for $i=1,\dots,4$, we obtain $A_i^TA_i=B_i^TB_i$. 
	
	One can check that
	$AA^T=BB^T$ if and only if
	\begin{align*}
	A_1A_1^T+A_2A_2^T+XX^T&=B_1B_1^T+B_2B_2^T+XX^T,\\
	A_1A_3^T+A_2Y^T+XA_4^T&=B_1B_3^T+B_2Y^T+XB_4^T,\\
	A_3A_3^T+A_4A_4^T+YY^T&=B_3B_3^T+B_4B_4^T+YY^T.
	\end{align*}
	Using \eqref{temp:lemma4:1}, $A_1A_1^T+A_2A_2^T=B_1B_1^T+B_2B_2^T$ and $A_3A_3^T+A_4A_4^T=B_3B_3^T+B_4B_4^T$, we find that $AA^T=BB^T$ if and only if $(B_2-A_2)Y^T+X(B_4^T-A_4^T)=A_1J^T+JA_3^T-JJ^T$.
	
	One can verify that $A^TA=B^TB$ if and only if
	\begin{align*}
	&A_1^TA_1+A_3^TA_3=B_1^TB_1+B_3^TB_3,\;A_2^TA_2+Y^TY=B_2^TB_2+Y^TY, \\
	&A_4^TA_4+X^TX=B_4^TB_4+X^TX,\; A_1^TA_2+A_3^TY=B_1^TB_2+B_3^TY, \\
	&A_1^TX+A_3^TA_4=B_1^TX+B_3^TB_4,\; A_2^TX+Y^TA_4=B_2^TX+Y^TB_4.
	\end{align*}
	By \eqref{temp:lemma4:2}, \eqref{temp:lemma4:3} and the fact that $A_i^TA_i=B_i^TB_i$ for $i=1,\dots,4$, the desired conclusion follows.
\end{proof}

\begin{remark}\label{Remark:types M2 and M3}
	Let $C$ and $D$ be Gram mates, and let $\alpha$ and $\beta$ be sets of some row and column indices, respectively. Then, $C[\alpha,\beta]$ and $D[\alpha,\beta]$ are not necessarily Gram mates. That is, the submatrices do not necessarily inherit properties of being Gram mates from $C$ and $D$. However, the matrices with the hypothesis in Lemma \ref{Lemma:supporting theorem of rank 21} yield the following submatrices with inherited properties in the matrices.
	
	Let $A=\begin{bmatrix}
	A_1 & A_2\\
	A_3 & Y
	\end{bmatrix}$ and $B=\begin{bmatrix}
	B_1 & B_2\\
	B_3 & Y 
	\end{bmatrix}$. Suppose that $(A-B)\mathbf{1}=0$, and $\mathbf{1}^TA_i=\mathbf{1}^TB_i$, $A_i+B_i=J$ for $i=1,2,3$. It can be found from Definition \ref{Def:GramMates} that $A$ and $B$ are Gram mates if and only if $(B_2-A_2)Y^T=A_1J^T+JA_3^T-JJ^T$ and $(B_3^T-A_3^T)Y=0$. Then, the equivalent condition for Gram mates $A$ and $B$ is the same as that obtained from the conditions \eqref{rank2:conditions1}--\eqref{rank2:conditions4} in Lemma \ref{Lemma:supporting theorem of rank 21} by removing the terms containing $A_4$, $B_4$ or $X$. 
	
	Similarly, given $A=\begin{bmatrix}
	A_2 & X\\
	Y & A_4
	\end{bmatrix}$ and $B=\begin{bmatrix}
	B_2 & X\\
	Y & B_4
	\end{bmatrix}$ where $A_i\mathbf{1}=B_i\mathbf{1}$, $\mathbf{1}^TA_i=\mathbf{1}^TB_i$ and $A_i+B_i=J$ for $i=2,4$, we can find from Definition \ref{Def:GramMates} that $A$ and $B$ are Gram mates if and only if $(B_2-A_2)Y^T+X(B_4^T-A_4^T)=0$ and $(B_2^T-A_2^T)X+Y^T(B_4-A_4)=0$. Then, the equivalent condition for Gram mates $A$ and $B$ can be also obtained by annihilating the terms having $A_1$, $A_3$, $B_1$ or $B_3$ in the conditions \eqref{rank2:conditions1}--\eqref{rank2:conditions4} in Lemma \ref{Lemma:supporting theorem of rank 21}.
	
	Therefore, equivalent conditions for Gram mates via matrices of forms \ref{index m2} and \ref{index m3} can be induced by those for Gram mates via matrices of form \ref{index m4}.
\end{remark}

\begin{theorem}\label{Thm:Gram mates rank 21}
	Let $$E=\begin{bmatrix}
	J_{k,a} & J_{k,b} & -J_{k,c} & -J_{k,d} & J_{k,e} & -J_{k,f} & 0 & 0 \\
	-J_{k,a} & -J_{k,b} & J_{k,c} & J_{k,d} & -J_{k,e} & J_{k,f} & 0 & 0 \\
	J_{l,a} & -J_{l,b} & J_{l,c} & -J_{l,d} & 0 & 0 & J_{l,g} & -J_{l,h} \\
	-J_{l,a} & J_{l,b} & -J_{l,c} & J_{l,d} & 0 & 0 & -J_{l,g} & J_{l,h} \\
	\end{bmatrix}$$
	where any column index in each block is a nonnegative integer, and $k,l>0$, $a+b+c+d>0$, $e+f>0$, $g+h>0$, $E\mathbf{1}=0$ and $\mathbf{1}^TE=0^T$. Let 
	\begin{align}\label{matrix A regarding M3}
	A=\begin{bmatrix}
	0 & 0 & J_{k,c} & J_{k,d} & 0 & J_{k,f} & X_{11} & X_{12} \\
	J_{k,a} & J_{k,b} & 0 & 0 & J_{k,e} & 0 & X_{21} & X_{22} \\
	0 & J_{l,b} & 0 & J_{l,d} & Y_{11} & Y_{12} & 0 & J_{l,h} \\
	J_{l,a} & 0 & J_{l,c} & 0 & Y_{21} & Y_{22} & J_{l,g} & 0 \\
	\end{bmatrix}
	\end{align}
	where each block of $A$ is a $(0,1)$ matrix. Then, $A$ and $A+E$ are Gram mates if and only if the following conditions are satisfied: 
	\begin{enumerate}[label=(\roman*)]
		\item \label{tempc1} $\mathbf{1}^TX_{1i}=\mathbf{1}^TX_{2i}$ and $\mathbf{1}^TY_{1i}=\mathbf{1}^TY_{2i}$ for $i=1,2$;
		\item \label{tempc2}$\begin{bmatrix}
		X_{11} & X_{12}\\
		X_{21} & X_{22}
		\end{bmatrix}\begin{bmatrix}
		\mathbf{1}\\
		-\mathbf{1}
		\end{bmatrix}=\frac{g-h}{2}\begin{bmatrix}
		\mathbf{1}\\
		\mathbf{1}
		\end{bmatrix}$ and $\begin{bmatrix}
		Y_{11} & Y_{12}\\
		Y_{21} & Y_{22}
		\end{bmatrix}\begin{bmatrix}
		\mathbf{1}\\
		-\mathbf{1}
		\end{bmatrix}=\frac{e-f}{2}\begin{bmatrix}
		\mathbf{1}\\
		\mathbf{1}
		\end{bmatrix}$ where $g-h$ and $e-f$ are even.
	\end{enumerate}
\end{theorem}
\begin{proof}
	Let $A_1=\begin{bmatrix}
	0 & 0 & J_{k,c} & J_{k,d} \\
	J_{k,a} & J_{k,b} & 0 & 0 
	\end{bmatrix}$, $A_3=\begin{bmatrix}
	0 & J_{l,b} & 0 & J_{l,d} \\
	J_{l,a} & 0 & J_{l,c} & 0
	\end{bmatrix}$, $A_2=\begin{bmatrix}
	0 & J_{k,f}\\
	J_{k,e} & 0
	\end{bmatrix}$ and $A_4=\begin{bmatrix}
	0 & J_{l,h}\\
	J_{l,g} & 0
	\end{bmatrix}$.	Set $X=\begin{bmatrix}
	X_{11} & X_{12}\\
	X_{21} & X_{22}
	\end{bmatrix}$, $Y=\begin{bmatrix}
	Y_{11} & Y_{12}\\
	Y_{21} & Y_{22}
	\end{bmatrix}$ and $B_i=J-A_i$ for $i=1,\dots,4$. 
	
	Applying Lemma \ref{Lemma:supporting theorem of rank 21} to our setup, it is enough to show that our desired conditions \ref{tempc1} and \ref{tempc2} are deduced from the conditions \eqref{rank2:conditions1}--\eqref{rank2:conditions4}. From the condition \eqref{rank2:conditions2}, which is $(B_1^T-A_1^T)X=0$, we have $$\begin{bmatrix}
	J_{k,a} & J_{k,b} & -J_{k,c} & -J_{k,a} \\
	-J_{k,a} & -J_{k,b} & J_{k,c} & J_{k,d} 
	\end{bmatrix}^T\begin{bmatrix}
	X_{11} & X_{12}\\
	X_{21} & X_{22}
	\end{bmatrix}=0.$$ 
	Since $a+b+c+d>0$, there is at least one row in $B_1^T-A_1^T$ that is $\begin{bmatrix}
	\mathbf{1}^T & -\mathbf{1}^T
	\end{bmatrix}$ or $\begin{bmatrix}
	-\mathbf{1}^T & \mathbf{1}^T
	\end{bmatrix}$. Hence, $\mathbf{1}^TX_{1i}=\mathbf{1}^TX_{2i}$ for $i=1,2$. Conversely, $\mathbf{1}^TX_{1i}=\mathbf{1}^TX_{2i}$ for $i=1,2$ implies $(B_1^T-A_1^T)X=0$. Similarly, we can find that the condition \eqref{rank2:conditions3} is equivalent to $\mathbf{1}^TY_{1i}=\mathbf{1}^TY_{2i}$ for $i=1,2$. Since $k,l>0$, each of $B_2^T-A_2^T$ and $B_4^T-A_4^T$ consists of rows that are $\pm\begin{bmatrix}
	\mathbf{1}^T & -\mathbf{1}^T
	\end{bmatrix}$. Thus, $\mathbf{1}^TX_{1i}=\mathbf{1}^TX_{2i}$ and $\mathbf{1}^TY_{1i}=\mathbf{1}^TY_{2i}$ for $i=1,2$ imply \eqref{rank2:conditions4}. 
	
	Consider the condition \eqref{rank2:conditions1}. Then, it can be checked that \eqref{rank2:conditions1} is equivalent to 
	\begin{align*}
	\begingroup % keep the change local
	\setlength\arraycolsep{1.5pt}
	\begin{bmatrix}
	J_{k,e} & -J_{k,f}\\
	-J_{k,e} & J_{k,f}
	\end{bmatrix}\begin{bmatrix}
	Y_{11}^T & Y_{21}^T\\
	Y_{12}^T & Y_{22}^T
	\end{bmatrix}+\begin{bmatrix}
	X_{11} & X_{12}\\
	X_{21} & X_{22}
	\end{bmatrix}\begin{bmatrix}
	J_{g,l} & -J_{g,l}\\
	-J_{h,l} & J_{h,l}
	\end{bmatrix}=\begin{bmatrix}
	(d-a)J_{k,l} & (c-b)J_{k,l}\\
	(b-c)J_{k,l} & (a-d)J_{k,l}
	\end{bmatrix}.
	\endgroup
	\end{align*}
	It follows from Lemma \ref{Lemma:JY+XJ} that the condition \eqref{rank2:conditions1} is equivalent to $Y_{11}\mathbf{1}_e-Y_{12}\mathbf{1}_f=y_1\mathbf{1}_l$, $X_{11}\mathbf{1}_g-X_{12}\mathbf{1}_h=x_1\mathbf{1}_k$ and $Y_{21}\mathbf{1}_e-Y_{22}\mathbf{1}_f=y_2\mathbf{1}_l$, $X_{21}\mathbf{1}_g-X_{22}\mathbf{1}_h=x_2\mathbf{1}_k$ where $x_1+y_1=x_2+y_2=d-a$ and $y_2-x_1=y_1-x_2=c-b$. Furthermore, pre-multiplying both sides of $Y_{11}\mathbf{1}_e-Y_{12}\mathbf{1}_f=y_1\mathbf{1}_l$ and $Y_{21}\mathbf{1}_e-Y_{22}\mathbf{1}_f=y_2\mathbf{1}_l$ by $\mathbf{1}_l^T$, respectively, we have 
	\begin{align*}
	\mathbf{1}_l^TY_{11}\mathbf{1}_e-\mathbf{1}_l^TY_{12}\mathbf{1}_f=y_1\mathbf{1}_l^T\mathbf{1}_l=y_1l\;\text{and}\;
	\mathbf{1}_l^TY_{21}\mathbf{1}_e-\mathbf{1}_l^TY_{22}\mathbf{1}_f=y_2\mathbf{1}_l^T\mathbf{1}_l=y_2l.
	\end{align*}
	Since $\mathbf{1}_l^TY_{1i}=\mathbf{1}_l^TY_{2i}$ for $i=1,2$, from subtraction of the two equations, we obtain $(y_1-y_2)l=0$ and so $y_1=y_2$. Similarly, using $X_{11}\mathbf{1}_g-X_{12}\mathbf{1}_h=x_1\mathbf{1}_k$ and $X_{21}\mathbf{1}_g-X_{22}\mathbf{1}_h=x_2\mathbf{1}_k$, it can be checked that $x_1=x_2$. It follows from $E\mathbf{1}=0$ that $x_1=\frac{-a+b-c+d}{2}=\frac{g-h}{2}$ and $y_1=\frac{-a-b+c+d}{2}=\frac{e-f}{2}$. Furthermore, since $x_1$ and $y_1$ are integers, $g-h$ and $e-f$ must be even.	
\end{proof}

\begin{corollary}\label{Cor:equivalent condition for realizable 21}
	Let $E$ be a $(0,1,-1)$ matrix of form \ref{index m4}. Then, $E$ is realizable if and only if $g-h$ and $e-f$ are even.
\end{corollary}
\begin{proof}
	Suppose that $E$ is realizable. By Theorem \ref{Thm:Gram mates rank 21}, the conclusion is straightforward. Conversely, let $g-h$ and $e-f$ be even. We only need to show the existence of  $X=\begin{bmatrix}
	X_{11} & X_{12}\\
	X_{21} & X_{22}
	\end{bmatrix}$ and $Y=\begin{bmatrix}
	Y_{11} & Y_{12}\\
	Y_{21} & Y_{22}
	\end{bmatrix}$ satisfying \ref{tempc1} and \ref{tempc2} of Theorem \ref{Thm:Gram mates rank 21}. If $g-h=0$, then our desired matrix can be obtained by choosing $X=0$. We can also obtain $X$ with \ref{tempc1} and \ref{tempc2} by choosing $X_{11}=X_{21}=\begin{bmatrix}
	J_{k,\frac{g-h}{2}} & 0
	\end{bmatrix}$ and $X_{12}=X_{22}=0$ if $g-h>0$, and choosing $X_{11}=X_{21}=0$ and $X_{12}=X_{22}=\begin{bmatrix}
	J_{k,\frac{h-g}{2}} & 0
	\end{bmatrix}$ if $g-h<0$. In this manner, we can also construct $Y$ satisfying \ref{tempc1} and \ref{tempc2}.
\end{proof}

\begin{remark}\label{Remark:Grammates M1,M2,M3}
	Consider a $(0,1,-1)$ matrix $E$ of form \ref{index m3}. Then, $E\mathbf{1}=0$ implies $c-b=d-a=\frac{e-f}{2}$. The matrix $E$ can be regarded as a matrix of form \ref{index m4} by setting $g=0$ and $h=0$. From Remark \ref{Remark:types M2 and M3}, equivalent conditions for Gram mates via $E$ can be induced by those for Gram mates via matrices of form \ref{index m4}, annihilating the conditions related to $X$ in \ref{tempc1} and \ref{tempc2} of Theorem \ref{Thm:Gram mates rank 21}. Hence, $A$ and $A+E$ are Gram mates where $A$ is in form \eqref{matrix A regarding M3} with $g=h=0$ if and only if $\mathbf{1}^TY_{1i}=\mathbf{1}^TY_{2i}$ for $i=1,2$ and $\begin{bmatrix}
	Y_{11} & Y_{12}\\
	Y_{21} & Y_{22}
	\end{bmatrix}\begin{bmatrix}
	\mathbf{1}\\
	-\mathbf{1}
	\end{bmatrix}=\frac{e-f}{2}\begin{bmatrix}
	\mathbf{1}\\
	\mathbf{1}
	\end{bmatrix}$ where $e-f$ is even. Furthermore, $E$ is realizable if and only if $e-f$ is even.
	
	Let $E$ be a $(0,1,-1)$ matrix of form \ref{index m2}. Then, $E\mathbf{1}=0$ implies $e=f>0$ and $g=h>0$. By an analogous argument with Remark \ref{Remark:types M2 and M3}, we can find that $A$ and $A+E$ are Gram mates where $A$ is in form \eqref{matrix A regarding M3} with $a=b=c=d=0$ if and only if $\begin{bmatrix}
	X_{11} & X_{12}\\
	X_{21} & X_{22}
	\end{bmatrix}\begin{bmatrix}
	\mathbf{1}\\
	-\mathbf{1}
	\end{bmatrix}=0$, $\begin{bmatrix}
	Y_{11} & Y_{12}\\
	Y_{21} & Y_{22}
	\end{bmatrix}\begin{bmatrix}
	\mathbf{1}\\
	-\mathbf{1}
	\end{bmatrix}=0$, $\mathbf{1}^TX_{1i}=\mathbf{1}^TX_{2i}$ and $\mathbf{1}^TY_{1i}=\mathbf{1}^TY_{2i}$ for $i=1,2$. Since $e-f=g-h=0$, $E$ is realizable for any $e,f,g,h>0$.
	
	Finally, from Remark \ref{Remark:Grammates M1}, there is only one pair of Gram mates via $E$ of form \ref{index m1}, and so $E$ is realizable. Thus, one may consider $E$ as a matrix of form \ref{index m4} by setting $e=f=g=h=0$. 
\end{remark}

\begin{remark}\label{Remark: Gram mates from wtE to E}
	Let $E=\begin{bmatrix}
	\wtE & 0\\
	0 & 0
	\end{bmatrix}$ be a realizable matrix where $\wtE$ is of one of types \ref{index m1}--\ref{index m4}, and let $\wtA$ and $\wtE$ be Gram mates. Suppose that $A=\begin{bmatrix}
	\wtA & X_1\\
	X_2 & X_3
	\end{bmatrix}$ is a $(0,1)$ matrix compatible with the partition of $E$. By Proposition \ref{Prop:Grammates for E from tilde of E} and Remark \ref{Remark:left right null vectors for E}, completely determining Gram mates via $E$ is equivalent to finding all left and right $(0,1)$ null vectors of $\wtE$ for $X_1$ and $X_2$. 
\end{remark}

Let $E$ be a matrix of form \ref{index m4}. Suppose that $A$ and $A+E$ are Gram mates where $A$ is in the form \eqref{matrix A regarding M3}. Then,
\begin{align}
2A+E=\begin{bmatrix}\label{temp122}
J_{k,a} & J_{k,b} & J_{k,c} & J_{k,d} & J_{k,e} & J_{k,f} & 2X_{11} & 2X_{12} \\
J_{k,a} & J_{k,b} & J_{k,c} & J_{k,d} & J_{k,e} & J_{k,f} & 2X_{21} & 2X_{22} \\
J_{l,a} & J_{l,b} & J_{l,c} & J_{l,d} & 2Y_{11} & 2Y_{12} & J_{l,g} & J_{l,h} \\
J_{l,a} & J_{l,b} & J_{l,c} & J_{l,d} & 2Y_{21} & 2Y_{22} & J_{l,g} & J_{l,h} \\
\end{bmatrix}
\end{align}
where $e+f$ and $g+h$ are positive even numbers, $k,l>0$ and $a+b+c+d>0$. Consider two vectors $\bx_1^T$ and $\bx_2^T$ that form a basis of $\mathrm{Row}(E)$:
\begin{align}\label{temp123}
\begin{split}
\bx_1^T&=\left[\begin{array}{cccccccc}
\mathbf{1}_a^T & \mathbf{1}_b^T & -\mathbf{1}_c^T & -\mathbf{1}_{d}^T & \mathbf{1}_e^T & -\mathbf{1}_f^T & \mathbf{0}_{g}^T & \mathbf{0}_h^T\end{array}\right],\\
\bx_2^T&=\left[\begin{array}{cccccccc}\mathbf{1}_a^T & -\mathbf{1}_b^T & \mathbf{1}_c^T & -\mathbf{1}_{d}^T & \mathbf{0}_e^T & \mathbf{0}_f^T & \mathbf{1}_{g}^T & -\mathbf{1}_h^T
\end{array}\right].
\end{split}
\end{align}
From $E\mathbf{1}=0$, we have $a+b+e=c+d+f$ and $a+c+g=b+d+h$. By Theorem \ref{Thm:Gram mates rank 21}, $2(X_{i1}\mathbf{1}_g-X_{i2}\mathbf{1}_h)=g-h$ and $2(Y_{i1}\mathbf{1}_e-Y_{i2}\mathbf{1}_f)=e-f$ for $i=1,2$. It follows from a computation that $(2A+E)\bx_j=0$ for $j=1,2$. Since $\mathrm{Row}(E)=\mathrm{span}\{\bx_1,\bx_2\}$, we obtain $(2A+E)E^T=0$. By Theorem \ref{Thm:equivalent conditions}, $A+E$ is obtained from $A$ by changing the signs of $2$ positive singular values. Hence, $A$ is convertible to $A+E$. Furthermore, for $E'$ in each of types \ref{index m1}, \ref{index m2} and \ref{index m3}, we may consider $e=f=g=h=0$, $a=b=c=d=0$ and $g=h=0$, respectively, in \eqref{temp122} and \eqref{temp123}. One can readily check for each case we have $(2A'+E')(E')^T=0$ where $A'$ is a Gram mate to $A'+E'$. Therefore, for any pair of Gram mates $A$ and $A+E$ via $E$ in any form among \ref{index m1}--\ref{index m4}, $A$ and $A+E$ are convertible Gram mates. 

We now consider right singular vectors $\bv_1$ and $\bv_2$ corresponding to the Gram singular values of $A$ and $A+E$. Since $\mathrm{Row}(E)=\mathrm{span}\{\bx_1,\bx_2\}$, we have $\mathrm{span}\{\bx_1,\bx_2\}=\mathrm{span}\{\bv_1,\bv_2\}$. Using $a+b-c-d+e-f=0$, $a-b+c-d+g-h=0$ and the conditions \ref{tempc1} and \ref{tempc2} in Theorem \ref{Thm:Gram mates rank 21}, one can verify from computations that
\begin{align*}
A^TA\bx_1&=\begin{bmatrix}
(k(a+b+e)+\frac{1}{2}l(a-b-c+d))\mathbf{1}_a\\
(k(a+b+e)-\frac{1}{2}l(a-b-c+d))\mathbf{1}_b\\
(-k(a+b+e)+\frac{1}{2}l(a-b-c+d))\mathbf{1}_c\\
(-k(a+b+e)-\frac{1}{2}l(a-b-c+d))\mathbf{1}_d\\
k(a+b+e)\mathbf{1}_e \\
-k(a+b+e)\mathbf{1}_f \\
\frac{1}{2}l(a-b-c+d)\mathbf{1}_g\\
-\frac{1}{2}l(a-b-c+d)\mathbf{1}_h
\end{bmatrix},\\
A^TA\bx_2&=\begin{bmatrix}
(\frac{1}{2}k(a-b-c+d)+l(a+c+g))\mathbf{1}_a\\
(\frac{1}{2}k(a-b-c+d)-l(a+c+g))\mathbf{1}_b\\
(-\frac{1}{2}k(a-b-c+d)+l(a+c+g))\mathbf{1}_c\\
(-\frac{1}{2}k(a-b-c+d)-l(a+c+g))\mathbf{1}_d\\
\frac{1}{2}k(a-b-c+d)\mathbf{1}_e\\
-\frac{1}{2}k(a-b-c+d)\mathbf{1}_f\\
l(a+c+g)\mathbf{1}_g\\
-l(a+c+g)\mathbf{1}_h
\end{bmatrix}.
\end{align*}
Consider an equation $A^TA(\zeta_1\bx_1+\zeta_2\bx_2)=\lambda(\zeta_1\bx_1+\zeta_2\bx_2)$ where $\zeta_1$, $\zeta_2$ and $\lambda$ are real numbers. Set $a_1=k(a+b+e)$, $a_2=k(a-b-c+d)$, $a_3=l(a+c+g)$ and $a_4=l(a-b-c+d)$. Then,
\begin{align*}
A^TA(\zeta_1\bx_1+\zeta_2\bx_2)=\begin{bmatrix}
(\zeta_1(a_1+\frac{1}{2}a_4)+\zeta_2(\frac{1}{2}a_2+a_3))\mathbf{1}_a\\
(\zeta_1(a_1-\frac{1}{2}a_4)+\zeta_2(\frac{1}{2}a_2-a_3))\mathbf{1}_b\\
(\zeta_1(-a_1+\frac{1}{2}a_4)+\zeta_2(-\frac{1}{2}a_2+a_3))\mathbf{1}_b\\
(\zeta_1(-a_1-\frac{1}{2}a_4)+\zeta_2(-\frac{1}{2}a_2-a_3))\mathbf{1}_d\\
(\zeta_1a_1+\frac{1}{2}\zeta_2a_2)\mathbf{1}_e\\
(-\zeta_1a_1-\frac{1}{2}\zeta_2a_2)\mathbf{1}_f\\
(\frac{1}{2}\zeta_1a_4+\zeta_2a_3)\mathbf{1}_g\\
(-\frac{1}{2}\zeta_1a_4-\zeta_2a_3)\mathbf{1}_h
\end{bmatrix}=\lambda\begin{bmatrix}
(\zeta_1+\zeta_2)\mathbf{1}_a\\
(\zeta_1-\zeta_2)\mathbf{1}_b\\
(-\zeta_1+\zeta_2)\mathbf{1}_b\\
(-\zeta_1-\zeta_2)\mathbf{1}_d\\
\zeta_1\mathbf{1}_e\\
-\zeta_1\mathbf{1}_f\\
\zeta_2\mathbf{1}_g\\
-\zeta_2\mathbf{1}_h
\end{bmatrix}.
\end{align*}
This implies that $\zeta_1\bx_1+\zeta_2\bx_2$ is an eigenvector of $A^TA$ corresponding to an eigenvalue $\lambda$ if and only if $(\zeta_1,\zeta_2)$ is a solution to the system of equations $a_1\zeta_1+\frac{1}{2}a_2\zeta_2=\lambda\zeta_1$ and $\frac{1}{2}a_4\zeta_1+a_3\zeta_2=\lambda\zeta_2$. Hence, for an eigenvector $(\zeta_1,\zeta_2)$ of the matrix $\begin{bmatrix}
a_1 & \frac{1}{2}a_2\\
\frac{1}{2}a_4 & a_3
\end{bmatrix}$ associated to an eigenvalue $\lambda$, a normalized vector of $\zeta_1\bx_1+\zeta_2\bx_2$ is a right singular vector corresponding to a Gram singular value $\sqrt{\lambda}$ of $A$. Furthermore, for $E'$ in each of types \ref{index m1}, \ref{index m2} and \ref{index m3}, considering the extra conditions $e=f=g=h=0$, $a=b=c=d=0$ and $g=h=0$, respectively, one can find that analogous results are established for right singular vectors associated to the Gram singular values of Gram mates via $E'$. Therefore, we have the following result.

\begin{theorem}\label{Thm:Gramsingular rank21}
	Let $E$ be a realizable matrix of rank $2$ corresponding to one of forms \ref{index m1}--\ref{index m4}. Let 
	$$M=\begin{bmatrix}
	k(a+b+e) & \frac{1}{2}k(a-b-c+d)\\
	\frac{1}{2}l(a-b-c+d) & l(a+c+g)
	\end{bmatrix}$$
	where the entries in $M$ correspond to the sub-indices in \ref{index m1}--\ref{index m4}. Suppose that $A$ and $A+E$ are Gram mates via $E$. Then, $A$ and $A+E$ are convertible, and their Gram singular values are the square roots of the eigenvalues $\lambda$ of $M$. Furthermore, a right singular vector corresponding to $\sqrt{\lambda}$ is a normalized vector of $\zeta_1\bx_1+\zeta_2\bx_2$, where $\bx_1$ and $\bx_2$ are the vectors in \eqref{temp123} and $(\zeta_1,\zeta_2)$ is an eigenvector of $M$ associated to $\lambda$.
\end{theorem}

\subsubsection{Realizable matrices in the form \ref{index m5}}
We now investigate Gram mates via matrices of the form \ref{index m5}. Furthermore, we shall show that there exist non-convertible Gram mates via matrices of the form \ref{index m5}, while any Gram mates via matrices in any form among \ref{index m1}--\ref{index m4} have Gram singular values.

\begin{lemma}\label{Lemma:supporting Grammates rank 22}
	Let $$A=\begin{bmatrix}
	A_1 & A_2 & X\\
	A_3 & Y & A_4\\
	Z & A_5 & A_6
	\end{bmatrix},\; B=\begin{bmatrix}
	B_1 & B_2 & X\\
	B_3 & Y & B_4\\
	Z & B_5 & B_6
	\end{bmatrix}$$ be $(0,1)$ matrices such that $(A-B)\mathbf{1}=0$, $\mathbf{1}^T(A-B)=0^T$, and $A_i+B_i=J$ for $i=1,\dots 6$. Then, $AA^T=BB^T$  if and only if 
	\begin{align*}
	&(B_2-A_2)Y^T+X(B_4^T-A_4^T)=A_1A_3^T-B_1B_3^T,\\
	&(B_1-A_1)Z^T+X(B_6^T-A_6^T)=A_2A_5^T-B_2B_5^T,\\
	&(B_3-A_3)Z^T+Y(B_5^T-A_5^T)=A_4A_6^T-B_4B_6^T.
	\end{align*}
	Furthermore, $A^TA=B^TB$ if and only if
	\begin{align*}
	&(B_2^T-A_2^T)X+Y^T(B_4-A_4)=A_5^TA_6-B_5^TB_6,\\
	&(B_1^T-A_1^T)X+Z^T(B_6-A_6)=A_3^TA_4-B_3^TB_4,\\
	&(B_3^T-A_3^T)Y+Z^T(B_5-A_5)=A_1^TA_2-B_1^TB_2.
	\end{align*}
\end{lemma}
\begin{proof}
	From $(A-B)\mathbf{1}=0$, we have $\begin{bmatrix}
	A_1 & A_2
	\end{bmatrix}\mathbf{1}=\begin{bmatrix}
	B_1 & B_2
	\end{bmatrix}\mathbf{1}$, $\begin{bmatrix}
	A_3 & A_4
	\end{bmatrix}\mathbf{1}=\begin{bmatrix}
	B_3 & B_4
	\end{bmatrix}\mathbf{1}$, and $\begin{bmatrix}
	A_5 & A_6
	\end{bmatrix}\mathbf{1}=\begin{bmatrix}
	B_5 & B_6
	\end{bmatrix}\mathbf{1}$. Moreover, $A_i+B_i=J$ for $i=1,\dots 6$. By Lemma \ref{Lemma:Grammates A+B=J}, $A_1A_1^T+A_2A_2^T=B_1B_1^T+B_2B_2^T$, $A_3A_3^T+A_4A_4^T=B_3B_3^T+B_4B_4^T$ and $A_5A_5^T+A_6A_6^T=B_5B_5^T+B_6B_6^T$. Similarly, one can deduce from $\mathbf{1}^T(A-B)=0^T$ that $A_1^TA_1+A_3^TA_3=B_1^TB_1+B_3^TB_3$, $A_2^TA_2+A_5^TA_5T=B_2^TB_2+B_5^TB_5$ and $A_4^TA_4+A_6^TA_6=B_4^TB_4+B_6^TB_6$. Considering $AA^T=BB^T$ and $A^TA=B^TB$, the desired conclusion is straightforward.
\end{proof}
\begin{remark}\label{Remark:sizes for M5}
	Continuing with the notation and hypothesis of Lemma \ref{Lemma:supporting Grammates rank 22}, consider $A$ that is a $3\times 3$ block partitioned matrix. Let $n_i$ be the number of columns in the $i^\text{th}$ column partition of $A$ for $i=1,2,3$. Then, $(A-B)\mathbf{1}=0$ and $A_i+B_i=J$ imply that $n_1+n_2$, $n_1+n_3$ and $n_2+n_3$ are even. Hence, $n_i$ is either even for $i=1,2,3$ or odd for $i=1,2,3$. Similarly, the number of rows in each row partition of $A$ has the same parity.
\end{remark}

\begin{theorem}\label{Thm:Grammates Rank22}
	Let $$E=\begin{bmatrix}
	J_{k,a} & -J_{k,b}  & J_{k,c} & -J_{k,d} & 0 & 0 \\
	-J_{l,a} & J_{l,b} & -J_{l,c} & J_{l,d} & 0 & 0 \\
	J_{p,a} & -J_{p,b}  & 0 & 0 & J_{p,e} & -J_{p,f} \\
	-J_{q,a} & J_{q,b} & 0 & 0 & -J_{q,e} & J_{q,f} \\
	0 & 0 & J_{r,c} & -J_{r,d} & -J_{r,e} & J_{r,f} \\
	0 & 0 & -J_{s,c} & J_{s,d} & J_{s,e} & -J_{s,f} \\
	\end{bmatrix}$$ where $E\mathbf{1}=0$, $\mathbf{1}^TE=0^T$, and $a+b$, $c+d$, $e+f$, $k+l$, $p+q$ and $r+s$ are positive. Let
	\begin{align}\label{matrix A rank 22}
	A=\begin{bmatrix}
	0 & J_{k,b}  & 0 & J_{k,d} & X_{11} & X_{12} \\
	J_{l,a} & 0 & J_{l,c} & 0 & X_{21} & X_{22} \\
	0 & J_{p,b}  & Y_{11} & Y_{12} & 0 & J_{p,f} \\
	J_{q,a} & 0 & Y_{21} & Y_{22} & J_{q,e} & 0 \\
	Z_{11} & Z_{12} & 0 & J_{r,d} & J_{r,e} & 0  \\
	Z_{21} & Z_{22} & J_{s,c} & 0 & 0 & J_{s,f} \\
	\end{bmatrix}
	\end{align}
	be a $(0,1)$ matrix conformally partitioned with $E$. Suppose that $X=\begin{bmatrix}
	X_{11} & X_{12}\\
	X_{21} & X_{22}
	\end{bmatrix}$, $Y=\begin{bmatrix}
	Y_{11} & Y_{12}\\
	Y_{21} & Y_{22}
	\end{bmatrix}$ and $Z=\begin{bmatrix}
	Z_{11} & Z_{12}\\
	Z_{21} & Z_{22}
	\end{bmatrix}$. Then, $A$ and $A+E$ are Gram mates if and only if $X$, $Y$ and $Z$ satisfy the following conditions: 
	\begin{enumerate}[label=(\roman*)]
		\item\label{rank22 c1} there are integers $x_1,x_2,y_1,y_2,z_1$ and $z_2$ such that $X\begin{bmatrix}
		\mathbf{1}\\
		-\mathbf{1}
		\end{bmatrix}=\begin{bmatrix}
		x_1\mathbf{1}\\
		x_2\mathbf{1}
		\end{bmatrix}$, $Y\begin{bmatrix}
		\mathbf{1}\\
		-\mathbf{1}
		\end{bmatrix}=\begin{bmatrix}
		y_1\mathbf{1}\\
		y_2\mathbf{1}
		\end{bmatrix}$, $Z\begin{bmatrix}
		\mathbf{1}\\
		-\mathbf{1}
		\end{bmatrix}=\begin{bmatrix}
		z_1\mathbf{1}\\
		z_2\mathbf{1}
		\end{bmatrix}$ and $x_1=y_2=-z_2$, $x_2=y_1=-z_1$, $x_1+x_2=y_1+y_2=-(z_1+z_2)=e-f$; and
		\item\label{rank22 c2} there are integers $\alpha_1,\alpha_2,\beta_1,\beta_2,\gamma_1$ and $\gamma_2$ such that $X^T\begin{bmatrix}
		\mathbf{1}\\
		-\mathbf{1}
		\end{bmatrix}=\begin{bmatrix}
		\alpha_1\mathbf{1}\\
		\alpha_2\mathbf{1}
		\end{bmatrix}$, $Y^T\begin{bmatrix}
		\mathbf{1}\\
		-\mathbf{1}
		\end{bmatrix}=\begin{bmatrix}
		\beta_1\mathbf{1}\\
		\beta_2\mathbf{1}
		\end{bmatrix}$, $Z^T\begin{bmatrix}
		\mathbf{1}\\
		-\mathbf{1}
		\end{bmatrix}=\begin{bmatrix}
		\gamma_1\mathbf{1}\\
		\gamma_2\mathbf{1}
		\end{bmatrix}$ and $\gamma_1=\beta_2=-\alpha_2$, $\gamma_2=\beta_1=-\alpha_1$, $\gamma_1+\gamma_2=\beta_1+\beta_2=-(\alpha_1+\alpha_2)=l-k$.
	\end{enumerate}
	In particular, \ref{rank22 c1} and \ref{rank22 c2} imply \begin{enumerate*}[label=(\alph*)]
		\item\label{rank22 c3} $e\alpha_1-f\alpha_2=kx_1-lx_2$, $c\beta_1-d\beta_2=py_1-qy_2$, and $a\gamma_1-b\gamma_2=rz_1-sz_2$.
	\end{enumerate*}
\end{theorem}
\begin{proof}	
	Let $A=\begin{bmatrix}
	A_1 & A_2 & X\\
	A_3 & Y & A_4\\
	Z & A_5 & A_6
	\end{bmatrix}$ and $E=\begin{bmatrix}
	E_1 & E_2 & 0\\
	E_3 & 0 & E_4\\
	0 & E_5 & E_6
	\end{bmatrix}$. Set $B=A+E$. By Lemma \ref{Lemma:supporting Grammates rank 22}, $AA^T=BB^T$ if and only if
	\begin{align}\label{temp:eqn1}
	&\begin{bmatrix}
	J_{k,c} & -J_{k,d} \\
	-J_{l,c} & J_{l,d}
	\end{bmatrix}Y^T+X\begin{bmatrix}
	J_{p,e} & -J_{p,f} \\
	-J_{q,e} & J_{q,f}
	\end{bmatrix}^T=\begin{bmatrix}
	(b-a)J_{k,p} & 0 \\
	0 & (a-b)J_{l,q}
	\end{bmatrix},\\\label{temp:eqn2}
	&\begin{bmatrix}
	J_{k,a} & -J_{k,b} \\
	-J_{l,a} & J_{l,b}
	\end{bmatrix}Z^T+X\begin{bmatrix}
	-J_{r,e} & J_{r,f} \\
	J_{s,e} & -J_{s,f}
	\end{bmatrix}^T=\begin{bmatrix}
	(d-c)J_{k,r} & 0 \\
	0 & (c-d)J_{l,s}
	\end{bmatrix},\\\label{temp:eqn3}
	&\begin{bmatrix}
	J_{p,a} & -J_{p,b} \\
	-J_{q,a} & J_{q,b}
	\end{bmatrix}Z^T+Y\begin{bmatrix}
	J_{r,c} & -J_{r,d} \\
	-J_{s,c} & J_{s,d}
	\end{bmatrix}^T=\begin{bmatrix}
	0 & (f-e)J_{p,s} \\
	(e-f)J_{q,r} & 0
	\end{bmatrix}.
	\end{align}
	From \eqref{temp:eqn1}, we have 
	\begin{align*}
	&J_{k,c}Y_{11}^T-J_{k,d}Y_{12}^T+X_{11}J_{e,p}-X_{12}J_{f,p}=(b-a)J_{k,p},\\ 
	&-J_{l,c}Y_{11}^T+J_{l,d}Y_{12}^T+X_{21}J_{e,p}-X_{22}J_{f,p}=0,\\ 
	&J_{k,c}Y_{21}^T-J_{k,d}Y_{22}^T-X_{11}J_{e,q}+X_{12}J_{f,q}=0,\\
	&-J_{l,c}Y_{21}^T+J_{l,d}Y_{22}^T-X_{21}J_{e,q}+X_{22}J_{f,q}=(a-b)J_{l,q}.
	\end{align*}
	By Lemma \ref{Lemma:JY+XJ}, \eqref{temp:eqn1} is equivalent to the conditions that $X_{11}\mathbf{1}_e-X_{12}\mathbf{1}_f=x_1\mathbf{1}_k$, $X_{21}\mathbf{1}_e-X_{22}\mathbf{1}_f=x_2\mathbf{1}_l$, $Y_{11}\mathbf{1}_c-Y_{12}\mathbf{1}_d=y_1\mathbf{1}_p$ and $Y_{21}\mathbf{1}_c-Y_{22}\mathbf{1}_d=y_2\mathbf{1}_q$ where $x_1+y_1=x_2+y_2=b-a$ and $x_1-y_2=x_2-y_1=0$. Similarly, applying Lemma \ref{Lemma:JY+XJ} to \eqref{temp:eqn2} and \eqref{temp:eqn3}, we find that \eqref{temp:eqn2} is equivalent to $z_1-x_1=z_2-x_2=d-c$ and $z_2+x_1=z_1+x_2=0$ where $Z_{11}\mathbf{1}_a-Z_{12}\mathbf{1}_b=z_1\mathbf{1}_r$ and $Z_{21}\mathbf{1}_a-Z_{22}\mathbf{1}_b=z_2\mathbf{1}_s$; \eqref{temp:eqn3} is equivalent to $z_2-y_1=z_1-y_2=f-e$ and $z_1+y_1=z_2+y_2=0$. Note that $x_1-y_2=x_2-y_1=z_2+x_1=z_1+x_2=z_1+y_1=z_2+y_2=0$ if and only if $x_1=y_2=-z_2$ and $x_2=y_1=-z_1$. Since $E\mathbf{1}=0$ and $\mathbf{1}^TE=0^T$, we have $a-b=d-c=f-e$. It follows that $x_1+x_2=y_1+y_2=-(z_1+z_2)=e-f$.
	
	Applying an analogous argument to the equivalent conditions for $A^TA=B^TB$ from Lemma \ref{Lemma:supporting Grammates rank 22}, it can be found that $\gamma_1=\beta_2=-\alpha_2$, $\gamma_2=\beta_1=-\alpha_1$, $\gamma_1+\gamma_2=\beta_1+\beta_2=-(\alpha_1+\alpha_2)=l-k$ where $X_{11}^T\mathbf{1}_k-X_{21}^T\mathbf{1}_l=\alpha_1\mathbf{1}_e$, $X_{12}^T\mathbf{1}_k-X_{22}^T\mathbf{1}_l=\alpha_2\mathbf{1}_f$, $Y_{11}^T\mathbf{1}_p-Y_{21}^T\mathbf{1}_q=\beta_1\mathbf{1}_c$,  $Y_{12}^T\mathbf{1}_p-Y_{22}^T\mathbf{1}_q=\beta_2\mathbf{1}_d$,  $Z_{11}^T\mathbf{1}_r-Z_{21}^T\mathbf{1}_s=\gamma_1\mathbf{1}_a$, and $Z_{12}^T\mathbf{1}_r-Z_{22}^T\mathbf{1}_s=\gamma_2\mathbf{1}_b$.
	
	On the other hand, subtracting $\mathbf{1}_l^TX_{21}\mathbf{1}_e-\mathbf{1}_l^TX_{22}\mathbf{1}_f=x_2\mathbf{1}_l^T\mathbf{1}_l$ from $\mathbf{1}_k^TX_{11}\mathbf{1}_e-\mathbf{1}_k^TX_{12}\mathbf{1}_f=x_1\mathbf{1}_k^T\mathbf{1}_k$, we have 
	$$
	(\mathbf{1}_k^TX_{11}-\mathbf{1}_l^TX_{21})\mathbf{1}_e-(\mathbf{1}_k^TX_{12}-\mathbf{1}_l^TX_{22})\mathbf{1}_f=x_1\mathbf{1}_k^T\mathbf{1}_k-x_2\mathbf{1}_l^T\mathbf{1}_l.
	$$
	Hence, $e\alpha_1-f\alpha_2=kx_1-lx_2$. Similarly, we can find $c\beta_1-d\beta_2=py_1-qy_2$ from $\mathbf{1}_p^TY_{11}\mathbf{1}_c-\mathbf{1}_p^TY_{12}\mathbf{1}_d=y_1\mathbf{1}_p^T\mathbf{1}_p$ and $\mathbf{1}_q^TY_{21}\mathbf{1}_c-\mathbf{1}_q^TY_{22}\mathbf{1}_d=y_2\mathbf{1}_q^T\mathbf{1}_q$; and $a\gamma_1-b\gamma_2=rz_1-sz_2$ from $\mathbf{1}_r^TZ_{11}\mathbf{1}_a-\mathbf{1}_r^TZ_{12}\mathbf{1}_b=z_1\mathbf{1}_r^T\mathbf{1}_r$ and $\mathbf{1}_s^TZ_{21}\mathbf{1}_a-\mathbf{1}_s^TZ_{22}\mathbf{1}_b=z_2\mathbf{1}_s^T\mathbf{1}_s$. Again by a similar argument, one can verify that \ref{rank22 c1} implies \ref{rank22 c3}.
\end{proof}

\begin{remark}\label{Remark:rank 22, zero cases}
	Let us continue the notation and result of Theorem \ref{Thm:Grammates Rank22}. Suppose $q=0$. This is equivalent to the fourth row partition of $E$ being annihilated. Since $Y_{21}$ and $Y_{22}$ in the corresponding $A$ are removed, the parameter $y_2$ does not appear for this case. Furthermore, examining the proof of Theorem \ref{Thm:Grammates Rank22}, in order to attain the equivalent condition for $A$ and $A+E$ to be Gram mates, we only need to modify the conditions as follows: $x_1=-z_2$, $x_2=y_1=-z_1$ and $x_1+x_2=-(z_1+z_2)=e-f$; that is, we can obtain the equivalent condition by annihilating the constraints involved with $y_2$ in \ref{rank22 c1} and \ref{rank22 c2}. In this manner, one can check that if the sub-indices indicating positions of row or column partitions in $E$ are zero, then the equivalent condition is achieved by removing the constraints involved with the parameters in \ref{rank22 c1} and \ref{rank22 c2} that do not appear, due to the resulting matrix being obtained from $E$ by deleting row or column partitions corresponding to the sub-indices. 
\end{remark}

\begin{example}\label{type1}
	Maintaining the notation and result of Theorem \ref{Thm:Grammates Rank22}, for $n\geq 1$, set $k=a=s=f=n+1$, $l=b=r=e=n$, $p=c=0$, and $q=d=1$. By Remark \ref{Remark:rank 22, zero cases}, we obtain the equivalent condition for $A$ and $A+E$ to be Gram mates, from \ref{rank22 c1} and \ref{rank22 c2} by modifying as follows: $x_1=y_2=-z_2$, $x_2=-z_1$, $x_1+x_2=-(z_1+z_2)=-1$, $\gamma_1=\beta_2=-\alpha_2$, $\gamma_2=-\alpha_1$, $\gamma_1+\gamma_2=-(\alpha_1+\alpha_2)=-1$.
	Since $p=c=0$ and $q=d=1$, we have $Y=-y_2=-\beta_2$. For ease of exposition, permuting the $4^\text{th}$ and $5^\text{th}$ row partitions and the $4^\text{th}$ and $5^\text{th}$ column partitions, we have $$E=\left[\begin{array}{cc|c|cc}
	J_{n+1} & -J_{n+1,n} & -\mathbf{1}_{n+1} & \multicolumn{2}{c}{\multirow{2}{*}{$0$}}  \\
	-J_{n,n+1} & J_{n} & \mathbf{1}_{n} & \multicolumn{2}{c}{} \\\hline
	-\mathbf{1}^T_{n+1} & \mathbf{1}^T_{n} & 0 & \mathbf{1}^T_{n+1} & -\mathbf{1}^T_n\\\hline
	\multicolumn{2}{c|}{\multirow{2}{*}{$0$}} & \mathbf{1}_{n+1} & -J_{n+1} & J_{n+1,n}\\
	\multicolumn{2}{c|}{} & -\mathbf{1}_{n} & J_{n,n+1} & -J_{n}
	\end{array}\right].$$ 
	Then, one can verify from the equivalent condition that $A$ is a Gram mate to $A+E$ if and only if $A$ is of the form as follows: 
	$$A=\left[\begin{array}{cc|c|cc}
	0 & J_{n+1,n} & \mathbf{1}_{n+1} &  \multicolumn{2}{c}{\multirow{2}{*}{$M_1$}} \\
	J_{n,n+1} & 0 & 0 & \multicolumn{2}{c}{} \\\hline
	\mathbf{1}^T_{n+1} & 0 & m & 0 & \mathbf{1}^T_{n}\\\hline
	\multicolumn{2}{c|}{\multirow{2}{*}{$M_2$}} & 0 & J_{n+1} & 0\\
	\multicolumn{2}{c|}{} & \mathbf{1}_{n} & 0 & J_{n}
	\end{array}\right]$$
	where for $m\in \{0,1\}$, $M_1$ and $M_2$ are $(2n+1)\times(2n+1)$ matrices such that
	$M_i\left[\begin{array}{c}
	\mathbf{1}_{n+1}\\
	-\mathbf{1}_n
	\end{array}\right]=\left[\begin{array}{c}
	m\mathbf{1}_{n+1}\\
	(1-m)\mathbf{1}_n
	\end{array}\right]$ and $M_i^T\left[\begin{array}{c}
	\mathbf{1}_{n+1}\\
	-\mathbf{1}_n
	\end{array}\right]=\left[\begin{array}{c}
	m\mathbf{1}_{n+1}\\
	(1-m)\mathbf{1}_n
	\end{array}\right]$ for $i\in \{1,2\}$.
\end{example}

\begin{remark}
	Let $E=\begin{bmatrix}
	\wtE & 0\\
	0 & 0
	\end{bmatrix}$ be a realizable matrix where $\wtE$ is of type \ref{index m5}. Then, Gram mates via $E$ can be characterized as in Remark \ref{Remark: Gram mates from wtE to E}.
\end{remark}

\begin{remark}\label{Remark:conditions for indices of M5}
	Continuing the hypotheses and notation in Theorem \ref{Thm:Grammates Rank22}, let us consider $e\alpha_1-f\alpha_2=kx_1-lx_2$, $c\beta_1-d\beta_2=py_1-qy_2$ and $a\gamma_1-b\gamma_2=rz_1-sz_2$ where $x_1=y_2=-z_2$, $x_2=y_1=-z_1$, $\gamma_1=\beta_2=-\alpha_2$, $\gamma_2=\beta_1=-\alpha_1$, $x_1+x_2=e-f$ and $\alpha_1+\alpha_2=k-l$. Then,
	\begin{align*}
	e\alpha_1-f\alpha_2&=kx_1-lx_2,\\
	-c\alpha_1+d\alpha_2&=px_2-qx_1,\\
	-a\alpha_2+b\alpha_1&=-rx_2+sx_1.
	\end{align*}
	Furthermore, we obtain a linear system
	\begin{align}\label{temp;linearsystem}
	\begin{split}
	(e+f)\alpha_1-(k+l)x_1&=f(k-l)-l(e-f),\\
	(c+d)\alpha_1-(p+q)x_1&=d(k-l)-p(e-f),\\
	(a+b)\alpha_1-(r+s)x_1&=a(k-l)-r(e-f).
	\end{split}
	\end{align}
	Note that the equations $(e+f)(p+q)=(k+l)(c+d)$, $(e+f)(r+s)=(k+l)(a+b)$ and $(c+d)(r+s)=(p+q)(a+b)$, which are from the determinant of the coefficient matrix of each pair of equations in \eqref{temp;linearsystem}, are equivalent to $\frac{e+f}{k+l}=\frac{c+d}{p+q}=\frac{a+b}{r+s}$, \textit{i.e.}, the ratios of the numbers of rows and columns of $X$, $Y$ and $Z$ are in proportion. For such case, the sizes of $X$, $Y$ and $Z$ are said to be \textit{proportional}.
	
	We observe that a pair of $\alpha_1$ and $x_1$ completely determine the integers in \ref{rank22 c1} and \ref{rank22 c2}. Using $a-b=d-c=f-e$ and $k-l=q-p=s-r$ from $E\mathbf{1}=0$ and $\mathbf{1}^TE=0^T$, respectively, it can be checked that $(\alpha_1,x_1)=(\frac{k-l}{2},\frac{e-f}{2})$ is a solution to the linear system \eqref{temp;linearsystem}. Then, $\alpha_1=\alpha_2$ and $x_1=x_2$. Considering $(0,1)$ matrices, $k-l$ and $e-f$ are even whenever there exist Gram mates $A$ of form \eqref{matrix A rank 22} and $A+E$ with $(\alpha_1,x_1)=(\frac{k-l}{2},\frac{e-f}{2})$.
	
	Suppose that there are two matrices among $X$, $Y$ and $Z$ such that their sizes are not proportional. Then, the system \eqref{temp;linearsystem} has a unique solution $(\alpha_1,x_1)=(\frac{k-l}{2},\frac{e-f}{2})$. So, if $k-l$ or $e-f$ is odd, then there are no $(0,1)$ matrices $X$, $Y$ and $Z$ satisfying \ref{rank22 c1} and \ref{rank22 c2}. In other words, if $E$ is realizable, then $k-l$ and $e-f$ both are even, \textit{i.e.,} $k+l$ and $e+f$ both are even.
	
	Assume that the sizes of $X$, $Y$ and $Z$ are proportional. Considering the equation $(e+f)\alpha_1-(k+l)x_1=f(k-l)-l(b-a)$ with a solution $(\alpha_1,x_1)=(\frac{k-l}{2},\frac{e-f}{2})$, we obtain $(\alpha_1,x_1)=t(k+l,e+f)+(\frac{k-l}{2},\frac{e-f}{2})$ for any $t$ as the solutions to \eqref{temp;linearsystem}. Set $t=\frac{1}{2}$. Then, $(\alpha_1,x_1)=(k,e)$. So, $x_1=y_2=-z_2=e$, $x_2=y_1=-z_1=-f$, $\gamma_1=\beta_2=-\alpha_2=l$ and $\gamma_2=\beta_1=-\alpha_1=-k$. This particular solution is used for showing that $E$ is realizable under the condition that $k-l$ or $e-f$ is odd, \textit{i.e.,} $k+l$ or $e+f$ is odd.
\end{remark}

Continuing the hypotheses and notation in Theorem \ref{Thm:Grammates Rank22}, by Remark \ref{Remark:sizes for M5} the number of rows (resp. columns) in each of $X$, $Y$ and $Z$ has the same parity. Let $X$ be an $m\times n$ $(0,1)$ matrix. We shall establish the equivalent condition for $E$ to be realizable (Theorem \ref{Thm:realizable E for rank22}) by considering two cases: $m$ and $n$ are even (Lemma \ref{Lemma:construction for X when a+b even}), and $m$ or $n$ is odd (Lemma \ref{Lemma:construction for X when a+b odd}). Note that we only need to show the existence of $X$, $Y$ and $Z$ satisfying \ref{rank22 c1} and \ref{rank22 c2} of Theorem \ref{Thm:Grammates Rank22}.

\begin{lemma}\label{Lemma:WLOG for contstruction M}
	Let $m_1,m_2,n_1,n_2>0$. Let $X=\begin{bmatrix}
	X_{11} & X_{12}\\
	X_{21} & X_{22}\end{bmatrix}$ be a $(0,1)$ matrix where $X_{11}$ and $X_{22}$ are $m_1\times n_1$ and $m_2\times n_2$ matrices, respectively. Suppose that $X\begin{bmatrix}
	\mathbf{1}_{n_1}\\-\mathbf{1}_{n_2}
	\end{bmatrix}=\begin{bmatrix}
	a_1\mathbf{1}_{m_1}\\a_2\mathbf{1}_{m_2}
	\end{bmatrix}$ and $X^T\begin{bmatrix}
	\mathbf{1}_{m_1}\\-\mathbf{1}_{m_2}
	\end{bmatrix}=\begin{bmatrix}
	b_1\mathbf{1}_{n_1}\\b_2\mathbf{1}_{n_2}
	\end{bmatrix}$. For $X_1=\begin{bmatrix}
	X_{12} & X_{11}\\
	X_{22} & X_{21}\end{bmatrix}$, we have $X_1\begin{bmatrix}
	\mathbf{1}_{n_2}\\-\mathbf{1}_{n_1}
	\end{bmatrix}=\begin{bmatrix}
	-a_1\mathbf{1}_{m_1}\\-a_2\mathbf{1}_{m_2}
	\end{bmatrix}$ and $X_1^T\begin{bmatrix}
	\mathbf{1}_{m_1}\\-\mathbf{1}_{m_2}
	\end{bmatrix}=\begin{bmatrix}
	b_2\mathbf{1}_{n_2}\\b_1\mathbf{1}_{n_1}
	\end{bmatrix}$. Furthermore, given $X_2=\begin{bmatrix}
	X_{22} & X_{21}\\
	X_{12} & X_{11}\end{bmatrix}$, we have $X_2\begin{bmatrix}
	\mathbf{1}_{n_2}\\-\mathbf{1}_{n_1}
	\end{bmatrix}=\begin{bmatrix}
	-a_2\mathbf{1}_{m_2}\\-a_1\mathbf{1}_{m_1}
	\end{bmatrix}$ and $X_2^T\begin{bmatrix}
	\mathbf{1}_{m_2}\\-\mathbf{1}_{m_1}
	\end{bmatrix}=\begin{bmatrix}
	-b_2\mathbf{1}_{n_2}\\-b_1\mathbf{1}_{n_1}
	\end{bmatrix}$.
\end{lemma}

\begin{lemma}\label{Lemma:construction for X when a+b even}
	Let $m_1,m_2,n_1,n_2>0$. Suppose that $m_1+m_2$ and $n_1+n_2$ are even. Then, there exists an $(m_1+m_2)\times (n_1+ n_2)$ $(0,1)$ matrix $X=\begin{bmatrix}
	X_{11} & X_{12}\\
	X_{21} & X_{22}\end{bmatrix}$, where $X_{11}$ and $X_{22}$ are $m_1\times n_1$ and $m_2\times n_2$ matrices, respectively, such that $X\begin{bmatrix}
	\mathbf{1}_{n_1}\\-\mathbf{1}_{n_2}
	\end{bmatrix}=\frac{n_1-n_2}{2}\begin{bmatrix}
	\mathbf{1}_{m_1}\\\mathbf{1}_{m_2}
	\end{bmatrix}$ and $X^T\begin{bmatrix}
	\mathbf{1}_{m_1}\\-\mathbf{1}_{m_2}
	\end{bmatrix}=\frac{m_1-m_2}{2}\begin{bmatrix}
	\mathbf{1}_{n_1}\\\mathbf{1}_{n_2}
	\end{bmatrix}$. 
\end{lemma}
\begin{proof}
	If $m_1=m_2$ and $n_1=n_2$, then $X$ can be chosen as the zero matrix. By Lemma \ref{Lemma:WLOG for contstruction M}, we only need to consider two cases: \begin{enumerate*}[label=(\roman*)]
		\item $m_1>m_2$ and $n_1=n_2$, \item $m_1>m_2$ and $n_1>n_2$.
	\end{enumerate*}
	If $m_1>m_2$ and $n_1=n_2$, then $X=\begin{bmatrix}
	J_{\frac{m_1-m_2}{2},n_1+n_2}\\0
	\end{bmatrix}$ is one of our desired matrices. Suppose $m_1>m_2$ and $n_1>n_2$. Let $\alpha=\frac{m_1-m_2}{2}$ and $\beta=\frac{n_1-n_2}{2}$. Then, it is straightforward to check that the following matrix can be our desired matrix: 
	\[X=\left[\begin{array}{ccc|c}
	J_{\alpha,\beta} & 0 & 0 & 0 \\ 
	0 & 0 & J_{m_2,\beta} & 0 \\
	0 & J_{\alpha,n_2} & J_{\alpha,\beta} & J_{\alpha,n_2} \\ \hline
	0 & 0 & J_{m_2,\beta} & 0
	\end{array}\right].
	\]
\end{proof}

Let $n>0$, and let $\alpha=(\alpha_1,\dots,\alpha_n)$ and $\beta=(\beta_1,\dots,\beta_n)$ be real vectors. Suppose that $\alpha'=(\alpha_1',\dots,\alpha_n')$ and $\beta'=(\beta_1',\dots,\beta_n')$ are obtained from $\alpha$ and $\beta$, respectively, by rearrangements such that $\alpha_1'\geq\dots\geq\alpha_n'$ and $\beta_1'\geq\dots\geq\beta_n'$. The vector $\alpha$ \textit{majorizes} $\beta$, denoted by $\alpha\succ\beta$, if $\sum_{i=1}^{k}\alpha_i'\geq\sum_{i=1}^{k}\beta_i'$ for all $1\leq k\leq n$, and $\sum_{i=1}^{n}\alpha_i'=\sum_{i=1}^{n}\beta_i'$. For non-negative integers $\alpha_1,\dots,\alpha_n$, define $\alpha_i^*=|\{\alpha_j|\alpha_j\geq i, j=1,\dots,n\}|$ for $i=1,\dots,n$. The vector $\alpha^*:=(\alpha_1^*,\dots,\alpha_n^*)$ is said to be \textit{conjugate} to $\alpha$. For instance, if $\alpha=(3,3,3,3,3)$, then $\alpha^*=(5,5,5,0,0)$.

Let $\mathcal{U}(R,S)$ denote the set of all $(0,1)$ matrices with row sum vector $R$ and column sum vector $S$. Let $R=(r_1,\dots,r_m)$ and $S=(s_1,\dots,s_n)$ be nonnegative integral vectors. In the context of majorization, we may adjoin zeros to $S$ (resp. $R$) if $m>n$ (resp. $m<n$). 

\begin{theorem}\label{Thm:Gale Ryser}\cite{Ryser:zeros-ones}(the Gale--Ryser theorem)
	Let $R=(r_1,\dots,r_m)$ and $S=(s_1,\dots,s_n)$ be nonnegative integral vectors. Then, there exists a $(0,1)$ matrix in $\mathcal{U}(R,S)$ if and only if $S\prec R^*$ and $r_i\leq n$ for $i=1,\dots,m$. 
\end{theorem}

We refer the reader to \cite{Ryser:zeros-ones} for the construction of a matrix in $\mathcal{U}(R,S)$. From Theorem \ref{Thm:Gale Ryser}, we immediately have two lemmas as follows.

\begin{lemma}\label{Lemma:existence of matrix with regular vectors}
	Let $R=(r,\dots,r)$ and $S=(s,\dots,s)$ be nonnegative integral vectors of size $m(\geq s)$ and $n(\geq r)$, respectively. Then, $rm=sn$ if and only if there exists a matrix in $\cU(R,S)$.
\end{lemma}

\begin{lemma}\label{Lemma:evenS}
	Let $S=(s_1,\dots,s_n)$ be a nonnegative integral vector where $\sum_{i=1}^n s_i=\ell$. Let $m$ be a positive integer such that $s_i\leq m$ for all $i=1,\dots n$ and $\ell=qm+r$ for some $q\geq 0$ and $0\leq r\leq m-1$. Suppose $$R=(\underbrace{q+1,\dots,q+1}_{r\;\text{times}},\underbrace{q,\dots,q}_{m-r\;\text{times}}).$$
	Then, there exists a matrix $A\in\mathcal{U}(R,S)$.
\end{lemma}
\begin{proof}
	Evidently, $l\leq mn$, so $q+1\leq n$ whenever $r>0$. Since we have $R^*=(m,\dots,m,r,0,\dots,0)$ where $m$ appears $q$ times, we can see that $S\prec R^*$. By Theorem \ref{Thm:Gale Ryser}, our desired result is obtained.
\end{proof}

Given a nonnegative integral vector $R=(r_1,\dots,r_m)$ where $\sum_{i=1}^m r_i=\ell$, we shall establish an analogous result for Lemma \ref{Lemma:evenS} by constructing a concrete matrix. Let $n$ be a positive integer such that $r_i\leq n$ for all $i=1,\dots, m$. Choose $q\geq 0$ so that $\ell=qn+r$ for some $0\leq r\leq n-1$. Let $r_0=0$, and let $A_{0}$ be the $m\times n(q+1)$ matrix such that if $r_i>0$, then the $i^{\text{th}}$ row of $A_0$ consists of $1$'s from the $(1+\sum_{j=1}^{i}r_{j-1})^{\text{th}}$ position to the $(\sum_{j=1}^{i}r_j)^{\text{th}}$ position and $0$'s elsewhere; if $r_i=0$, then the $i^{\text{th}}$ row of $A_0$ is a row of zeros. Then, $A_0$ can be partitioned into $(q+1)$ submatrices $A_1,\dots,A_{q+1}$ so that for $1\leq i\leq q+1$, $A_i$ is an $m\times n$ submatrix of $A_0$ whose columns are indexed by $(i-1)n+1,\dots,(i-1)n+n$.

Let $A=A_1+\cdots+A_{q+1}$. It is clear that the row sum vector of $A$ is $R$. Since ones in each row of $A_0$ appear consecutively and each row contains at most $n$ ones, $A$ must be a $(0,1)$ matrix. Furthermore, every column of $A_i$ for $i=1,\dots, q$ contains precisely a single one, and each of the first $r$ columns of $A_{q+1}$ contains a single one. Therefore, $A\in\mathcal{U}(R,S)$ where \begin{align}\label{Result:even R}
S=(\underbrace{q+1,\dots,q+1}_{r\;\text{times}},\underbrace{q,\dots,q}_{n-r\;\text{times}}).
\end{align}

\begin{example}\label{Example:construction for S*}
	Let $R=(3,3,0,2,3)$. Consider the described matrix $A_0$ above:
	$$
	A_0=\left[\begin{array}{cccc|cccc|cccc}
	1 & 1 & 1 & 0 & 0 & 0 & 0 & 0 & 0 & 0 & 0 & 0 \\
	0 & 0 & 0 & 1 & 1 & 1 & 0 & 0 & 0 & 0 & 0 & 0 \\
	0 & 0 & 0 & 0 & 0 & 0 & 0 & 0 & 0 & 0 & 0 & 0 \\
	0 & 0 & 0 & 0 & 0 & 0 & 1 & 1 & 0 & 0 & 0 & 0 \\
	0 & 0 & 0 & 0 & 0 & 0 & 0 & 0 & 1 & 1 & 1 & 0 
	\end{array}\right].
	$$
	As explained above, we obtain 
	$$
	A=\left[\begin{array}{cccc}
	1 & 1 & 1 & 0 \\
	1 & 1 & 0 & 1 \\
	0 & 0 & 0 & 0 \\
	0 & 0 & 1 & 1 \\
	1 & 1 & 1 & 0 
	\end{array}\right].
	$$
\end{example}

Continuing with the hypotheses and notation in Theorem \ref{Thm:Grammates Rank22}, assume that $E$ is a $(0,1,-1)$ matrix of the form \ref{index m5} and $A$ is a $(0,1)$ matrix of the form \eqref{matrix A rank 22}. Recall from Remark \ref{Remark:conditions for indices of M5} that given that the number of rows or columns of $X$ in $A$ is odd, if $E$ is realizable, then necessarily the sizes $X$, $Y$ and $Z$ in $A$ are proportional.

We now claim that the converse holds. In order to establish the claim, we shall construct $(0,1)$ matrices $X$, $Y$ and $Z$ satisfying \ref{rank22 c1} and \ref{rank22 c2} of Theorem \ref{Thm:Grammates Rank22}, under the condition that the sizes of $X$, $Y$ and $Z$ are in proportion in absence of the parities of their size. Since the sizes of $X$ and $Y$ are proportional, if the number of rows of $Y$ is bigger than that of rows of $X$, then the number of columns of $Y$ is bigger than that of columns of $X$, and vice versa. Hence, we may assume that $X$ has the smallest size. Then, it is enough to show that the existence of $X$ implies that of $Y$. Considering the last paragraph of Remark \ref{Remark:conditions for indices of M5}, we shall find $X$ and $Y$ with the condition that $x_1=y_2=e$, $x_2=y_1=-f$, $\beta_2=-\alpha_2=l$ and $\beta_1=-\alpha_1=-k$. Then, one of the desired matrices for $X$ can be obtained as $\begin{bmatrix}
J_{k,e} & 0\\
0 & J_{l,f}
\end{bmatrix}$.

If the size of $Y$ is the same as that of $X$, then we choose $Y=X$. Suppose $p+q>k+l$ and $c+d>e+f$. Since the sizes of $X$ and $Y$ are proportional, $\frac{c+d}{p+q}=\frac{e+f}{k+l}$. From $E\mathbf{1}=0$ and $\mathbf{1}^TE=0^T$, we have $q-p=k-l$ and $c-d=e-f$. We need to construct a $(0,1)$ matrix $Y$ such that $$Y\begin{bmatrix}
\mathbf{1}_c\\
-\mathbf{1}_d
\end{bmatrix}=\begin{bmatrix}
-f\mathbf{1}_p\\
e\mathbf{1}_q
\end{bmatrix}\;\text{and}\;Y^T\begin{bmatrix}
\mathbf{1}_p\\
-\mathbf{1}_q
\end{bmatrix}=\begin{bmatrix}
-k\mathbf{1}_c\\
l\mathbf{1}_d
\end{bmatrix}.$$
By Lemma \ref{Lemma:WLOG for contstruction M}, we may prove that there exists a $(0,1)$ matrix $Y_1$ such that $$Y_1\begin{bmatrix}
\mathbf{1}_c\\
-\mathbf{1}_d
\end{bmatrix}=\begin{bmatrix}
e\mathbf{1}_q\\
-f\mathbf{1}_p
\end{bmatrix}\;\text{and}\;Y_1^T\begin{bmatrix}
\mathbf{1}_q\\
-\mathbf{1}_p
\end{bmatrix}=\begin{bmatrix}
k\mathbf{1}_c\\
-l\mathbf{1}_d
\end{bmatrix}.$$

\begin{lemma}\label{Lemma:construction for X when a+b odd}
	Let $a_1,a_2,b_1,b_2>0$. Suppose that $m_1$, $m_2$, $n_1$ and $n_2$ are positive integers such that $m_1+m_2>a_1+a_2$, $n_1+n_2>b_1+b_2$, $m_1-m_2=a_1-a_2$, $n_1-n_2=b_1-b_2$ and $\frac{n_1+n_2}{m_1+m_2}=\frac{b_1+b_2}{a_1+a_2}$. Then, there exists a $(0,1)$ matrix $Y$ of size $(m_1+m_2)\times(n_1+n_2)$ such that 
	\begin{align*}
	Y\begin{bmatrix}
	\mathbf{1}_{n_1}\\-\mathbf{1}_{n_2}
	\end{bmatrix}=\begin{bmatrix}
	b_1\mathbf{1}_{m_1}\\-b_2\mathbf{1}_{m_2}\end{bmatrix},&\;Y^T\begin{bmatrix}
	\mathbf{1}_{m_1}\\-\mathbf{1}_{m_2}
	\end{bmatrix}=\begin{bmatrix}
	a_1\mathbf{1}_{n_1}\\-a_2\mathbf{1}_{n_2}\end{bmatrix}.
	\end{align*}
\end{lemma}
\begin{proof}
	We shall construct a $(0,1)$ matrix $Y=\begin{bmatrix}
	Y_{11} & Y_{12}\\
	Y_{21} & Y_{22}
	\end{bmatrix}$ with our desired property where $Y_{11}$ and $Y_{22}$ are $m_1\times n_1$ and $m_2\times n_2$ matrices, respectively. Here, for $i,j\in\{1,2\}$, we denote the row sum and column sum vectors of $Y_{ij}$ by $R_{ij}$ and $S_{ij}$, respectively.
	
	Using $m_1-m_2=a_1-a_2$, $n_1-n_2=b_1-b_2$, we find from $m_1+m_2>a_1+a_2$ and $n_1+n_2>b_1+b_2$ that $m_1>a_1$, $m_2>a_2$, $n_1>b_1$ and $n_2>b_2$; and we see from $\frac{n_1+n_2}{m_1+m_2}=\frac{b_1+b_2}{a_1+a_2}$ that $m_1b_1+m_2b_2=n_1a_1+n_2a_2$. Suppose that $m_1b_1=n_1a_1$. Then, $m_2b_2=n_2a_2$. By Lemma \ref{Lemma:existence of matrix with regular vectors}, there exist $Y_{11}\in\cU(R_{11},S_{11})$ and $Y_{22}\in\cU(R_{22},S_{22})$ such that $R_{11}=b_1\mathbf{1}_{m_1}$, $S_{11}=a_1\mathbf{1}_{n_1}$, $R_{22}=b_2\mathbf{1}_{m_2}$ and $S_{22}=a_2\mathbf{1}_{n_2}$. Choosing $Y_{12}=0$ and $Y_{21}=0$, our desired matrix is obtained.
	
	Considering the transpose of $Y$, we may assume that $m_1b_1>n_1a_1$. Set $Y_{12}=0$. Let $\ell=m_1b_1-n_1a_1$. We first construct $Y_{11}$. Choose $q_1\geq 0$ such that $\ell=n_1q_1+r_1$ for some $0\leq r_1< n_1$. Let $\widetilde{S}=(q_1+1,\dots,q_1+1,q_1,\dots,q_1)$ where $q_1+1$ and $q_1$ appear $r_1$ times and $n_1-r_1$ times in $\widetilde{S}$, respectively. Let $S_{11}=a_1\mathbf{1}_{n_1}+\widetilde{S}$ and $R_{11}=b_1\mathbf{1}_{m_1}$. Since $\ell=m_1b_1-n_1a_1$, the sum of the entries of $S_{11}$ is the same as the sum for $R_{11}$. We shall show that each entry in $S_{11}$ is not greater than $m_1$. From $n_1>b_1$, we have
	\begin{align*}
	m_1-(a_1+q_1)&=m_1-\frac{n_1a_1+\ell-r_1}{n_1}\\
	&=m_1-\frac{m_1b_1-r_1}{n_1}=\frac{m_1(n_1-b_1)+r_1}{n_1}>0.
	\end{align*}
	So, $m_1-(q_1+1+a_1)\geq 0$. Furthermore, $R_{11}^*\succ S_{11}$. By Theorem \ref{Thm:Gale Ryser}, there exists a matrix $Y_{11}\in\cU(R_{11},S_{11})$. 
	
	Now, we shall construct $Y_{21}$ with a column sum vector as $\widetilde{S}$. Then, we need to show $m_2\geq q_1+1$. From $m_1-m_2=a_1-a_2$ and $n_1>b_1$, we have
	\begin{align*}
	m_2-q_1&=\frac{m_2n_1-\ell+r_1}{n_1}\\
	&=\frac{m_2n_1-m_1b_1+n_1a_1+r_1}{n_1}=\frac{m_1(n_1-b_1)+n_1a_2+r_1}{n_1}>0.
	\end{align*}
	So, $m_2\geq q_1+1$. Choose $q_2\geq 0$ such that $\ell=m_2q_2+r_2$ for some $0\leq r_2< m_2$. Applying Lemma \ref{Lemma:evenS} with $\widetilde{S}$, there exists $Y_{21}\in\cU(R_{21},S_{21})$ such that $S_{21}=\widetilde{S}$ and $R_{21}=(q_2+1,\dots,q_2+1,q_2,\dots,q_2)$ where $q_2+1$ and $q_2$ appear $r_2$ times and $m_2-r_2$ times in $R_{21}$, respectively. 
	
	Let $R_{22}=b_2\mathbf{1}_{m_2}+R_{21}$ and $S_{22}=a_2\mathbf{1}_{n_2}$. Then, the sum of the entries of $R_{22}$ is equal to the sum for $S_{22}$. Using $\ell=m_1b_1-n_1a_1$, $m_1b_1+m_2b_2=n_1a_1+n_2a_2$ and $m_2>a_2$, we obtain
	\begin{align*}
	n_2-(b_2+q_2)=n_2-\frac{m_2b_2+\ell-r_2}{m_2}=\frac{n_2(m_2-a_2)+r_2}{m_2}>0.
	\end{align*}
	Then, $n_2-(q_2+1+b_2)\geq 0$. Moreover, $S_{22}^*\succ R_{22}$. By Theorem \ref{Thm:Gale Ryser}, there exists $Y_{22}\in\cU(R_{22},S_{22})$. Therefore, the conclusion follows.
\end{proof}

\begin{example}
	Let $a_1=4$, $a_2=6$, $b_1=5$, $b_2=3$, $m_1=9$, $m_2=11$, $n_1=9$ and $n_2=7$. One can check that the indices satisfy the hypothesis of Lemma \ref{Lemma:construction for X when a+b odd}. We use the results and notation in the proof of Lemma \ref{Lemma:construction for X when a+b odd}. Then, $m_1b_1-n_1a_1=9$ and so $\widetilde{S}=\mathbf{1}_9$. It can be found that we have $Y_{11}\in\cU(R_{11},S_{11})$ where $R_{11}=5\mathbf{1}_9$ and $S_{11}=4\mathbf{1}_9+\widetilde{S}$; $Y_{21}\in\cU(R_{21},S_{21})$ where $R_{21}^T=(\mathbf{1}_9^T,0,0)$ and $S_{21}=\widetilde{S}$; and $Y_{22}\in\cU(R_{22},S_{22})$ where $R_{22}=R_{21}+3\mathbf{1}_{11}$ and $S_{22}=6\mathbf{1}_7$. As in the construction described in Example \ref{Example:construction for S*}, we can obtain $Y_{21}=\begin{bmatrix}
	I_9 \\
	0
	\end{bmatrix}$. (One can obtain $Y_{11}$ and $Y_{22}$ by using the process illustrated in \cite{Ryser:zeros-ones}.)
\end{example}

Applying Lemmas \ref{Lemma:construction for X when a+b even} and \ref{Lemma:construction for X when a+b odd}, the following result is established.

\begin{theorem}\label{Thm:realizable E for rank22}
	Let $E$ be a $(0,1,-1)$ matrix of the form \ref{index m5}. Then, $E$ is realizable if and only if one of the following conditions is satisfied:
	\begin{enumerate}[label=(\roman*)]
		\item $a+b$, $c+d$, $e+f$, $k+l$, $p+q$ and $r+s$ are all even, and
		\item three numbers $a+b$, $c+d$ and $e+f$ are all odd or three numbers $k+l$, $p+q$ and $r+s$ are all odd; and $\frac{e+f}{k+l}=\frac{c+d}{p+q}=\frac{a+b}{r+s}$.
	\end{enumerate}
\end{theorem}

Assume that $A$ and $A+E$ are Gram mates via a $(0,1,-1)$ matrix $E$ of the form \ref{index m5}. We use the results and notation in Theorem \ref{Thm:Grammates Rank22}. Consider $$2A+E=\begin{bmatrix}
J_{k,a} & J_{k,b}  & J_{k,c} & J_{k,d} & 2X_{11} & 2X_{12} \\
J_{l,a} & J_{l,b} & J_{l,c} & J_{l,d} & 2X_{21} & 2X_{22} \\
J_{p,a} & J_{p,b}  & 2Y_{11} & 2Y_{12} & J_{p,e} & J_{p,f} \\
J_{q,a} & J_{q,b} & 2Y_{21} & 2Y_{22} & J_{q,e} & J_{q,f} \\
2Z_{11} & 2Z_{12} & J_{r,c} & J_{r,d} & J_{r,e} & J_{r,f}  \\
2Z_{21} & 2Z_{22} & J_{s,c} & J_{s,d} & J_{s,e} & J_{s,f} \\
\end{bmatrix}.$$ 
Let 
\begin{align}\label{temp13}
\begin{split}
\bx_1^T&=\left[\begin{array}{cccccc}
\mathbf{1}_{a}^T  & -\mathbf{1}_{b}^T & \mathbf{1}_{c}^T & -\mathbf{1}_{d}^T & \mathbf{0}_{e}^T & \mathbf{0}_{f}^T \end{array}\right],\\
\bx_2^T&=\left[\begin{array}{cccccc}
\mathbf{1}_{a}^T  & -\mathbf{1}_{b}^T & \mathbf{0}_{c}^T & \mathbf{0}_{d}^T & \mathbf{1}_{e}^T & -\mathbf{1}_{f}^T \end{array}\right].
\end{split}
\end{align}
From $E\mathbf{1}=0$, we have $b-a=c-d=e-f$. So,
\begin{align*}
(2A+E)\bx_1=\begin{bmatrix}
0\\ 0 \\ (f-e+2y_1)\mathbf{1}_p \\ (f-e+2y_2)\mathbf{1}_q \\ (2z_1+e-f)\mathbf{1}_r \\ (2z_2+e-f)\mathbf{1}_s
\end{bmatrix},\;
(2A+E)\bx_2=\begin{bmatrix}
(f-e+2x_1)\mathbf{1}_k \\ (f-e+2x_2)\mathbf{1}_l \\ 0 \\ 0\\ (2z_1+e-f)\mathbf{1}_r \\ (2z_2+e-f)\mathbf{1}_s
\end{bmatrix}.
\end{align*}
Since $A$ and $A+E$ are Gram mates, the condition \ref{rank22 c1} in Theorem \ref{Thm:Grammates Rank22} holds, so we have $x_1=y_2=-z_2$, $x_2=y_1=-z_1$ and $x_1+x_2=y_1+y_2=-(z_1+z_2)=e-f$. Note that $\bx_1$ and $\bx_2$ form a basis of $\mathrm{Row}(E)$. Hence, $(2A+E)E^T=0$ if and only if $x_1=\frac{e-f}{2}$. Therefore, by Theorem \ref{Thm:equivalent conditions}, $A$ and $A+E$ are convertible if and only if $x_1=\frac{e-f}{2}$.

Suppose that $x_1=\frac{e-f}{2}$. Since $(\alpha_1,x_1)$ is a solution to the system \eqref{temp;linearsystem} in Remark \ref{Remark:conditions for indices of M5}, we have $\alpha_1=\frac{k-l}{2}$. Hence, from $E\mathbf{1}=0$, $\mathbf{1}^TE=0^T$ and the conditions \ref{rank22 c1} and \ref{rank22 c2} of Theorem \ref{Thm:Grammates Rank22}, for $i=1,2$, $x_i=y_i=-z_i=\frac{e-f}{2}=\frac{c-d}{2}=\frac{b-a}{2}$ and $\gamma_i=\beta_i=-\alpha_i=\frac{l-k}{2}=\frac{p-q}{2}=\frac{r-s}{2}$. We now find the Gram singular values of $A$ and $A+E$ and the corresponding right singular vectors. It can be computed that
\begin{align*}
A^TA\begin{bmatrix}
\mathbf{1}_a\\-\mathbf{1}_b\\\mathbf{1}_c\\-\mathbf{1}_d\\ 0 \\ 0
\end{bmatrix}=\begin{bmatrix}
(l(a+c)+q(a+y_2))\mathbf{1}_a+(z_1-d)Z_{11}^T\mathbf{1}_r+(z_2+c)Z_{21}^T\mathbf{1}_s\\
(k(-b-d)+p(-b+y_1))\mathbf{1}_b+(z_1-d)Z_{12}^T\mathbf{1}_r+(z_2+c)Z_{22}^T\mathbf{1}_s\\
(l(a+c)+s(z_2+c))\mathbf{1}_c+(-b+y_1)Y_{11}^T\mathbf{1}_p+(a+y_2)Y_{21}^T\mathbf{1}_q\\
(k(-b-d)+r(z_1-d))\mathbf{1}_d+(-b+y_1)Y_{12}^T\mathbf{1}_p+(a+y_2)Y_{22}^T\mathbf{1}_q\\
(q(a+y_2)+r(z_1-d))\mathbf{1}_e+(-b-d)X_{11}^T\mathbf{1}_k+(a+c)X_{21}^T\mathbf{1}_l\\
(p(-b+y_1)+s(z_2+c))\mathbf{1}_f+(-b-d)X_{21}^T\mathbf{1}_k+(a+c)X_{22}^T\mathbf{1}_l
\end{bmatrix}.
\end{align*}
Substituting $\frac{d-c}{2}$ into $z_1$ and $z_2$ in the first entry of the right side, we have $-\frac{1}{2}(c+d)Z_{11}^T\mathbf{1}_r+\frac{1}{2}(c+d)Z_{21}^T\mathbf{1}_s=-\gamma_1(c+d)\mathbf{1}_a$. So, from $y_2=\frac{b-a}{2}$, we find $(l(a+c)+\frac{1}{2}q(a+b)+\gamma_1(c+d))\mathbf{1}_a$ in the first entry. Applying a similar argument for the remaining entries of the right side, it follows that
$$A^TA\begin{bmatrix}
\mathbf{1}_a\\-\mathbf{1}_b\\\mathbf{1}_c\\-\mathbf{1}_d\\ 0 \\ 0
\end{bmatrix}=\begin{bmatrix}
(l(a+c)+\frac{1}{2}q(a+b)-\frac{1}{2}\gamma_1(c+d))\mathbf{1}_a\\(-k(a+c)-\frac{1}{2}p(a+b)-\frac{1}{2}\gamma_2(c+d))\mathbf{1}_b\\ (l(a+c)+\frac{1}{2}s(c+d)-\frac{1}{2}\beta_1(a+b))\mathbf{1}_c\\(-k(a+c)-\frac{1}{2}r(c+d)-\frac{1}{2}\beta_2(a+b))\mathbf{1}_d\\ (\frac{1}{2}q(a+b)-\frac{1}{2}r(c+d)-\alpha_1(b+d))\mathbf{1}_e\\(-\frac{1}{2}p(a+b)+\frac{1}{2}s(c+d)-\alpha_2(b+d))\mathbf{1}_f
\end{bmatrix}.
$$
Similarly, one can verify that
\begin{align*}
A^TA\begin{bmatrix}
\mathbf{1}_a\\-\mathbf{1}_b\\0 \\0 \\\mathbf{1}_e\\-\mathbf{1}_f
\end{bmatrix}
=\begin{bmatrix}
(\frac{1}{2}l(a+b)+q(a+e)+\frac{1}{2}\gamma_1(e+f))\mathbf{1}_a\\(-\frac{1}{2}k(a+b)-p(a+e)+\frac{1}{2}\gamma_2(e+f)\mathbf{1}_b\\ (\frac{1}{2}l(a+b)-\frac{1}{2}s(e+f)-\beta_1(b+f))\mathbf{1}_c\\(-\frac{1}{2}k(a+b)+\frac{1}{2}r(e+f)-\beta_2(b+f))\mathbf{1}_d\\ (q(a+e)+\frac{1}{2}r(e+f)-\frac{1}{2}\alpha_1(a+b))\mathbf{1}_e\\(-p(a+e)-\frac{1}{2}s(e+f)-\frac{1}{2}\alpha_2(a+b))\mathbf{1}_f
\end{bmatrix}.
\end{align*}

Consider an equation $A^TA(\zeta_1\bx_1+\zeta_2\bx_2)=\lambda(\zeta_1\bx_1+\zeta_2\bx_2)$ where $\zeta_1$, $\zeta_2$ and $\lambda$ are real numbers. Using  $\gamma_i=\beta_i=\alpha_i=\frac{l-k}{2}=\frac{p-q}{2}=\frac{r-s}{2}$, it can be checked that the equations from the first, third and fifth row blocks are identical with those from the second, fourth and sixth row blocks, respectively. Moreover, we can find that the equation from the first row block is the same as the addition of two equations from the third and fifth row blocks. Hence, $A^TA(\zeta_1\bx_1+\zeta_2\bx_2)=\lambda(\zeta_1\bx_1+\zeta_2\bx_2)$ is equivalent to a linear system of two equations from the third and fifth row blocks:

\footnotesize\begin{align*}
&\zeta_1\left(l(a+c)+\frac{1}{2}s(c+d)-\frac{1}{2}\beta_1(a+b)\right)+\zeta_2\left(\frac{1}{2}l(a+b)-\frac{1}{2}s(e+f)-\beta_1(b+f)\right)=\lambda \zeta_1,\\
&\zeta_1\left(\frac{1}{2}q(a+b)-\frac{1}{2}r(c+d)-\alpha_1(b+d)\right)+\zeta_2\left(q(a+e)+\frac{1}{2}r(e+f)-\frac{1}{2}\alpha_1(a+b)\right)=\lambda \zeta_2.
\end{align*}\normalsize
Therefore, we have the following result.
\begin{theorem}\label{Thm:Gramsingular rank 22}
	Let $E$ be a realizable matrix of form \ref{index m5}, and let
	\begin{align*}
	A=\begin{bmatrix}
	0 & J_{k,b}  & 0 & J_{k,d} & X_{11} & X_{12} \\
	J_{l,a} & 0 & J_{l,c} & 0 & X_{21} & X_{22} \\
	0 & J_{p,b}  & Y_{11} & Y_{12} & 0 & J_{p,f} \\
	J_{q,a} & 0 & Y_{21} & Y_{22} & J_{q,e} & 0 \\
	Z_{11} & Z_{12} & 0 & J_{r,d} & J_{r,e} & 0  \\
	Z_{21} & Z_{22} & J_{s,c} & 0 & 0 & J_{s,f} \\
	\end{bmatrix}
	\end{align*}
	be a $(0,1)$ matrix conformally partitioned with $E$. Suppose that $A$ and $A+E$ are Gram mates. Then, $A$ and $A+E$ are convertible if and only if one of the following hold: \begin{enumerate*}[label=(\roman*)]
		\item $(X_{i1}-X_{i2})\mathbf{1}=\frac{e-f}{2}\mathbf{1}$ for $i=1,2$; \item $(X^T_{i1}-X^T_{i2})\mathbf{1}=\frac{k-l}{2}\mathbf{1}$ for $i=1,2$; \item $(Y_{i1}-Y_{i2})\mathbf{1}=\frac{c-d}{2}\mathbf{1}$ for $i=1,2$; \item $(Y^T_{i1}-Y^T_{i2})\mathbf{1}=\frac{p-q}{2}\mathbf{1}$ for $i=1,2$; \item $(Z_{i1}-Z_{i2})\mathbf{1}=\frac{a-b}{2}\mathbf{1}$ for $i=1,2$; and \item $(Z^T_{i1}-Z^T_{i2})\mathbf{1}=\frac{r-s}{2}\mathbf{1}$ for $i=1,2$.
	\end{enumerate*} 
	Furthermore, if $A$ is convertible to $A+E$, then the Gram singular values of $A$ and $A+E$ are the square roots of the eigenvalues $\lambda$ of $M$ where 
	$$
	M=\footnotesize\begin{bmatrix}
	l(a+c)+\frac{1}{2}s(c+d)+\frac{1}{4}(k-l)(a+b) & \frac{1}{2}l(a+b)-\frac{1}{2}s(e+f)+\frac{1}{2}(k-l)(a+e)\\
	\frac{1}{2}q(a+b)-\frac{1}{2}r(c+d)+\frac{1}{2}(p-q)(a+c) & q(a+e)+\frac{1}{2}r(e+f)+\frac{1}{4}(p-q)(a+b)
	\end{bmatrix}.
	$$\normalsize
	Further, a right singular vector associated to $\sqrt{\lambda}$ is a normalized vector of $\zeta_1\bx_1+\zeta_2\bx_2$, where $\bx_1$ and $\bx_2$ are the vectors in \eqref{temp13} and $(\zeta_1,\zeta_2)$ is an eigenvector of $M$ associated to $\lambda$.
\end{theorem}

%%%%%%%%%%%%%%%%%%%%%%%%%%%%%%%%%%%%%%%%%%%%%%%%%%%%%%%%%
\section{Construction of Gram mates}\label{section:construction}
%%%%%%%%%%%%%%%%%%%%%%%%%%%%%%%%%%%%%%%%%%%%%%%%%%%%%%%%%
In this section, using given pairs of Gram mates, we provide several tools for constructing other pairs. Similarly, with given realizable matrices, we construct others; and we discuss realizable matrices of rank more than $2$.

All propositions in this section can be proved by directly verifying Definition \ref{Def:GramMates}. Hence, we introduce those propositions without proofs.

\begin{proposition}\label{construction:1}
	Let $A$ and $B$ be Gram mates. Then, $J-A$ and $J-B$ are Gram mates.
\end{proposition}

\begin{proposition}\label{construction:2}
	Suppose that $(A_1,B_1)$ and $(A_2,B_2)$ are pairs of Gram mates (that are not necessarily square matrices). Then, we have the following pairs of Gram mates: 
	$$\left(\begin{bmatrix}
	A_1 & 0\\
	0 & A_2
	\end{bmatrix},\begin{bmatrix}
	B_1 & 0\\
	0 & B_2
	\end{bmatrix}\right)\;\text{and}\;\left(\begin{bmatrix}
	A_1 & J\\
	J & A_2
	\end{bmatrix},\begin{bmatrix}
	B_1 & J\\
	J & B_2
	\end{bmatrix}\right).$$
\end{proposition}

Let $\otimes$ denote the Kronecker product of matrices \cite{Demmel:AppliedNumerical}.

\begin{proposition}\label{construction:3}
	Suppose that $(A_1,B_1)$ and $(A_2,B_2)$ are pairs of Gram mates. Then, $A_1\otimes B_1$ and $A_2\otimes B_2$ are Gram mates. 
\end{proposition}

\begin{remark}
	In analogy to graph operations, one might regard ways of the constructions for Gram mates in Propositions \ref{construction:1}--\ref{construction:3} as the complement of a graph, the disjoint union of two graphs and the join of two graphs, and the Cartesian product of two graphs, respectively.
\end{remark}

We can apply analogous approaches with Propositions \ref{construction:2} and \ref{construction:3} to realizable matrices. Given realizable matrices $E_1$ and $E_2$, it follows from Proposition \ref{construction:2} that $\begin{bmatrix}
E_1 & 0\\
0 & E_2
\end{bmatrix}$ is realizable.

\begin{proposition}
	Let $E$ be a realizable, and let $X$ be a $(0,1)$ matrix. Then, $X\otimes E$ and $E\otimes X$ are realizable.
\end{proposition}

\begin{remark}
	With the aid of the results in this section, we can construct Gram mates via realizable matrices of rank more than $2$ from those of rank at most $2$ that we studied in Section \ref{section:rank1 and 2}. As an example, let us consider a realizable matrix $E$ of rank $2$. For a $(0,1)$ matrix $X$ of rank $k>0$, $X\otimes E$ is realizable, and its rank is $2k$.
\end{remark}

%%%%%%%%%%%%%%%%%%%%%%%%%%%%%%%%%%%%%%%%%%%%%%%%%%%%%%%%%%%%%%%%%%%%%%%%%%%%%%%%%%%%%%%%%%%%%%%%
\section{Non-isomorphic Gram mates via realizable matrices of rank $1$}\label{section:noniso}
%%%%%%%%%%%%%%%%%%%%%%%%%%%%%%%%%%%%%%%%%%%%%%%%%%%%%%%%%%%%%%%%%%%%%%%%%%%%%%%%%%%%%%%%%%%%%%%%

Here we revisit the speculative statement in \cite{everett:dual}: if Gram mates $A$ and $B$ with distinct positive singular values are convertible, then $A$ and $B$ are isomorphic with very high probability. We characterize non-isomorphic square Gram mates with all distinct singular values---in particular the multiplicity of $0$ as a singular value is $1$, if Gram mates are singular---via a realizable matrix of rank $1$. As seen in Subsection \ref{Subsec:realizable rank 1}, those Gram mates are convertible.

Let $A$ and $B$ be Gram mates such that $\mathrm{rank}(A-B)=1$. Let $\alpha$ (resp. $\beta$) be the set of row indices (resp. column indices) such that there is a nonzero entry in the corresponding rows (resp. columns) of $A-B$. Then, $A[\alpha^c,\beta^c]=B[\alpha^c,\beta^c]$. The submatrix $A[\alpha^c,\beta^c]$ is said to be the \textit{remaining matrix} of Gram mates $A$ and $B$. Rearranging rows in order of $\alpha$ and $\alpha^c$ and columns in order of $\beta$ and $\beta^c$, we can obtain isomorphic matrices to $A$ and $B$, respectively, as we preserve the structure of the remaining matrix. Without loss of generality, by Theorem \ref{Thm:Gram mates of rank1},
\begin{align}\label{temp:matrices A and B rank 1}
A=\begin{bmatrix}
0 & J_{k_1,k_2} & X_1\\
J_{k_1,k_2} & 0 & X_2\\
X_3 & X_4 & Y
\end{bmatrix},\; B=\begin{bmatrix}
J_{k_1,k_2} & 0 & X_1\\
0 & J_{k_1,k_2} & X_2\\
X_3 & X_4 & Y
\end{bmatrix}
\end{align}
$\mathbf{1}^TX_1=\mathbf{1}^TX_2$ and $X_3\mathbf{1}=X_4\mathbf{1}$ where $k_1,k_2>0$. Note that $Y$ is the remaining matrix of $A$ and $B$.

We now consider sets of particular permutation matrices in order to establish families of non-isomorphic Gram mates (Proposition \ref{Prop:fixable iff row and column sum vectors} and Theorem \ref{Thm:non-isomorphic distinct sv}). Let $Z_i$ be an $m\times n$ $(0,1)$ matrix for $i=1,\dots,4$. Define $\cR_{Z_1,Z_2}$ to be the set of all $3$-tuples $(P_1,P_2,Q)$, where $P_1$, $P_2$ and $Q$ are permutation matrices such that $Z_2=P_1Z_1Q$ and $Z_1=P_2Z_2Q$. We also define $\cL_{Z_3,Z_4}$ as the set of all $3$-tuples $(P,Q_3,Q_4)$, where $P$, $Q_3$ and $Q_4$ are permutation matrices such that $Z_3=PZ_3Q_3$ and $Z_4=PZ_4Q_4$. 

If there exist $(P_1,P_2,Q)\in \cR_{X_1,X_2}$ and $(P,Q_3,Q_4)\in \cL_{X_3,X_4}$ such that $Y=PYQ$, then 
$$
\begin{bmatrix}
0 & P_2 & 0\\
P_1 & 0 & 0\\
0 & 0 & P
\end{bmatrix}A\begin{bmatrix}
Q_3 & 0 & 0\\
0 & Q_4 & 0\\
0 & 0 & Q
\end{bmatrix}=\begin{bmatrix}
J_{k_1,k_2} & 0 & P_2X_2Q\\
0 & J_{k_1,k_2} & P_1X_1Q\\
PX_3Q_3 & PX_4Q_4 & PYQ
\end{bmatrix}=B.
$$
So, $A$ and $B$ are isomorphic. Similarly, if there exist $(P_1,P_2,Q)\in \cR_{X_3^T,X_4^T}$ and $(P,Q_3,Q_4)\in \cL_{X_1^T,X_2^T}$ such that $Y^T=PY^TQ$, then 
$$
\begin{bmatrix}
Q_3^T & 0 & 0\\
0 & Q_4^T & 0\\
0 & 0 & Q^T
\end{bmatrix}A\begin{bmatrix}
0 & P_1^T & 0\\
P_2^T & 0 & 0\\
0 & 0 & P^T
\end{bmatrix}=B.
$$
Hence, $A$ is isomorphic to $B$. The remaining matrix $Y$ of $A$ and $B$ is said to be \textit{fixable} if there exist permutation matrices $P$ and $Q$ such that $Y=PYQ$ and one of two cases holds: 
\begin{enumerate}[label=(\roman*)]
	\item $(P_1,P_2,Q)\in\cR_{X_1,X_2}$ and $(P,Q_3,Q_4)\in\cL_{X_3,X_4}$ for some $P_1,P_2,Q_3,Q_4$; and
	\item $(Q_3,Q_4,P^T)\in\cR_{X_3^T,X_4^T}$ and $(Q^T,P_1,P_2)\in\cL_{X_1^T,X_2^T}$  for some $P_1,P_2,Q_3,Q_4$.
\end{enumerate}

The following proposition immediately follows.

\begin{proposition}\label{Prop:fixable and isomorphic}
	Suppose that $A$ and $B$ are Gram mates and $\mathrm{rank}(A-B)=1$. If the remaining matrix of $A$ and $B$ is fixable, then $A$ is isomorphic to $B$.
\end{proposition}

\begin{remark}
	Let $Z_1$ and $Z_2$ be $(0,1)$ matrices of the same size. We have $(I,I,I)\in\cL_{Z_1,Z_2}$. If $Z_1$ and $Z_2$ are not isomorphic, then $\cR_{Z_1,Z_2}$ and $\cR_{Z_1^T,Z_2^T}$ are empty. For the matrices $A$ and $B$ in \eqref{temp:matrices A and B rank 1}, if $\cR_{X_1,X_2}$ and $\cR_{X_3^T,X_4^T}$ both are empty, then $Y$ is not fixable.
\end{remark}

\begin{example}\label{Example: nonisomorphic rank 1}
	Here we revisit Example \ref{Example:rank1 Gram mates}:
	$$A=\left[\begin{array}{cc|cc|ccc}
	0 & 0 & 1 & 1 & 1 & 1 & 0\\
	0 & 0 & 1 & 1 & 0 & 0 & 1\\\hline
	1 & 1 & 0 & 0 & 1 & 1 & 1\\
	1 & 1 & 0 & 0 & 0 & 0 & 0\\\hline
	1 & 0 & 1 & 0 & 1 & 1 & 1\\
	1 & 0 & 1 & 0 & 1 & 1 & 1\\
	0 & 1 & 1 & 0 & 1 & 1 & 1
	\end{array}\right],\;B=\left[\begin{array}{cc|cc|ccc}
	1 & 1 & 0 & 0 & 1 & 1 & 0\\
	1 & 1 & 0 & 0 & 0 & 0 & 1\\\hline
	0 & 0 & 1 & 1 & 1 & 1 & 1\\
	0 & 0 & 1 & 1 & 0 & 0 & 0\\\hline
	1 & 0 & 1 & 0 & 1 & 1 & 1\\
	1 & 0 & 1 & 0 & 1 & 1 & 1\\
	0 & 1 & 1 & 0 & 1 & 1 & 1
	\end{array}\right]$$
	where $A$ and $B$ are conformally partitioned with the matrices in \eqref{temp:matrices A and B rank 1}, and we use the same notation in \eqref{temp:matrices A and B rank 1}. Since $X_1$ and $X_2$ are not isomorphic, $\cR_{X_1,X_2}$ is empty. Similarly, $\cR_{X_3^T,X_4^T}$ is also empty. So, the remaining matrix is not fixable. 
	
	We claim that $A$ and $B$ are not isomorphic. The multi-sets of row sums and column sums of a matrix are invariant under permutation of rows and columns in the matrix. The row and column sum vectors for $A$ and $B$ are $(4,3,5,2,5,5,5)$. So, if $A$ and $B$ are isomorphic, then necessarily $A[\alpha,\alpha]$ is isomorphic to $B[\alpha,\alpha]$ where $\alpha=\{3,5,6,7\}$. However, $A[\alpha,\alpha]$ contains fewer ones than $B[\alpha,\alpha]$. Hence, $A$ and $B$ are not isomorphic.
\end{example}

We shall consider the converse of Proposition \ref{Prop:fixable and isomorphic} under some circumstances motivated by Example \ref{Example: nonisomorphic rank 1}.

\begin{remark}\label{Remark:permutations for block}
	Let $A=\begin{bmatrix}
	0 & J_{k_1,k_2}\\
	J_{k_1,k_2} & 0
	\end{bmatrix}$ and $B=\begin{bmatrix}
	J_{k_1,k_2} & 0\\
	0 & J_{k_1,k_2}
	\end{bmatrix}$ where $k_1,k_2>0$. One can verify that for any permutation matrices $P$ and $Q$ such that $B=PAQ$, a pair $(P,Q)$ is either 
	\begin{center}
		$\left(\begin{bmatrix}
		P_1 & 0\\
		0 & P_2
		\end{bmatrix},\begin{bmatrix}
		0 & Q_1\\
		Q_1 & 0
		\end{bmatrix}\right)$ or $\left(\begin{bmatrix}
		0 & P_1\\
		P_2 & 0
		\end{bmatrix},\begin{bmatrix}
		Q_1 & 0\\
		0 & Q_2
		\end{bmatrix}\right)$
	\end{center}
	for some permutation matrices $P_1,P_2,Q_1$ and $Q_2$.
\end{remark}

\begin{proposition}\label{Prop:fixable iff row and column sum vectors}
	Let $A$ and $B$ be Gram mates of the form \eqref{temp:matrices A and B rank 1}. Let $R$ and $S$ be the row and column sum vectors of $A$, and conformally partitioned with the rows and the columns of $A$ as $R=(R_1,R_2,R_3)$ and $S=(S_1,S_2,S_3)$, respectively. Suppose that for $i=1,2$, the set of all entries of $R_i$ (resp. $S_i$) does not have any element in common with that of $R_3$ (resp. $S_3$). Then, $A$ is isomorphic to $B$ if and only if the remaining matrix is fixable.
\end{proposition}
\begin{proof}
	Suppose that $A$ and $B$ are isomorphic, say $B=PAQ$ for some permutation matrices $P$ and $Q$. Since the multi-sets of row sums and column sums of a matrix are preserved by permutation, for $i=1,2$, any row of $A$ indexed by an entry in $R_i$ cannot turn into some row indexed by an entry in $R_3$ in order to obtain $B$ by permutation. Similarly, one can find an analogous result with respect to columns by using $S_1$, $S_2$, and $S_3$. Hence, considering Remark \ref{Remark:permutations for block}, a pair $(P,Q)$ must be one of the following:
	\begin{align}\label{permutation P Q mtx form}
	\left(\begin{bmatrix}
	P_1 & 0 & 0\\
	0 & P_2 & 0\\
	0 & 0 & P_3
	\end{bmatrix},\; 
	\begin{bmatrix}
	0 & Q_1 & 0\\
	Q_2 & 0 & 0\\
	0 & 0 & Q_3
	\end{bmatrix}\right),\;\left(\begin{bmatrix}
	0& P_1 & 0\\
	P_2 & 0 & 0\\
	0 & 0 & P_3
	\end{bmatrix},\; 
	\begin{bmatrix}
	Q_1 & 0 & 0\\
	0 & Q_2 & 0\\
	0 & 0 & Q_3
	\end{bmatrix}\right).
	\end{align}
	Let $(P,Q)$ be the former of the two cases. From $B=PAQ$, one can check that $X_3=P_3X_4Q_2$, $X_4=P_3X_3Q_1$, $X_1=P_1X_1Q_3$, $X_2=P_2X_2Q_3$ and $Y=P_3YQ_3$. Hence, $Y$ is fixable. The other case of $(P,Q)$ can be easily checked. Hence, the remaining matrix is fixable.
	
	The converse of the proof follows from Proposition \ref{Prop:fixable and isomorphic}.
\end{proof}

\begin{example}\label{Example:same entries in sum vectors}
	Let $E=\begin{bmatrix}
	J_3 & -J_3 & 0\\
	-J_3 & J_3 & 0\\
	0 & 0 & 0
	\end{bmatrix}$. Consider $$A=\begin{bmatrix}
	0 & J_3 & X_1\\
	J_3 & 0 & X_2\\
	X_3 & X_4 & Y
	\end{bmatrix}=\left[\begin{array}{ccc|ccc|cccc}
	0 & 0 & 0 & 1 & 1 & 1 & 1 & 1 & 0 & 0\\
	0 & 0 & 0 & 1 & 1 & 1 & 0 & 1 & 1 & 0\\
	0 & 0 & 0 & 1 & 1 & 1 & 0 & 0 & 1 & 1\\\hline
	1 & 1 & 1 & 0 & 0 & 0 & 0 & 1 & 0 & 1\\
	1 & 1 & 1 & 0 & 0 & 0 & 0 & 1 & 1 & 0\\
	1 & 1 & 1 & 0 & 0 & 0 & 1 & 0 & 1 & 0\\\hline
	1 & 0 & 1 & 1 & 1 & 0 & 0 & 1 & 1 & 1\\
	1 & 1 & 1 & 1 & 1 & 1 & 1 & 1 & 1 & 0\\
	1 & 1 & 0 & 0 & 1 & 1 & 1 & 1 & 1 & 1\\
	0 & 1 & 1 & 1 & 0 & 1 & 1 & 1 & 1 & 1
	\end{array}\right].$$
	Since $\mathbf{1}^TX_1=\mathbf{1}^TX_2$ and $X_3\mathbf{1}=X_4\mathbf{1}$, $A$ and $A+E$ are Gram mates. Moreover, $(P_1,P_2,Q)\in\cR_{X_1,X_2}$ where $P_1$, $P_2$ and $Q$ correspond to permutations $(1,3)$, $(1,3)$ and $(2,3)$ in cycle notation, respectively. Note that $(I,I,I)\in\cL_{X_3,X_4}$. Since $Y$ is invariant under the column permutation corresponding to $(2,3)$, the remaining matrix of $A$ is fixable. The row and column sum vectors of $A$ are $(5\mathbf{1}_6^T,7,9,8,8)$ and $(6\mathbf{1}_6^T,5,8,8,5)$, respectively. By Proposition \ref{Prop:fixable iff row and column sum vectors}, $A$ and $B$ are isomorphic.
\end{example}

Let square matrices $A$ and $B$ be Gram mates, where $\mathrm{rank}(A-B)=1$. Suppose that all singular values of $A$ are distinct. We claim that if the remaining matrix of $A$ is not fixable, then $A$ and $B$ are not isomorphic. To establish the claim, we consider that for the adjacency matrix $X$ of a (multi-)graph, the \textit{automorphism group} of $X$, denoted $\Gamma(X)$, is defined as the set of all permutation matrices $P$ such that $PXP^T=X$. We refer the interested reader to \cite{Dragos:SpectraOfGraphs} for properties of $\Gamma(X)$ regarding eigenvalues of $X$, and to \cite{Godsil:AlgebraicGraph} for an introduction to the automorphism group of a connected simple graph regarding group action.

Let $A$ and $B$ (not necessarily square) be Gram mates. Suppose that $A$ and $B$ are isomorphic. Then, there exist permutation matrices $P$ and $Q$ such that $B=PAQ$. Since $A$ and $B$ are Gram mates, we have $PAA^TP^T=AA^T$ and $Q^TA^TAQ=A^TA$. Therefore, $P\in\Gamma(AA^T)$ and $Q\in\Gamma(A^TA)$. However, the converse does not hold.

\begin{example}
	Non-isomorphic symmetric balanced incomplete block designs $A$ and $B$ can be found in \cite{Wallis:DesignTheory}. Since $AA^T=A^TA=aI+bJ$ for some $a,b>0$, $\Gamma(AA^T)$ and $\Gamma(A^TA)$ are isomorphic to the symmetric group.
\end{example}

\begin{theorem}\cite{Dragos:SpectraOfGraphs,Mow:AutoAndDistinct}
	Let $X$ be the adjacency matrix of a multigraph. Suppose that $X$ has all distinct eigenvalues. Then, for any $P\in\Gamma(X)$, $P^2=I$. Furthermore, this implies $\Gamma(X)$ is abelian.
\end{theorem}

\begin{corollary}\label{Cor:automorphism and distinct sv}
	Let $n\times n$ matrices $A$ and $B$ be Gram mates with all distinct singular values. If $A$ and $B$ are isomorphic, then $B$ is obtained from $A$ by permuting rows and columns according to permutations, any cycle in which is of length at most $2$.
\end{corollary}

\begin{theorem}\label{Thm:non-isomorphic distinct sv}
	Let $n\times n$ $(0,1)$ matrices $A$ and $B$ be Gram mates of form \eqref{temp:matrices A and B rank 1} with all distinct singular values where $\mathrm{rank}(A-B)=1$. Then, $A$ and $B$ are isomorphic if and only if the remaining matrix is fixable.
\end{theorem}
\begin{proof}
	Assume for contradiction that $A$ and $B$ are isomorphic and the remaining matrix is not fixable, say $B=P_0AQ_0$ where $P_0$ and $Q_0$ are permutation matrices. By Corollary \ref{Cor:automorphism and distinct sv}, $P_0^2=Q_0^2=I$. Let $\sigma_0$ and $\tau_0$ be permutations corresponding to $P_0$ and $Q_0$, respectively. Adopting the notation in \eqref{temp:matrices A and B rank 1} for $A$ and $B$, let $\alpha_1=\{1,\dots,k_1\}$, $\alpha_2=\{k_1+1,\dots,2k_1\}$, $\alpha_3=\{2k_1+1,\dots,n\}$, $\beta_1=\{1,\dots,k_2\}$, $\beta_2=\{k_2+1,\dots,2k_2\}$ and $\beta_3=\{2k_2+1,\dots,n\}$. Let $\wtA$ be the resulting matrix after applying the permutation $\sigma_0$ to rows of $A$, and let $B$ be the resulting matrix after applying $\tau_0$ to columns of $\wtA$. Consider 
	\begin{align*}
	A=\begingroup % keep the change local
	\setlength\arraycolsep{2pt}\begin{bmatrix}
	0 & J_{k_1,k_2} & X_1\\
	J_{k_1,k_2} & 0 & X_2\\
	X_3 & X_4 & Y
	\end{bmatrix}\overset{\sigma_0}{\longrightarrow}\widetilde{A}=\begin{bmatrix}
	\widetilde{A}_{11} & \widetilde{A}_{12} & \widetilde{X}_1\\
	\widetilde{A}_{21} & \widetilde{A}_{22} & \widetilde{X}_2\\
	\widetilde{X}_3 & \widetilde{X}_4 & \widetilde{Y}
	\end{bmatrix}\overset{\tau_0}{\longrightarrow}B=\begin{bmatrix}
	J_{k_1,k_2} & 0 & X_1\\
	0 & J_{k_1,k_2} & X_2\\
	X_3 & X_4 & Y
	\end{bmatrix}.
	\endgroup
	\end{align*}
	
	Suppose to the contrary that $\sigma_0(a)\in\alpha_3$ for all $a\in\alpha_3$. We consider three cases: \begin{enumerate*}[label=(\alph*)]
		\item \label{temp:condition a}$\sigma_0(\alpha_1)\neq\alpha_1$ and $\sigma_0(\alpha_1)\neq\alpha_2$, \item \label{temp:condition b}$\sigma_0(\alpha_1)=\alpha_1$, \item\label{temp:condtion c} $\sigma_0(\alpha_1)=\alpha_2$.
	\end{enumerate*}
	Let the condition \ref{temp:condition a} hold. Then, $|\{x\in\alpha_1|\sigma_0(x)\in\alpha_1\}|>0$. So, neither $\begin{bmatrix}
	\widetilde{A}_{11}\\
	\widetilde{A}_{21}
	\end{bmatrix}$ nor $\begin{bmatrix}
	\widetilde{A}_{12}\\
	\widetilde{A}_{22}
	\end{bmatrix}$ contains the columns $\begin{bmatrix}
	\mathbf{1}_{k_1}\\
	\mathbf{0}_{k_1}
	\end{bmatrix}$ or $\begin{bmatrix}
	\mathbf{0}_{k_1}\\
	\mathbf{1}_{k_1}
	\end{bmatrix}$. Note that any cycle of $\tau_0$ is of length either $1$ or $2$. In order to obtain $B$ from $\widetilde{A}$, the first column $\tilde{\bx}_1$ of $\widetilde{A}$ must be swapped with some $j^\text{th}$ column $\tilde{\bx}_j$ of $\widetilde{A}$ for some $j\in \beta_3$. Then, the subvector $\tilde{\bx}_j[\alpha_1\cup\alpha_2]$ must be $\begin{bmatrix}\mathbf{1}_{k_1}\\
	\mathbf{0}_{k_1}
	\end{bmatrix}$. So, for the $j^{\text{th}}$ column $\bx_j$ of $A$, we have  $\bx_j[\alpha_1\cup\alpha_2]\overset{\sigma_0}{\longrightarrow}\begin{bmatrix}
	\mathbf{1}_{k_1}\\
	\mathbf{0}_{k_1}
	\end{bmatrix}$. Further, we can readily find $\begin{bmatrix}
	\mathbf{0}_{k_1}\\
	\mathbf{1}_{k_1}
	\end{bmatrix}\overset{\sigma_0}{\longrightarrow}\tilde{\bx}_1[\alpha_1\cup\alpha_2]$. Since $j\in \beta_3$, the $j^\text{th}$ columns of $A$ and $B$ must coincide, \textit{i.e.}, $\bx_j[\alpha_1\cup\alpha_2]=\tilde{\bx}_1[\alpha_1\cup\alpha_2]$. Hence, we obtain
	$$\begin{bmatrix}
	\mathbf{0}_{k_1}\\
	\mathbf{1}_{k_1}
	\end{bmatrix}\overset{\sigma_0}{\longrightarrow}\tilde{\bx}_1[\alpha_1\cup\alpha_2]=\bx_j[\alpha_1\cup\alpha_2]\overset{\sigma_0}{\longrightarrow}\begin{bmatrix}
	\mathbf{1}_{k_1}\\
	\mathbf{0}_{k_1}
	\end{bmatrix}.$$
	However, the mapping contradicts the fact that $\sigma_0^2$ is the identity. Therefore, the case \ref{temp:condition a} does not hold.
	
	Consider the case \ref{temp:condition b} that $\sigma_0(\alpha_1)=\alpha_1$. Since $\mathbf{1}^TX_1=\mathbf{1}^TX_2$, there are no columns in $\widetilde{A}[\alpha_1\cup\alpha_2,\beta_3]$ that are $\begin{bmatrix}
	\mathbf{1}_{k_1}\\
	\mathbf{0}_{k_1}
	\end{bmatrix}$ or $\begin{bmatrix}
	\mathbf{0}_{k_1}\\
	\mathbf{1}_{k_1}
	\end{bmatrix}$. This implies that $\tau_0(\beta_1)=\beta_2$, $\tau_0(\beta_2)=\beta_1$ and $\tau_0(\beta_3)=\beta_3$. Then, a pair $(P,Q)$ corresponds to the former in \eqref{permutation P Q mtx form}. By the argument of the proof in \ref{Prop:fixable iff row and column sum vectors}, $Y$ is fixable, which is a contradiction to the hypothesis. Finally, suppose that $\sigma_0(\alpha_1)=\alpha_2$. Since $\mathbf{1}^TX_1=\mathbf{1}^TX_2$, we have $\tau_0(\beta_1)=\beta_1$, $\tau_0(\beta_2)=\beta_2$ and $\tau_0(\beta_3)=\beta_3$. Using an analogous argument as for case \ref{temp:condition b}, we have a contradiction for the case \ref{temp:condtion c}. Therefore, there exists $a\in\alpha_3$ such that $\sigma_0(a)\notin\alpha_3$.
	
	We now suppose that there exists $j\in\alpha_3$ such that $\sigma_0(j)\notin\alpha_3$. Let $i=\sigma_0(j)$. Then, $i\in \alpha_1\cup\alpha_2$, say $i\in\alpha_1$. Let $\ba_i^T$ and $\ba_j^T$ be the $i^\text{th}$ and $j^\text{th}$ rows of $A$, respectively, where $\ba_i^T=\begin{bmatrix}
	\mathbf{0}_{k_1}^T & \mathbf{1}_{k_1}^T & \bx_1^T
	\end{bmatrix}$, $\ba_j^T=\begin{bmatrix}
	\bx_3^T & \bx_4^T & \by^T
	\end{bmatrix}$ and $\ba_j^T$ is compatible with the partition of $\ba_i^T$. Since each cycle of $\sigma_0$ is of length either $1$ or $2$, $\ba_j^T$ and $\ba_i^T$ are mapped to the $i^\text{th}$ row $\tilde{\ba}_i^T$ and the $j^\text{th}$ row $\tilde{\ba}_j^T$ of $\widetilde{A}$, respectively, after applying the row permutation $\sigma_0$. After permuting columns of $\tilde{\ba}_i^T$ and $\tilde{\ba}_j^T$ according to $\tau_0$, we must have $\bb_i^T$ and $\bb_j^T$ that are the $i^\text{th}$ and $j^\text{th}$ rows of $B$:
	$$
	\begin{bmatrix}
	\tilde{\ba}_i^T\\
	\tilde{\ba}_j^T
	\end{bmatrix}=\begin{bmatrix}
	\bx_3^T &  \bx_4^T & \by^T\\
	\mathbf{0}^T & \mathbf{1}^T & \bx_1^T
	\end{bmatrix}\overset{\tau_0}{\longrightarrow}\begin{bmatrix}
	\bb_i^T\\
	\bb_j^T
	\end{bmatrix}=\begin{bmatrix}
	\mathbf{1}^T & \mathbf{0}^T & \bx_1^T\\
	\bx_3^T &  \bx_4^T & \by^T
	\end{bmatrix}.
	$$
	Since $A$ and $B$ are Gram mates, the row sums of $\bx_3^T$ and $\bx_4^T$ are the same, say $\ell$. Permuting columns of $\begin{bmatrix}
	\tilde{\ba}_i^T\\
	\tilde{\ba}_j^T
	\end{bmatrix}$ and $\begin{bmatrix}
	\bb_i^T\\
	\bb_j^T
	\end{bmatrix}$ simultaneously, we can obtain
	$$
	Z_1=\begin{bmatrix}
	\mathbf{1}^T_\ell & \mathbf{0}^T_{k-\ell} & \mathbf{1}^T_\ell & \mathbf{0}^T_{k-\ell} & \by^T\\
	\mathbf{0}^T_\ell & \mathbf{0}^T_{k-\ell} & \mathbf{1}^T_\ell & \mathbf{1}^T_{k-\ell} & \bx_1^T
	\end{bmatrix}\overset{\tau_0}{\longrightarrow}Z_2=\begin{bmatrix}
	\mathbf{1}^T_\ell & \mathbf{1}^T_{k-\ell} & \mathbf{0}^T_\ell & \mathbf{0}^T_{k-\ell} & \bx_1^T\\
	\mathbf{1}^T_\ell & \mathbf{0}^T_{k-\ell} & \mathbf{1}^T_\ell & \mathbf{0}^T_{k-\ell} & \by^T
	\end{bmatrix}.
	$$
	Suppose that $\ell>0$. Note that if the $k\text{th}$ column of $Z_1$ for $k\in\beta_3$ is $\begin{bmatrix}
	1\\ 1
	\end{bmatrix}$, so is the $k\text{th}$ column of $Z_2$. Considering that each cycle in $\tau_0$ is of length either $1$ or $2$, the first column $\begin{bmatrix}
	1\\
	0
	\end{bmatrix}$ of $Z_1$ must be swapped with some $k_0^\text{th}$ column $\begin{bmatrix}
	1\\
	1
	\end{bmatrix}$ of $Z_2$ where $k_0\in\beta_3$. Then, $\begin{bmatrix}
	1\\
	0
	\end{bmatrix}$ is the $k_0^\text{th}$ column of $Z_2$, which is a contradiction. Similarly, applying an analogous argument of the case $\ell>0$ to the case $\ell=0$, one can obtain a contradiction. Therefore, our desired result is established.
	
	For the proof of the converse, apply Proposition \ref{Prop:fixable and isomorphic}.
\end{proof}

%Here are two sample references: \cite{Feynman1963118,Dirac1953888}.

{\em Acknowledgement}: The authors are grateful to Geir Dahl at the University of Oslo for interesting discussions about some topics in this paper.

%\bibliography{mybibfile}

\end{document}